\documentclass[dvipsnames]{amsart}
\pdfoutput=1

\usepackage[left=1.15in,right=1.15in,top=1in]{geometry}

\makeatletter
\renewcommand{\email}[2][]{%
  \ifx\emails\@empty\relax\else{\g@addto@macro\emails{,\space}}\fi%
  \@ifnotempty{#1}{\g@addto@macro\emails{\textrm{(#1)}\space}}%
  \g@addto@macro\emails{#2}%
}
\makeatother

\usepackage{amsfonts, amssymb,bbm }
\usepackage{graphicx}
\usepackage{mathrsfs,mathtools}
\usepackage{epstopdf}
\usepackage{algorithm}
\usepackage{algorithmic}
\usepackage{subcaption}

\usepackage{cleveref}
\usepackage{enumitem}
\usepackage{url}

\usepackage[foot]{amsaddr}

\setlist[enumerate,1]{label=(\arabic*),ref=(\arabic*)}

\ifpdf
  \DeclareGraphicsExtensions{.eps,.pdf,.png,.jpg}
\else
  \DeclareGraphicsExtensions{.eps}
\fi

\newcommand{\creflastconjunction}{, and~}

\newcommand{\mathd}{\mathrm{d}}
\newcommand{\R}{\mathbbm{R}}
\DeclareMathOperator*{\argmin}{argmin}

\newtheorem{theorem}{Theorem}[section]
\newtheorem{proposition}{Proposition}[section]

\crefname{hypothesis}{Hypothesis}{Hypotheses}

\title[Structure-preserving function approximation]{Structure-preserving function approximation via convex optimization}

\author{Vidhi Zala}
\author{Robert M. Kirby}
\author{Akil Narayan}

\address[V.~Zala, R.~M.~Kirby]{Scientific Computing and Imaging Institute and School of Computing, University of Utah, Salt Lake City, UT 84112}
\address[A.~Narayan]{Scientific Computing and Imaging Institute and Department of Mathematics, University of Utah, Salt Lake City, UT 84112}
\email[V.~Zala]{vidhi.zala@utah.edu}
\email[A2]{kirby@cs.utah.edu}
\email[A3]{akil@sci.utah.edu}

\usepackage{amsopn}

\newcommand{\ddx}[1]{\frac{\mathd}{\mathd #1}}

\newcommand{\dx}[1]{\mathd #1}
\newcommand{\N}{\mathbbm{N}}

\newcommand{\bs}[1]{\boldsymbol{#1}}
\renewcommand{\hat}[1]{\widehat{#1}}

\theoremstyle{example}
\newtheorem{example}{Example}[section]

\usepackage{array,multirow,booktabs} 
\usepackage[export]{adjustbox} 

\newcommand{\an}[1]{#1}
\DeclareMathOperator{\sign}{sign}

\begin{document}

\begin{abstract}

Approximations of functions with finite data often do not respect certain ``structural'' properties of the functions. For example, if a given function is non-negative, a polynomial approximation of the function is not necessarily also non-negative. We propose a formalism and algorithms for preserving certain types of such structure in function approximation. In particular, we consider structure corresponding to a convex constraint on the approximant (for which positivity is one example). The approximation problem then converts into a convex feasibility problem, but the feasible set is relatively complicated so that standard convex feasibility algorithms cannot be directly applied. We propose and discuss different algorithms for solving this problem. One of the features of our machinery is flexibility: relatively complicated constraints, such as simultaneously enforcing positivity, monotonicity, and convexity, are fairly straightforward to implement. We demonstrate the success of our algorithm on several problems in univariate function approximation.
\end{abstract}
\maketitle

\section{Introduction}

The approximation of functions as a linear combination of basis functions is a foundational technique in numerical analysis and scientific computing. For example, such a linear combination or expansion is often used as an emulator for the original function, or as an ansatz for the solution to a differential equation. If, e.g., the original function is smooth, then such approximations are often accurate, but they may not adhere to other kinds of \textit{structure} that the function possesses. The simplest example of such structure is positivity: if $f$ is a positive function, an accurate polynomial approximation of $f$ need not also be positive. Other types of structure that arise in practice are monotonicity or maximum and minimum value constraints. If an approximation violates the implicit structure of a function, the resulting computation may produce unphysical predictions, and may cause solvability issues in numerical schemes for solving differential equations \cite{zhang_maximum-principle-satisfying_2011}.

In this paper, we present a general framework for preserving structure in function approximation from a linear subspace. ``Structure" in our context refers to fairly general types of linear inequality constraints, including positivity and monotonicity. However, we demonstrate that our setup can also handle more exotic types of constraints. The model by which we impose structure is straightforward: construct the approximation that best fits the available data, subject to the structural constraints. We observe that imposing our type of structure on the approximation corresponds to a convex constraint on the vector of expansion coefficients (i.e., the coordinates of the approximation in a basis of the linear space). Thus, our notion of structure-preserving approximation corresponds to a convex optimization problem. Unfortunately, the resulting convex set is ``complicated", and we cannot utilize standard algorithms to solve this problem. We therefore develop two algorithms to solve this problem, each of which is advantageous in different situations. We subsequently formulate a hybrid algorithm that achieves superior performance compared to the original two algorithms. In summary, the contributions of this paper are as follows:
\begin{itemize}
  \item We formalize a new model for computing structure-preserving approximations of functions. This model can successfully compute function approximations that respect canonical structure such as positivity and/or monotonicity, but can also embed much richer, nontrivial structure, cf. Figure \ref{fig:Resexpt2}. A particular advantage of our approach is that the formalism is identical for all these types of structure; e.g., the procedure for preserving positivity versus monotonicity is fundamentally the same.
  \item We show that this model corresponds to a finite-dimensional convex \an{semi-infinite} optimization problem. We subsequently characterize the feasible set as an intersection of conic sets (Theorem \ref{thm:C-convex}), and show that the optimization problem, and hence our structure-preserving approximation model, has a unique solution. See Theorem \ref{thm:solution}.
  \item Our convex optimization problem can be cast as a problem of projecting onto a convex set (the feasible set). Unfortunately the feasible set is not, in general, a polytopic region in coefficient space. Hence, a finite number of linear inequality constraints cannot characterize the feasible set. We instead characterize the convex feasible set as one with an (uncountably) infinite number of supporting hyperplanes. We use this characterization to develop two types of algorithms for computing the solution to the optimization problem. We also combine these two algorithms into a hybrid approach that is more efficient than either algorithm alone. These three approaches are detailed in Section \ref{sec:algorithms}. 
  \item We demonstrate with numerical results in one dimension with polynomial approximations that the resulting algorithm produces approximations satisfying desired constraints. We also show that, for our examples, rates of convergence of polynomial approximation are unchanged compared to the unconstrained case.
\end{itemize}
\an{Our problem formulation (along with its mathematical properties) holds in the multivariate approximation case; the major drawback in such cases is that our algorithms require global optimization of multivariate functions, which is a difficult problem in general.}
In order to compute solutions to the constrained optimization problem, our algorithms iteratively ``correct" an unconstrained initial guess. For one of our algorithms, these corrections are essentially Dirichlet kernels for the approximation space. We visualize some correction functions for enforcing positivity in polynomial approximation in Figure \ref{fig:corrections}.

In Section \ref{sec:setup}, we introduce notation, describe the types of constraints we consider, and present the structure-preserving approximation model. Section \ref{sec:method} analyzes the feasible set of the model and shows that a unique solution exists. Section \ref{sec:algorithms} presents our proposed algorithms for computing solutions. Finally, Section \ref{sec:results} contains numerical results and demonstrations.

\begin{figure}[htbp]
  \begin{center}
    \resizebox{0.75\textwidth}{!}{
      \adjincludegraphics[width=0.33\textwidth,trim={0 0 {.5\width} 0},clip]{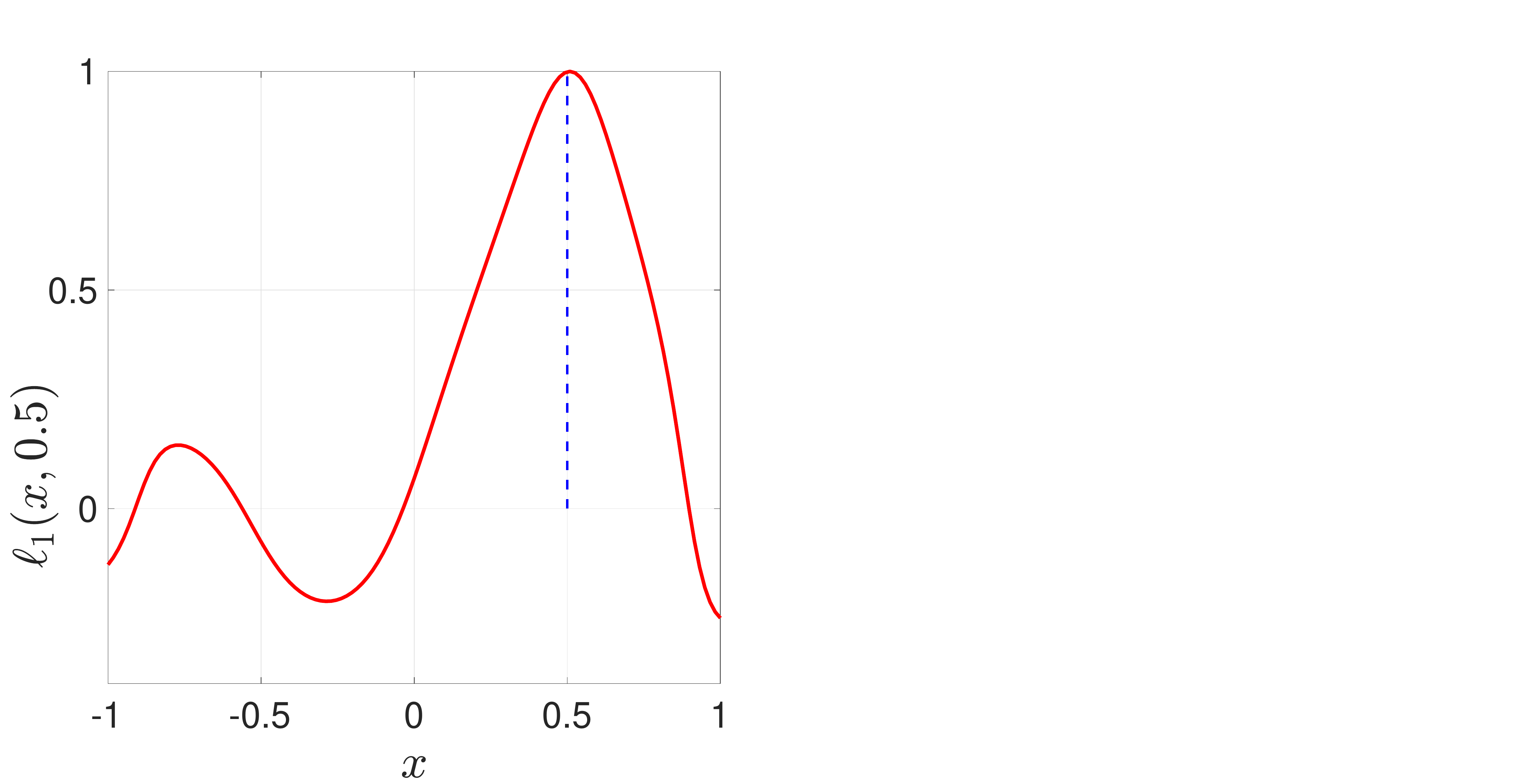}
      \adjincludegraphics[width=0.33\textwidth,trim={0 0 {.5\width} 0},clip]{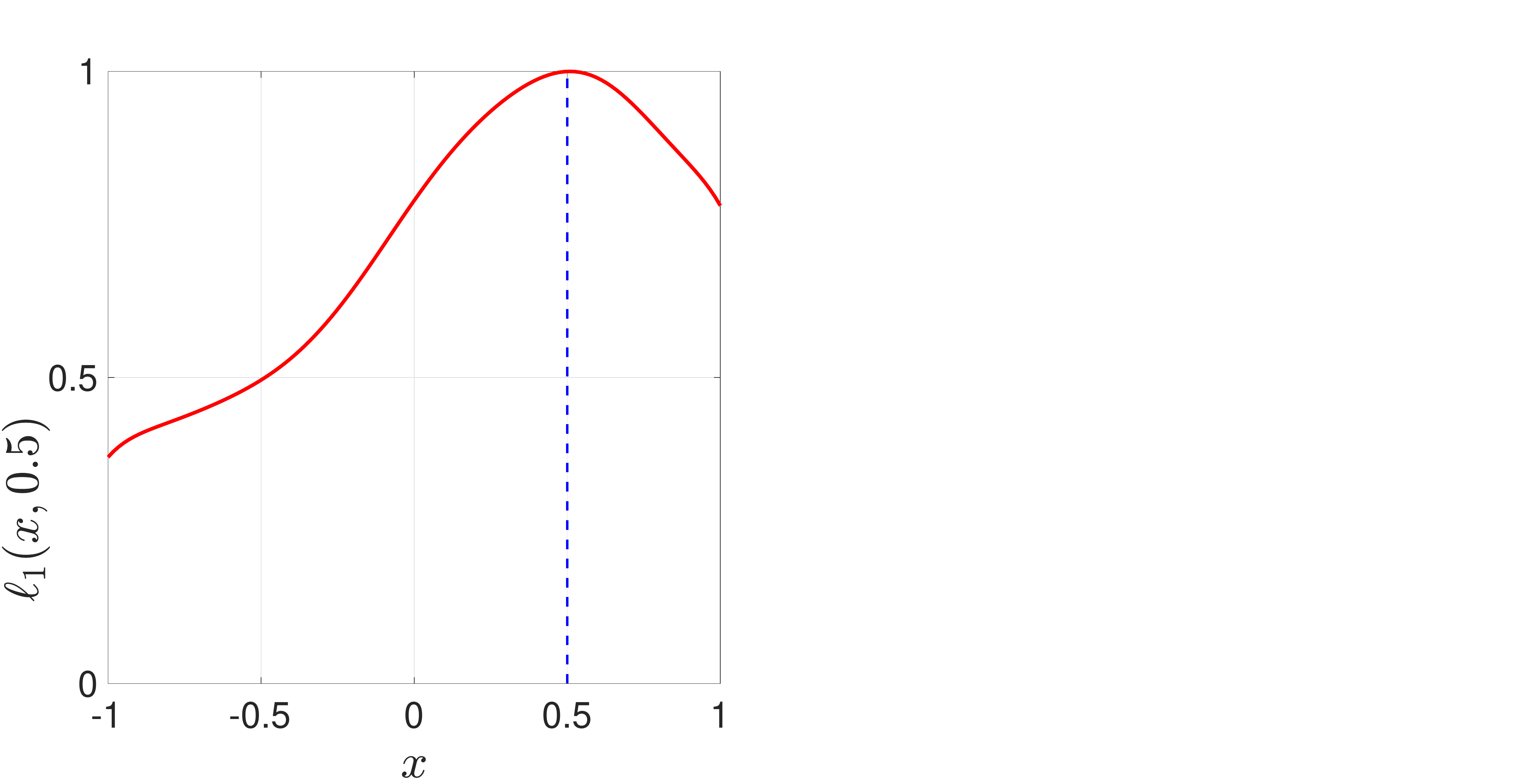}
      \adjincludegraphics[width=0.33\textwidth,trim={0 0 {.5\width} 0},clip]{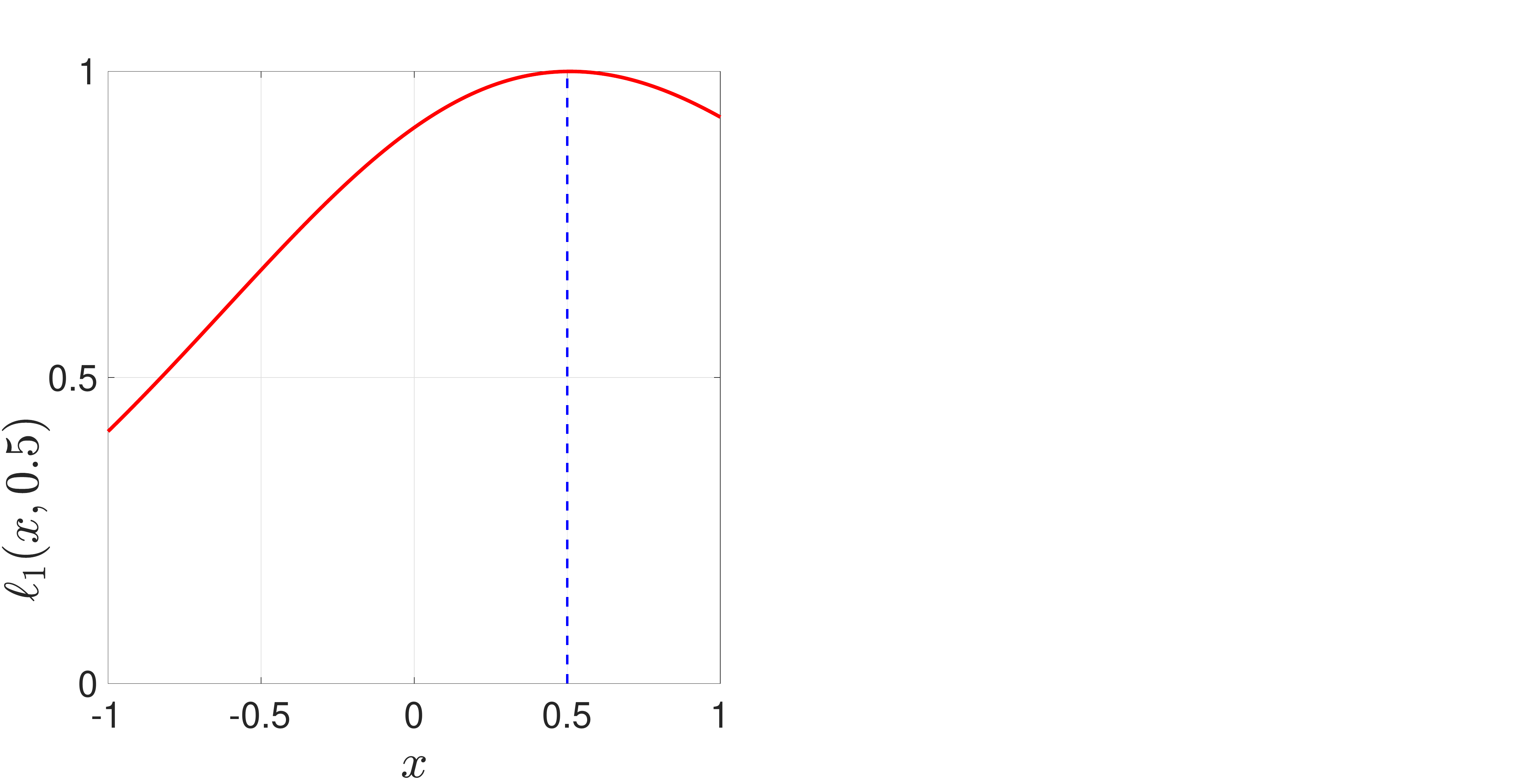}
    }\\
    \resizebox{0.75\textwidth}{!}{
      \adjincludegraphics[width=0.33\textwidth,trim={0 0 {.5\width} 0},clip]{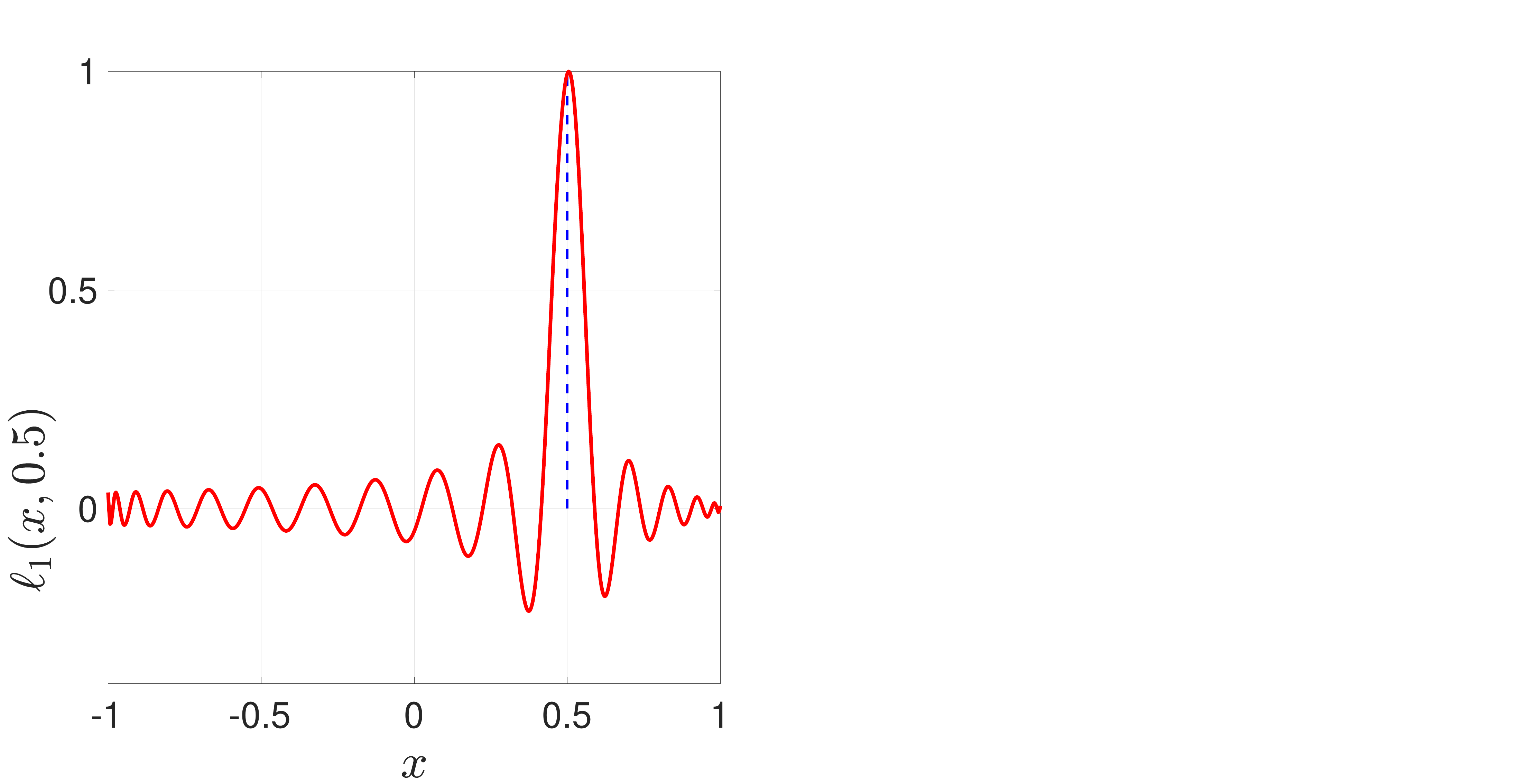}
      \adjincludegraphics[width=0.33\textwidth,trim={0 0 {.5\width} 0},clip]{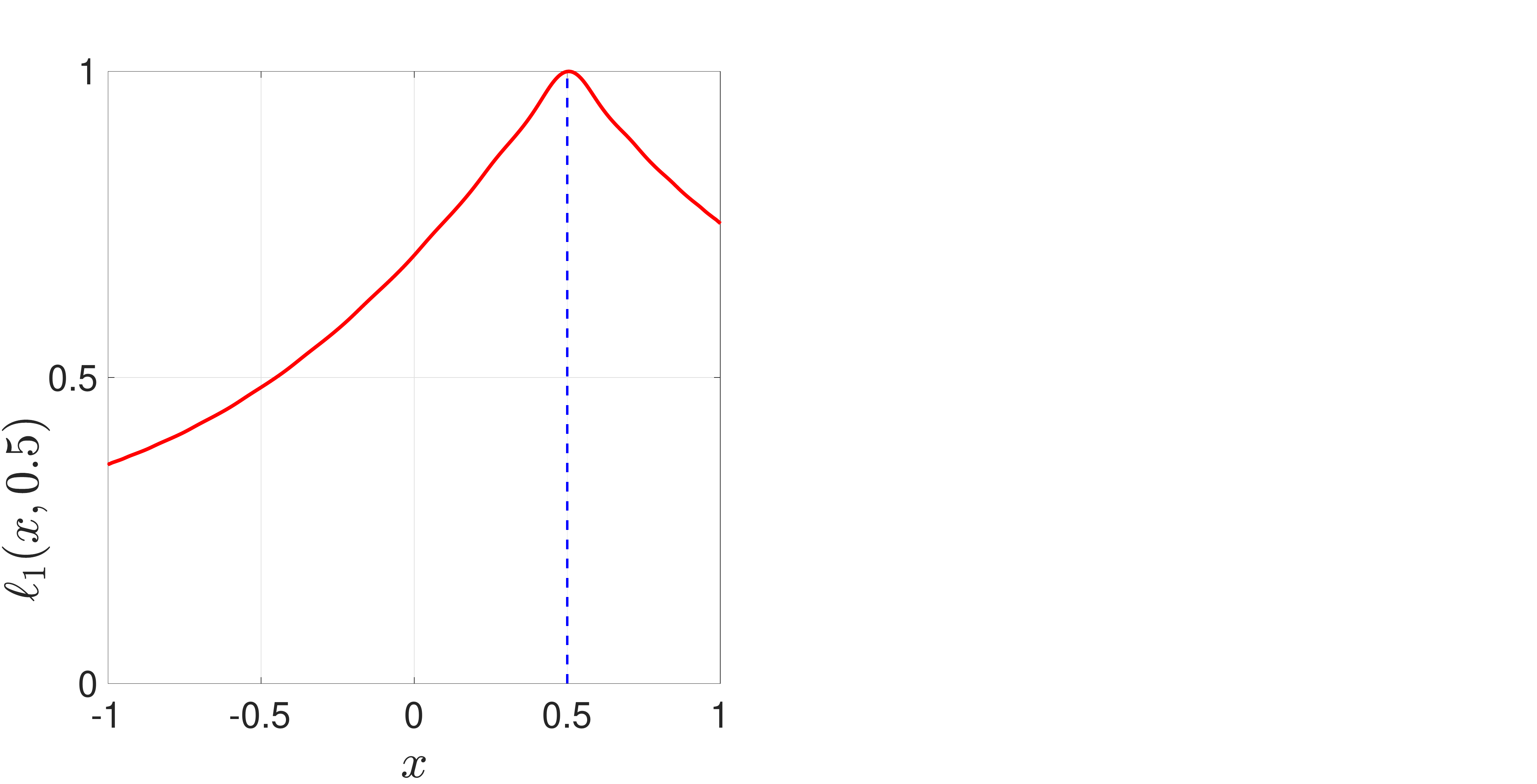}
      \adjincludegraphics[width=0.33\textwidth,trim={0 0 {.5\width} 0},clip]{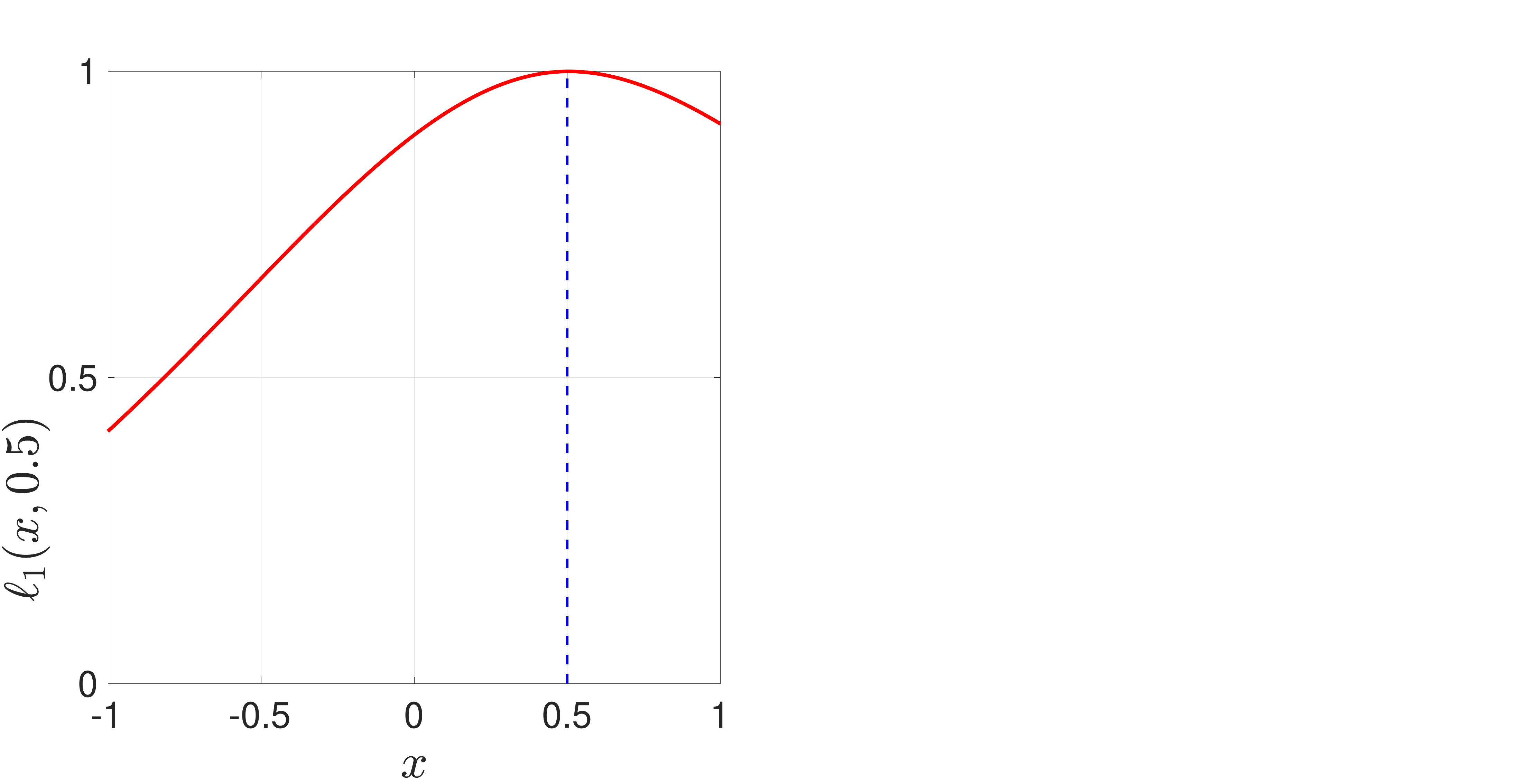}
    }
  \end{center}
  \caption{Correction functions used to enforce positivity in a univariate polynomial approximation. Our algorithm adds scaled/combined versions of these functions to enforce constraints. Shown are corrections targeted to enforce positivity at $x = 0.5$. Correction functions are shown for polynomials of degree 5 (top) and 30 (bottom). The columns correspond to corrections in different ambient Hilbert spaces. Left: $L^2([-1,1])$. Center: $H^1([-1,1])$. Right: $H^2([-1,1])$.}\label{fig:corrections}
\end{figure}

\subsection{Existing and alternative approaches}\label{ssec:lit}

There are several existing techniques for building special kinds of structure-preserving approximations. We will frequently use positivity as an explicit example below to make notions clear.

One simple technique in enforcing positivity in function approximation is to enforce positivity as a finite number of points in the domain. This technique makes the feasible set much easier to characterize and results in applicability of several off-the-shelf algorithms \cite{boyd_convex_2004}. However, these approaches do not guarantee positivity on the entire domain, which our structure-preserving model does enforce. Another class of techniques uses mapping methods. For example, if we approximate $\sqrt{f}$ and square the resulting approximation, then the squared approximation is guaranteed to be positive. There are also more complicated but successful approaches to construct positive approximations \cite{campos-pinto_algorithms_2019}. Although these approaches are attractive, such mapping functions are not easy to construct for more complicated constraints.

Another approach is to adapt the basis; for example, by expanding a function in Bernstein polynomials that are positive on some domain, we can ensure the positivity of the approximation on that domain if all the expansion coefficients are non-negative. Therefore, one forms an approximation subject to the positivity of the coefficients. However, this approach does not yield polynomial reproduction even in simple cases. Consider the following basis for quadratic polynomials in one dimension: $v_1(x) = 1-x^2$, $v_2(x) = (1-x)(x+3)$ and $v_3(x) = (x+1)(3-x)$. Note that on $[-1,1]$, these three functions are all non-negative. However, the (unique) representation of $f\equiv 1$ (that is also non-negative) in this basis is 
\begin{align*}
  f(x) = -\frac{1}{2} v_1(x) + \frac{1}{4} v_2(x) + \frac{1}{4} v_3(x),
\end{align*}
which clearly does not have positive expansion coefficients. Alternative approaches use an adaptive construction scheme for certain kinds of constraints \cite{berzins_adaptive_2007}; our framework allows much more general constraints and is not restricted by dimension, although in this paper we consider only univariate examples.

In general, each of the techniques above is different, and they must usually be nontrivially adapted when a new kind of structure is desired or if a different approximation space is used. The model we employ in this work is general-purpose, handling rather general types of constraints and very general approximation spaces. Finally, we note that there is prior theoretical investigation of error estimates for structure-preserving approximations \cite{devore_degree_1974,beatson_degree_1978,beatson_restricted_1982,nochetto_positivity_2002}.

\an{
The formulation we consider in this paper constructs an optimization problem of the form
\begin{align}\label{eq:sip}
  \min_{\widehat{\bs{v}} \in \R^N} \left\| \bs{A} \widehat{\bs{v}} - \bs{b} \right\|_2^2, \hskip 10pt \textrm{such that} \hskip 10pt g(\widehat{v}, y) \leq 0 \;\forall y \in \Omega,
\end{align}
where $\bs{A}$ and $\bs{b}$ are a given matrix and vector (of appropriate sizes), and $g(\cdot, y)$ is a scalar-valued function depending on a parameter $y$ that takes values in an infinite set $\Omega$. Hence, our problem is a \emph{semi-infinite programming} (SIP) problem \cite{hettich_semi-infinite_1993} since the feasible set is described by an infinite number of constraints. As is well-known in SIP methods, even assessing feasibility of a candidate $\widehat{\bs{v}}$ would require certifying satisfaction of the constraints, i.e., certifying that the maximum of $g(\widehat{\bs{v}},\cdot)$ over all $\Omega$ is non-positive. Globally solving this so-called lower-level problem is typically the main challenge in SIP algorithms, and is frequently circumvented by means of either discretization approaches (that replace $\Omega$ by a finite set) or by local reduction approaches (which partition $\Omega$ into subdomains and use specialized approaches on each subdomain). In both cases, there is a discrete approximation of $\Omega$ that is constructed (and perhaps refined). For generating positive approximations, this would correspond to requiring positivity at only a finite set of points on the domain. 
}

\an{
Our formulation, upon discretization/division of $\Omega$, can certainly leverage SIP algorithms. However, our aim in this paper is to discuss the solution of this problem \emph{without} discretization of $\Omega$, and hence we do not rely on existing SIP algorithms. In particular, we propose algorithms to solve the original SIP problem that presume the ability to compute global solutions to the SIP lower-level problem. Thus, our algorithms differ from many existing SIP algorithms \cite{goberna_semi-infinite_2001,stein_how_2012}, but also inherit the general challenge that global solutions to lower-level SIP problems must be provided.
}

\section{Setup}\label{sec:setup}

Let $\Omega \subset \R^d$ be a spatial domain. Whereas our setup and theoretical results are valid for general $\Omega$ and $d \geq 1$, our numerical examples in this paper will be restricted to $d = 1$ with $\Omega = [-1,1]$. The restriction affects only algorithms and not the model or mathematical properties of our discussion. Consider a Hilbert space formed from scalar-valued functions over $\Omega$:
\begin{align*}
  H &= H(\Omega) \coloneqq \left\{ f: \Omega \rightarrow \R \; \big| \; \|f\| < \infty \right\}, & \| f\|^2 &\coloneqq \left\langle f, f \right\rangle,
\end{align*}
with $\langle \cdot, \cdot \rangle = \langle \cdot, \cdot \rangle_H$ the inner product on $H$. We are mostly concerned with ``standard" function spaces such as $L^2(\Omega)$ or Sobolev spaces\footnote{We formally define $L^2$ and some Sobolev spaces in Section \ref{sec:results}.}.
Let $V$ be an $N$-dimensional subspace of $H$, with $\{v_n\}_{n=1}^N$ a collection of orthonormal basis functions,
\begin{align*}
  V &= \mathrm{span}\left\{ v_1, \ldots, v_N \right\}, & \left\langle v_j, v_k \right\rangle = \delta_{jk},
\end{align*}
for $j, k = 1, \ldots, N$ and with $\delta_{jk}$ the Kronecker delta function. For example, if $H$ is $L^2(\Omega)$ with $\Omega = [-1,1]$ and $V$ is spanned by polynomials up to degree $N-1$, then one choice for the $v_j$ basis functions are orthonormal Legendre polynomials. We will consider this particular case as an example several times in this paper.

\subsection{Riesz representors}\label{ssec:riesz-rep}
We consider the dual $V^\ast$ of $V$, i.e., the space of all bounded linear functionals mapping $V$ to $\R$. The Riesz representation theorem guarantees that a functional $L \in V^\ast$ can be associated with a unique $V$-representor $\ell \in V$ satisfying
\begin{align*}
  L(u) &= \left\langle u, \ell \right\rangle, & \forall \;\; u \in V.
\end{align*}
Furthermore, this $L \leftrightarrow \ell$ identification is an isometry. We will use these facts in what follows. Given $L$ that identifies $\ell$, we consider the coordinates $\widehat{\ell}_j$ of $\ell$ in a $V$-orthonormal basis,
\begin{align*}
  \ell(x) &= \sum_{j=1}^N \widehat{\ell}_j v_j(x), & \widehat{\ell}_j &= \left\langle \ell, v_j \right\rangle = L(v_j).
\end{align*}
Then we have the following relations:
\begin{align*}
  \left\| L \right\|_{V^\ast} = \left\| \ell\right\|_V &= \| \bs{\widehat{\ell}} \|_2, & \bs{\widehat{\ell}} &= \left( \widehat{\ell}_1, \; \widehat{\ell}_2, \; \ldots, \; \widehat{\ell_N} \right)^T,
\end{align*}
where $\|\cdot\|_2$ is the Euclidean norm on vectors in $\R^N$. 

\subsection{Least squares problems}
We are interested in a common least squares-type approximation problem. Suppose that $u \in H$ is an unknown function for which we have $M$ pieces of data. We wish to construct an approximation $p \in V$ to $u$ that best matches these data points. We now formulate this abstractly. Let $\phi_1, \ldots, \phi_M$ be $M$ linear functionals on $H$ that are bounded on $V$. We assume that the observations $\left\{ u_j \right\}_{j=1}^M = \left\{ \phi_j(u) \right\}_{j=1}^M \subset \R$ are available to us (and also bounded), and we seek to solve the optimization problem,
\begin{align*}
  p = \argmin_{v \in V} \sum_{j=1}^M \left( \phi_j(v) - u_j \right)^2.
\end{align*}
For example, if $\phi_j$ is a point-evaluation (the Dirac mass) at some location $x_j \in \Omega$ for each $j=1, \ldots, M$, then the problem above is equivalent to
\begin{align*}
  p = \argmin_{v \in V} \sum_{j=1}^M \left( v(x_j) - u(x_j)\right)^2.
\end{align*}
This problem has a unique solution if the matrix $\bs{A} \in \R^{M \times N}$ with entries
\begin{align*}
  (A)_{m,n} &= \phi_m(v_n), & 1 \leq n \leq N, & 1 \leq m \leq M,
\end{align*}
has the rank equal to $\dim V = N$; otherwise, infinitely many solutions exist. This least squares problem is well understood and computational algorithms to solve it given data $\phi_j(u)$ are ubiquitous \cite{boyd_introduction_2018}. For an overdetermined system, where $M>N$, the method of ordinary least squares can be used to find a solution. Some details for this are discussed in \ref{sec:disc}.

\subsection{Constraints}
The previous section explains how a function $p$ can be constructed from data. However, we are interested in a particular kind of \textit{constrained} approximation.
Our investigation can be motivated by the following examples of types of constraints:
\begin{itemize}
  \item Positivity: $p(x) \geq 0$ for all $x \in \Omega$
  \item Monotonicity: $p'(x) \geq 0$ for all $x \in \Omega \subset \R$
  \item Boundedness: $0 \leq p(x) \leq 1$ for all $x \in \Omega$.
\end{itemize}
Thus, the central focus of this paper is solving a \textit{linearly constrained} least squares problem, where constraints of the above type are imposed. We now give the abstract setup of our constraints, which specializes to the examples above.

Our abstraction defines $K$ \textit{families} of linear constraints; for fixed $k \in \{1, 2, \ldots, K\}$, family $k$ is prescribed by the tuple $(L_k, r_k, \omega_k)$: 
\begin{itemize}
  \item $\omega_k$: a subset of $\Omega$.
  \item $r_k$: an element of $V$
  \item $L_k$: for $y \in \omega_k$ fixed, $L_k(\cdot,y)$ is a $y$-parameterized \textit{unit norm} element of $V^\ast$.
\end{itemize}
Our $k$th family of constraints on $v$ is 
\begin{align}\label{eq:Lk-def}
  L_k(v,y) &\leq L_k(r_k,y), & y &\in \omega_k.
\end{align}
The subset of $V$ that satisfies constraint family $k$ is
\begin{equation}\label{eq:constraints}
  E_k \coloneqq \left\{ v \in V \; \big| \; L_k( v, y) \leq L_k( r, y) \textrm{ for all } y \in \omega_k \right\}.
\end{equation}
The elements in $V$ that satisfy all $K$ families of constraints simultaneously are
\begin{align}\label{eq:E-def}
  E \coloneqq \cap_{k=1}^K E_k.
\end{align}
We assume that $E$ is nonempty, i.e., that the constraints are consistent. Constraints can be inconsistent, e.g., simultaneously enforcing $f(x) \leq 0$ and $f(x) \geq 1$. However, one can create more subtle inconsistencies in more complicated settings. Our procedure does not provide a means to detect inconsistent constraints (and in this case the algorithm will simply not converge). Thus, we rely on the user to ensure consistent constraints. (Note that a corresponding constrained problem has no solution if inconsistent constraints are prescribed.)

Particularly important later will be the formula for the $\{v_j\}_{j=1}^N$-coordinates of the Riesz representor of $L_k$. As in Section \ref{ssec:riesz-rep}, $L_k(\cdot,y)$ for fixed $(k,y)$ can be identified with its Riesz representor $\ell_k(\cdot, y) \in V$ and its corresponding expansion coefficients $\bs{\widehat{\ell}}_k(y)$. The unit norm condition of $L_k$ then implies

\begin{align}\label{eq:Lk-normalized}
  \left\| L_k(\cdot,y) \right\|_{V^\ast}^2 = \left\|\bs{\widehat{\ell}}_k(y) \right\|^2_2 = \sum_{j=1}^N \left(L_k(v_n,y)\right)^2 = 1.
\end{align}
We consider some examples.
\begin{example}[Positivity]\label{ex:positivity}
  Consider $\Omega = [-1,1]$, and let $V$ be any $N$-dimensional subspace of $L^2(\Omega) \cap L^\infty(\Omega)$. We seek to impose $p(x) \geq 0$ for all $x \in \Omega$. Thus, we have $K=1$, and the linear operator $L_1$ should be point evaluation, appropriately normalized. Note that point evaluation is not a bounded functional in $L^2$, but it is on the finite-dimensional space $V$. Formally, this is
  \begin{align*}
    L_1(v,y) &\coloneqq -\lambda(y) v(y), & v &\in V,
  \end{align*}
  where $\lambda(y)$ is chosen so that $L_1$ has unit norm for every $y \in \omega_k$; the negative sign is chosen so that we can reverse the inequality in \eqref{eq:Lk-def}. We set $\omega_k = \Omega$, and choose $r_k \equiv 0 \in V$. Then, the constraint \eqref{eq:Lk-def} is equivalent to $v(y) \geq 0$ for every $y \in \Omega$.
  The constraint set $E_1$ defined in \eqref{eq:constraints} is
  \begin{align*}
    E_1 = \left\{ u \in V \;\big|\; -u(y) \leq 0 \textrm{ for all } y \in \omega_1 \right\}.
  \end{align*}
  It will be useful here to also demonstrate how $\bs{\widehat{\ell}}_1(y)$ can be computed.
  For fixed $y$, we can identify $L_1(\cdot, y)$ via its Riesz representor $\ell_1(\cdot, y)$:
  \begin{align}\label{eq:Riesz-example}
    \ell_1(\cdot, y) &\coloneqq -\lambda(y) \sum_{j=1}^N v_j(y) v_j(\cdot) \in V, & 
    \lambda(y) &= \left[ \sum_{j=1}^N v^2_j(y)\right]^{-1/2}, 
  \end{align}
  so that $\left\{ -\lambda(y) v_j(y) \right\}_{j=1}^N$ are the entries of \an{$\bs{\widehat{\ell}}_1(y)$}. The formula for $\lambda$ results from the normalization condition \eqref{eq:Lk-normalized}. Thus, the coefficient vector $\bs{\widehat{\ell}}_1(y) \in \R^N$ has explicit entries in terms of $y$ and the orthonormal basis $\{v_j\}_{j=1}^N$.
\end{example}

\begin{example}[Monotonicity]
  With the same setup as the previous example, we take $V$ as any $N$-dimensional subspace of $L^2(\Omega) \cap W^{1,\infty}(\Omega)$, where $W^{1,\infty}(\Omega)$ is the Sobolev space of functions that are in $L^\infty(\Omega)$ and whose derivatives are also in $L^\infty(\Omega)$. Again with $K=1$, we define $L_1$ and its corresponding Riesz representor as 
  \begin{align*}
    L_1(v,y) &\coloneqq -\tau(y) v'(y), & v &\in V, \\
    \ell_1(\cdot, y) &\coloneqq -\sum_{n=1}^N \tau(y) v_n'(y) v_n \in V, & \tau(y) &= \left[ \sum_{j=1}^N \left(v_j'\right)^2(y) \right]^{-1/2},
  \end{align*}
  where again $\tau$ is determined using the normalization condition \eqref{eq:Lk-normalized}. With $r_1 \equiv 0$, then \eqref{eq:constraints} enforces $v'(y) \geq 0$ for all $y \in \Omega$.
\end{example}

\begin{example}[Boundedness]
  With the same setup as Example \ref{ex:positivity}, we take $V$ as any $N$-dimensional subspace of $L^2 \cap L^\infty$, and we further assume that $V$ contains constant functions. Let $K = 2$, and define the operators $L_1$ and $L_2$ as
  \begin{align*}
    L_1(v,y) &\coloneqq -v(y), & v &\in V, \\
    L_2(v,y) &\coloneqq v(y), & v &\in V,
  \end{align*}
  for each $y \in \omega_1 = \omega_2 = [-1,1]$. Then, with constraint functions $r_1 \equiv 0$ and $r_2 \equiv 1 \in V$, we have that $E_k$, $k = 1, 2$ are the sets
  \begin{align*}
    E_1 &= \left\{ u \in V \; \big|\; -u(y) \leq 0 \; \forall \; y \in [-1,1]\right\}, & 
    E_2 &= \left\{ u \in V \; \big|\; u(y) \leq 1\; \forall\; y \in [-1,1] \right\},
  \end{align*}
  so that their intersection $E$ in \eqref{eq:E-def} is the set of elements $u$ in $V$ such that $0 \leq u(x) \leq 1$ for each $x \in \Omega$.
\end{example}

\begin{example}
  We can also form constraints on different subsets of $\Omega$. With all the notation in the previous example, we change only:
  \begin{align*}
    \omega_1 &= [-1, 0), & \omega_2 &= (0, 1],
  \end{align*}
  so that $E$ contains functions $u$ satisfying $u(x) \geq 0$ for $x \in [-1, 0)$ and $u(x)\leq 1$ for $x \in (0, 1]$.
\end{example}

The above examples illustrate the generality of our notation and the intuitive simplicity of the types of constraints that we consider. A constrained version of a least squares problem thus is formulated as
\begin{align}\label{eq:constopt-continuous}
  p = \argmin_{v \in E} \sum_{j=1}^M \left( \phi_j(v) - u_j \right)^2.
\end{align}

\subsection{Problem discretization}\label{sec:disc}
We now formulate the constrained problem \eqref{eq:constopt-continuous} via coordinates in the basis $\{v_j\}_{j=1}^N$, which results in a discrete form amenable to numerical computation. Any $v \in V$ has the expansion
\begin{align*}
  v(x) = \sum_{j=1}^N \widehat{v}_j v_j(x),
\end{align*}
and the expansion coefficient vector $\bs{\hat{v}} \coloneqq \left( \hat{v}_1, \ldots, \hat{v}_N \right)^T \in \R^N$ uniquely identifies the element $v \in V$. 
This identification defines subsets of $\R^N$ corresponding to the sets $E_k$:
\begin{align}\label{eq:C-def}
  C_k &\coloneqq \left\{ \bs{c} \in \R^N \; \big| \; \sum_{j=1}^N c_j v_j \in E_k \right\} \subset \R^N, & C &\coloneqq \bigcap_{k=1}^K C_k.
\end{align}
Then, the optimization problem \eqref{eq:constopt-continuous} is equivalent to 
\begin{align}\label{eq:constopt-discrete}
  \bs{c} &= \argmin_{\bs{\widehat{v}} \in C} \left\| \bs{A} \bs{\widehat{v}} - \bs{b} \right\|^2_2, & b_j &= \phi_j(u).
\end{align}
This problem is again a least squares problem and so is easily solved in principle, but unfortunately in practice the set $C$ is a quite complicated subset of $\R^N$. Nevertheless, $C$ is convex, which is a fact we exploit.

If $\bs{A}$ has full column rank, then the \textit{unconstrained} solution to \eqref{eq:constopt-discrete} (i.e., setting $C = \R^N$) is given by the solution to the normal equations \cite{anton2013elementary},
\begin{align*}
\bs{\widehat{v}} = (\bs{A}^T \bs{A})^{-1}\bs{A}^T \bs{b}
\end{align*}

\subsection{Geometry of sets}
We recall some basic properties of cones and convex sets and functions that we utilize. In all the discussion below, the ambient space is $\R^N$. A set $C$ is convex if, for every $x, y \in \partial C$, 
\begin{equation*}
  \lambda x + (1-\lambda) y \in C \hskip 10pt \forall \;\; \lambda \in (0,1),
\end{equation*}
A set $C$ is a convex cone if, for every $x, y \in C$,
\begin{equation*}
  a x + b y \in C \hskip 10pt \forall \;\; a, b \geq 0.
\end{equation*}
The set $C$ is an affine convex cone if it is the rigid translate of a convex cone, i.e., if $C = D + z$, where $z \in \R^N$ and $D$ is a convex cone. In this case, we call $z$ the vertex of the cone.

The convex sets we consider are generated by an uncountably infinite number of supporting hyperplanes. Given $y \in \R^N$ and $a \in \R$, a hyperplane $H_0$ is a set given by $H_0 = \left\{ x \; \big|\; \left\langle x, y \right\rangle = a \right\}$. The hyperplane $H_0$ separates $\R^N$ into two halfspaces, one of which is 
\begin{equation*}
  H(y, a) \coloneqq \left\{ x \in \R^N \; \big|\; \left\langle x, y \right\rangle \leq a \right\}
\end{equation*}
Note that $H(y,a)$ is a closed set in $\R^N$. A hyperplane $H_0(y,a)$ with an associated halfspace $H(y,a)$ is a supporting hyperplane for a closed convex set $C$ if $C \subset H(y,a)$ and if $H_0(y,a) \cap \partial C \neq \emptyset$.

\section{Constrained optimization}\label{sec:method}
The main task in this paper is solving the optimization problem \eqref{eq:constopt-discrete}. This optimization problem appears simple since it features a quadratic objective, and the feasible set $C$ is convex (which we show in the next section). The main difficulty here is that $C$ is not a computationally simple convex set in $\R^N$, and hence computing, e.g., projections onto this set, is difficult. To begin, we establish that $C$ is convex.

\subsection{Constraint set properties}
This section is devoted to establishing that the sets $C_k$ and $C$ are convex cones in $\R^N$. These properties will be used in the construction of algorithms for solving \eqref{eq:constopt-discrete}. 

Before proceeding, we note that each inequality function $r_k \in V$ for $k = 1, \ldots, K$, can be translated into its vector of expansion coefficients:
\begin{align}\label{eq:rla-def}
  r_k(x) &= \sum_{j=1}^N \hat{r}_{k,j} v_j(x), & \bs{\hat{r}}_k = \left( \hat{r}_{k,1}, \ldots, \hat{r}_{k,N} \right)^T.
\end{align}
Now the definitions of $C$ and $C_k$ immediately yield convexity and conic properties of these sets.
\begin{theorem}\label{thm:C-convex}
  The set $C$ is a closed convex set in $\R^N$, and each for $k = 1, \ldots, K$, $C_k$ is a closed, affine convex cone in $\R^N$ with vertex located at $\bs{\hat{r}}_k$. 
\end{theorem}
\begin{proof}
  Convexity, closure, and conic structure are preserved under isometries. Due to the isometric relation between $V$ and $\R^N$, we can thus prove properties in one space, which extends to the other space. We first show that $C_k$ is closed directly in $\R^N$: 
  Rewriting \eqref{eq:C-def} using the definition of $E_k$, we have
  \begin{align*}
    C_k 
        &= \bigcap_{y \in \omega_k} \left\{ \bs{c} \in \R^N \; \big| \; L_k\left(\sum_{j=1}^N c_j v_j, y\right) \leq L_k(r_k, y) \right\} \eqqcolon \bigcap_{y \in \omega_k} c_k(y).
  \end{align*}
  By definition, $c_k(y)$ is actually a halfspace in $\R^N$,
  \begin{align*}
    c_k(y) = H\left(\bs{\widehat{\ell}}_k(y), L_k(r_k,y) \right)
  \end{align*}
  and hence $c_k(y)$ is a closed set. Therefore, $C_k = \cap_y c_k(y)$ is also a closed set, and thus $C = \cap_k C_k$ is a closed set.\\
  We will now show the convexity and conic properties in $V$: fix $k \in \{1, \ldots, K\}$ and $y \in \omega_k$. Let $v, w \in V$ be two elements in $E_k$. 
  For any $\lambda \in [0, 1]$, 
  \begin{align*}
    L_k(\lambda v + (1-\lambda) w, y) = \lambda L_k(v, y) + (1-\lambda) L_k(w, y) \leq L_k(r_k,y),
  \end{align*}
  where the inequality is true since $v, w \in E_k$. Therefore $E_k$, and hence $C_k$, is convex. Thus we also have that $C$ is convex since it's an intersection of convex sets.
  
  We next show that $E_k$ is a cone with the vertex at $r_k$, i.e., we must show that for any $\tau \geq 0$ and $v \in V$, we have $L_k(r_k + \tau( v - r_k)) \leq L_k(r_k)$. This is true since
  \begin{align*}
    L_k(r_k + \tau( v - r_k)) = L_k(r_k) + \tau \left[ L_k(v) - L_k(r_k)\right] \leq L_k(r_k),
  \end{align*}
  so indeed, $E_k$ is a convex cone with the vertex at $r_k$, and hence $C_k$ is a convex cone with the vertex at $\bs{\widehat{r}}_k$.
\end{proof}

Despite their conic convexity, the sets $C_k$ are not polyhedral in general, and are hence ``complicated" to computationally encode.
Consider the setup of Example \ref{ex:positivity}. If we change the definition of $\omega_1$ to 
\begin{align*}
  \widetilde{\omega}_1 = \{x_1, \ldots, x_P \} \subset \Omega = [-1,1].
\end{align*}
for any arbitrary $P < \infty$, the new constraint set $\widetilde{C}_1 = C(L_1, r_1, \widetilde{\omega}_1)$ is strictly larger than the constraint set $C_1$ in Example \ref{ex:positivity}. In particular, $p \in V$ satisfying $p(x_j) \geq 0$ for $j = 1, \ldots, P$ does not imply that $p(x) \geq 0$ for all $x \in [-1,1]$ unless $V$ has very special properties (for example, if $V$ contains only certain piecewise constant functions). Note that the supporting hyperplanes of the constraint set $\widetilde{C}$ are $P < \infty$ halfspaces in $\R^N$ and hence $\widetilde{C}$ is polyhedral (if nonempty). However, if $V$ contains polynomials, it is easy to construct a polynomial that is non-negative on $\widetilde{\omega}_1$ but \textit{not} non-negative on $\Omega$. Hence, the constraint set $C_1$ defined by $(L_1, r_1, \omega)$ in example \ref{ex:positivity} is strictly smaller than $\widetilde{C}_1$, here defined by $(L_1, r_1, \widetilde{\omega})$. 

\an{Nevertheless, such discretization approaches, i.e. approaches that use a finite set $\widetilde{\omega}_1$ as a surrogate for an infinite set $\Omega$, are common and frequently effective algorithms for solving \eqref{eq:constopt-discrete}, as is commonly done in semi-infinite programming problems. However, in this manuscript we present algorithms that insist on global satisfaction of the constraints, and hence adopt alternative approaches.}
Thus, the main computational difficulty of our optimization problem is that the set $C$ cannot be exactly represented as a polyhedron in general, and in particular that projections onto $C$ are in general difficult to compute.

\subsection{Solutions to \eqref{eq:constopt-discrete}}
Our main goal in this section is to demonstrate the unique solution to our constrained optimization problem. The result is straightforward from the closed convexity of the constraint set and strict convexity of the objective function.
\begin{theorem}\label{thm:solution}
  Assume that the design matrix $\bs{A}$ has rank $N$ and the feasible set $C$ is nonempty. Then, the constrained optimization problem \eqref{eq:constopt-discrete} has a unique solution.
\end{theorem}
\begin{proof}
  The first step is to observe that since $\bs{A}$ has full column rank, we can write the problem in transformed coordinates as a convex feasibility problem (specifically as a projection problem).
  Let $\bs{A} = \bs{U} \bs{\Sigma} \bs{V}^\ast$ be the \textit{reduced} singular value decomposition of $\bs{A}$. Since $\mathrm{rank}(\bs{A}) = N \leq M$, $\Sigma$ is $N \times N$, diagonal, and invertible; $\bs{V}$ is $N \times N$ and orthogonal; and $\bs{U}$ is $M \times N$ with orthonormal columns.

  With $\mathcal{P}_{\mathcal{W}}$ the $\R^N$-orthogonal projector onto a subspace $\mathcal{W}$, and $\mathcal{R}(\bs{A})$ the range of $\bs{A}$, then \eqref{eq:constopt-discrete} can be written as
  \begin{align}\nonumber
    \argmin_{\bs{\widehat{v}} \in C} \left\| \bs{A} \bs{\widehat{v}} - \bs{b} \right\|^2_2 &=
    \argmin_{\bs{\widehat{v}} \in C} \left\| \mathcal{P}_{\mathcal{R}(\bs{A})^\perp} \bs{b} \right\|_2^2 + \left\| \bs{A} \bs{\widehat{v}} - \mathcal{P}_{\mathcal{R}(\bs{A})} \bs{b} \right\|^2_2 \\\nonumber
    &= \argmin_{\bs{\widehat{v}} \in C} \left\| \bs{\Sigma} \bs{V}^\ast \bs{\widehat{v}} - \bs{U}^\ast \bs{b} \right\|^2_2 \\\label{eq:transformed-opt-full}
    &= \bs{V} \bs{\Sigma}^{-1} \argmin_{\bs{z} \in \bs{\Sigma} \bs{V}^\ast C} \left\| \bs{z} - \bs{U}^\ast \bs{b} \right\|^2_2,
  \end{align}
  where $\bs{\Sigma} \bs{V}^\ast C \coloneqq \left\{ \bs{\Sigma} \bs{V}^\ast \bs{y} \in \R^N \; \big| \; \bs{y} \in C \right\}$.
  Thus, \eqref{eq:constopt-discrete} has a unique solution if and only if 
  \begin{align}\label{eq:transformed-opt}
    \argmin_{\bs{z} \in \bs{\Sigma} \bs{V}^\ast C} \left\| \bs{z} - \bs{U}^\ast \bs{b} \right\|^2_2
  \end{align}
  has a unique solution. 
  Theorem \ref{thm:C-convex} establishes that $C$ is closed and convex; thus, $\bs{\Sigma} \bs{V}^\ast C$ is a linear transformation of a closed convex set, so it is also closed and convex. Therefore, \eqref{eq:transformed-opt} seeks the $\ell^2(\R^N)$-closest point to $\bs{U}^\ast \bs{b}$ from a nonempty, closed, convex set. The Hilbert Projection Theorem guarantees the existence and uniqueness of such a point.
\end{proof}
The study of existence and uniqueness of approximations under convex constraints is not new \cite{rice_approximation_1963,lewis_approximation_1973}. Indeed, our result is a corollary of these earlier results, but we have presented a brief proof above in order to be self-contained.

\section{Algorithms: Convex Feasibility}\label{sec:algorithms}
We now concentrate on solving the problem defined by \eqref{eq:constopt-discrete}, equivalently \eqref{eq:transformed-opt}. To simplify the presentation, we will assume first that $\bs{A} = \bs{I}$ so that both \eqref{eq:transformed-opt} and \eqref{eq:constopt-discrete} reduce to
\begin{align}\label{eq:reduced-discrete}
  \argmin_{\bs{c} \in C} \left\| \bs{c} - \bs{b} \right\|^2_2,
\end{align}
i.e., a standard problem of projecting $\bs{b}$ onto a convex set $C$. The main bottleneck to applying standard optimization tools is that the feasible set $C$ is not easily defined in terms of a finite number of conditions on $\bs{c}$. The difficulty in our problem is not in minimizing the objective function, but instead the convex feasibility problem, i.e., to identify points in the convex feasible set. 

Some of the most successful algorithms for solving the convex feasibility problem are alternating- or splitting-type algorithms. If $C_1, \ldots, C_r$ are convex sets with non-empty intersection $C$, these algorithms assume that projection onto any one of these sets is computationally feasible. A solution to \eqref{eq:reduced-discrete} can be computed by alternating these individual projections. The original projection onto convex sets algorithm via iteration is due to Von Neumann \cite{von_neumann_functional_1951}, and much work has proceeded from this \cite{bregman_method_1965,gubin_method_1967,bauschke_projection_1996,deutsch_rate_2006,lewis_local_2009,deutsch_best_2012}. When $r > 2$, the alternating algorithm becomes a cyclic one, and these cyclic projection algorithms have substantial theoretical underpinning, including convergence guarantees.

The difficulty in applying these algorithms to our situation is that they characterize the feasible region with a \textit{finite} number of convex sets. Although our collection of sets $\{C_j\}_{j=1}^K$ is finite, we do not know how to project onto any of them individually. However, we have 
\begin{align}\label{eq:C-decomposition}
  C &= \bigcap_{k=1}^K C_k = \bigcap_{k=1}^K \bigcap_{y \in \omega_k} H_k(y), \\\nonumber
  H_k(y) &\coloneqq H\left(\bs{\ell_k}(y), L_k(r_k, y) \right),
\end{align}
so that $C$ is comprised of an (in general uncountably) \textit{infinite} intersection of half-spaces, each of which is straightforward to project onto, see \cref{fig:intersectingplanes} for a geometric visual. Our strategy here is to generalize certain types of cyclic/alternating algorithms to the case of an infinite number of convex sets (halfspaces). We broadly employ two strategies: greedy projection and averaged projection.

\begin{figure}[htbp]
  \begin{center}
    \resizebox{\textwidth}{!}{
      \includegraphics[width=0.33\textwidth]{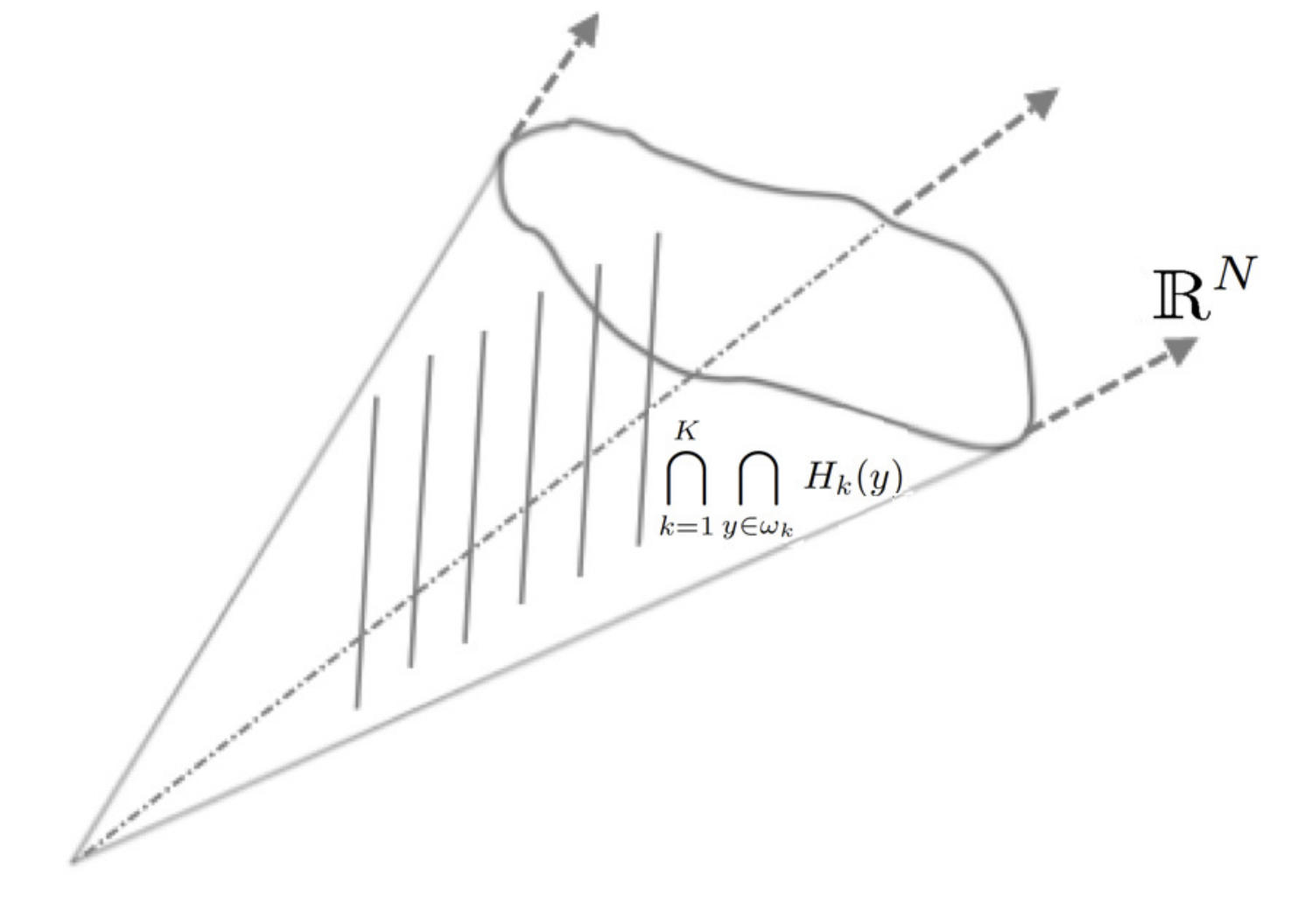}
      \includegraphics[width=0.33\textwidth]{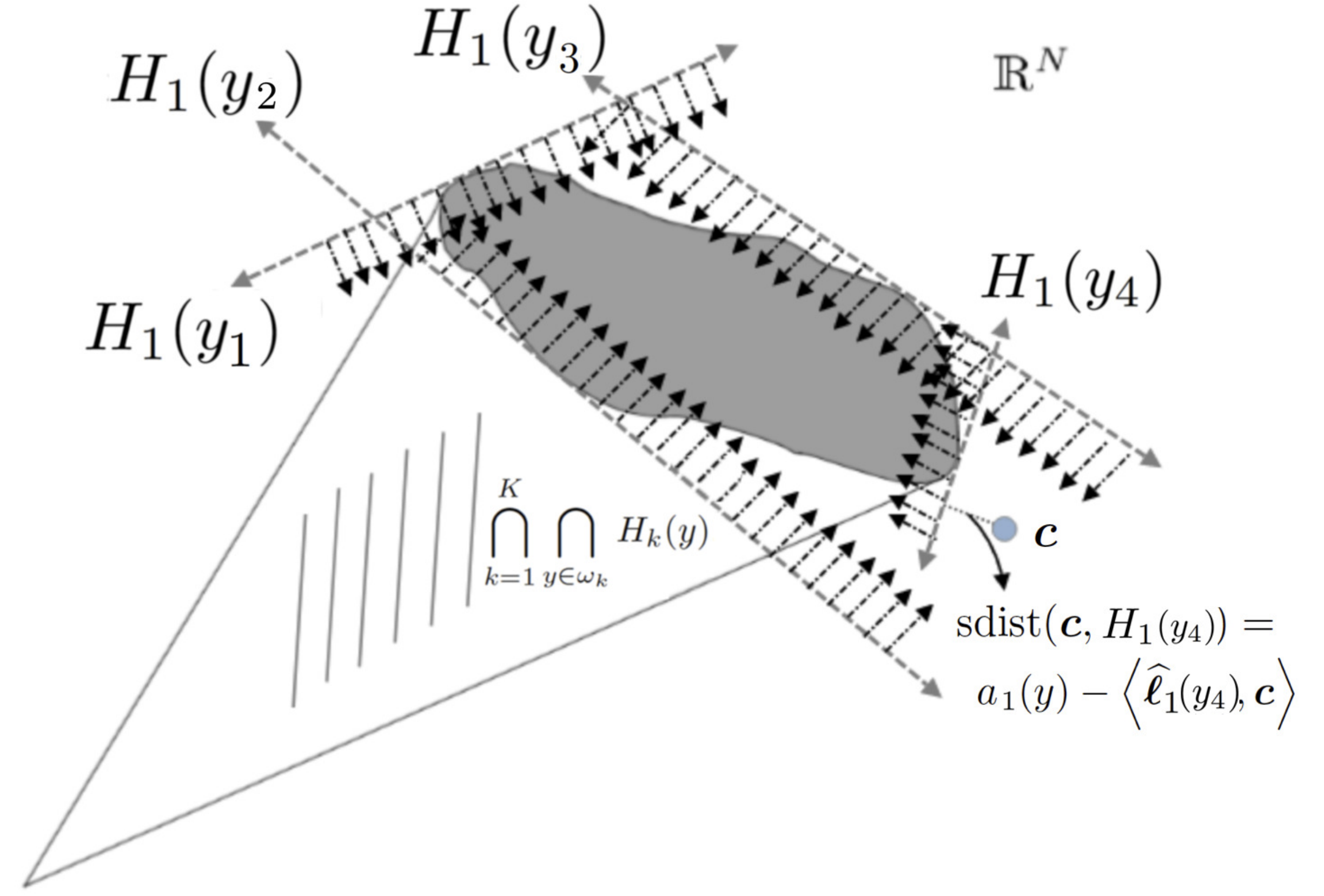}
      \includegraphics[width=0.33\textwidth]{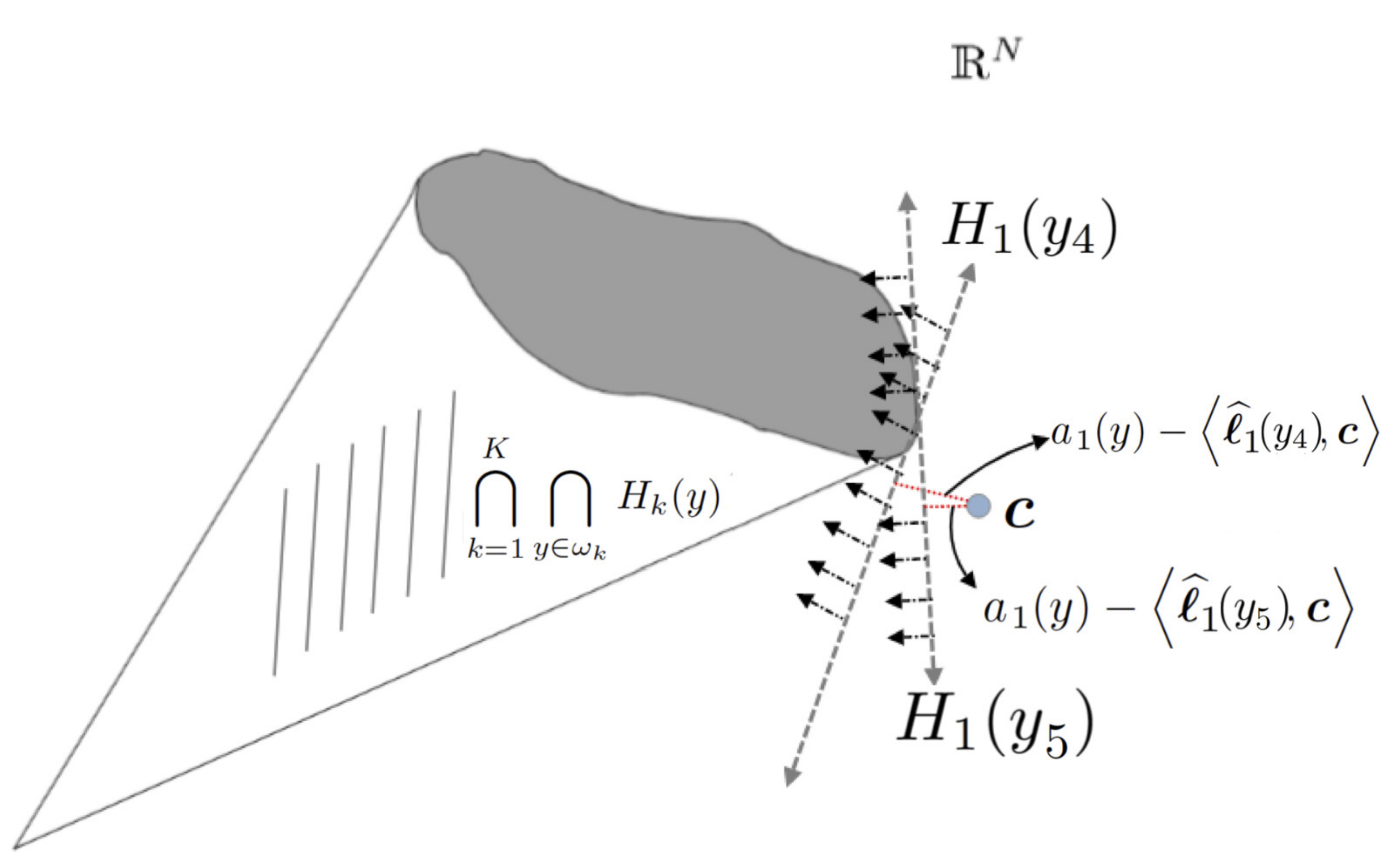}
    }
  \end{center}
 \caption{
   Left: The hatched volume represents the closed convex cone $C_1$
   Middle: Geometric depiction of intersecting hyperspaces $H_1(y)$ and their respective boundaries defined by hyperplanes parameterized by $y \in \Omega$. Also shown is the distance calculation corresponding to \eqref{eq:sdist}. 
   Right: A scenario that demonstrates the greedy strategy to select the direction in which $y$ moves in the next step of the algorithm: $H_1(y_4)$ is farther away from $\bs{c}$ than $H_1(y_5)$. The optimization \eqref{eq:global-minimization} seeks the hyperplane that is farthest away from $\bs{c}$.} \label{fig:intersectingplanes}
\end{figure}


The major ingredient in our approaches is the ability to project onto any halfspace $H_k(y)$. Since the functionals $L_k(\cdot,y)$ are unit norm, a computation shows that the signed distance between some point $\bs{c} \in \R^N$ and $H_k(y)$ is 
\begin{align}\label{eq:sdist}
  \mathrm{sdist}(\bs{c}, H_k(y)) = L_k(r_k, y) - \left\langle \bs{\widehat{\ell}}_k(y), \bs{c} \right\rangle,
\end{align}
which is positive if $\bs{c} \in H_k(y)$ and negative otherwise. Thus, the nearest-distance projection of $\bs{c}$ onto $H_k(y)$ is 
\begin{align*}
  P_{H_k(y)} \bs{c} = \bs{c} + \bs{\ell}_k(y) \min\left\{0, \mathrm{sdist}(\bs{c}, H_k(y)) \right\}.
\end{align*}
We consider an example to illustrate that these projections are easily computable.

\begin{example}\label{eq:positivity-lambda}
  Consider the positivity constraint setup of Example \ref{ex:positivity}. The constraint functional $L_1(\cdot,y)$ is a (normalized, negative) point evaluation at $y$, and $\{v_n\}_{n=1}^N$ are the first $N$ \textit{orthonormal} Legendre polynomials on $[-1,1]$. Then, the Riesz representor $\ell_1(y) \in V$ and its coordinates $\{\widehat{\ell}_{1,j}(y) \}_{j=1}^N$ are explicit in terms of the Legendre polynomials via \eqref{eq:Riesz-example}.
  In the context of harmonic analysis, $\ell_1(y)$ is the $y$-centered, negative, normalized Dirichlet kernel for $V$.
  The function $r_1$ describing the constraint is $r_1 \equiv 0$, so that $\widehat{\bs{r}}_1 = \bs{0}$ and $L_1(r_1,y) = 0$.
  Now let $v \in V$ be any element with coordinates $\bs{c} \in \R^N$ in the orthonormal Legendre polynomials. Then, 
  \begin{align}\label{eq:sdist-1d}
    \mathrm{sdist}(\bs{c}, H_1(y)) = -\left\langle \bs{\widehat{\ell}}_1(y), \bs{c} \right\rangle = \lambda(y) v(y).
  \end{align}
  Thus, the signed distance at $y \in \Omega$ is simply scaled evaluation of the original function $v$. The projection of $\bs{c}$ onto the halfspace defined by $H_k(y)$ is therefore
  \begin{align*}
    P_{H_k(y)} \bs{c} = \bs{c} + \bs{\widehat{\ell}}_1(y) \min\left\{0, v(y) \lambda(y) \right\}.
  \end{align*}
  Note that since $\lambda(y) > 0$, this projection equals $\bs{c}$ if $v(y) \geq 0$, as expected.
\end{example}

\subsection{Greedy projections}\label{ssec:alg-greedy}
Since projections onto individual halfspaces defined by $H_k(y)$ are relatively simple to compute, we can devise one algorithm for computing the solution to \eqref{eq:reduced-discrete} as a modification of cyclic projections. Although cyclic projection-type algorithms proceed by cycling through the enumerable constraint sets, our (uncountably) infinite collection of sets prevents such a simple cycling. Instead, we can project onto the \textit{farthest} or most violated constraint, i.e., with 
\begin{align}\label{eq:global-minimization}
  (y^\ast, k^\ast) \coloneqq \argmin_{y\in \omega_k, k \in [K]} \mathrm{sdist}(\bs{c}, H_k(y)),
\end{align}
We can update $\bs{c}$ via
\begin{align}\label{eq:greedy-update}
  \bs{c} \gets \bs{c} + \bs{\ell}_{k^\ast}(y^\ast) \min\left\{0, \mathrm{sdist}(\bs{c}, H_{k^\ast}(y^\ast)) \right\}.
\end{align}
The geometric picture associated to \eqref{eq:global-minimization} is shown in the right panel of Figure \ref{fig:intersectingplanes}.
The update process \eqref{eq:greedy-update} can be repeated, resulting in an iterative algorithm. We summarize this procedure in Algorithm \ref{alg:greedy}. This algorithm proceeds by iteratively ``correcting" the vector $\bs{c}$ in \eqref{eq:greedy-update}. The associated operation in the function space $V$ is that an unconstrained function is additively augmented by the Riesz representor correction function $\ell_{k^\ast}(y^\ast) \in V$. These corrections are visualized in Figure \ref{fig:corrections} for polynomials. A more detailed understanding of these function is provided in Figures \ref{fig:surf5} and \ref{fig:surf30} where we show $\ell_k(y)(x)$ as a function of $(x,y)$ for polynomials.

\begin{algorithm}[H]
  \caption{Iterative greedy projection algorithm to compute the solution to \eqref{eq:reduced-discrete}. The unspecified ``extra termination criteria" can be standard metrics, such as number of iterations, improvement in objective function, etc.}
\label{alg:greedy}
\begin{algorithmic}[1]
\STATE{Input: constraints $(L_k, r_k, \omega_k)_{k=1}^K$}
\STATE{Input: coordinates $\bs{c} \in \R^N$ of a function $v \in V$}
\WHILE{True}
  \STATE{Compute $(y^\ast, k^\ast)$ via \eqref{eq:global-minimization}.}\label{alg:greedy:minimization}
  \IF{$\mathrm{sdist}(\bs{c}, H_{k^\ast}(y^\ast)) \geq 0$ or extra termination criteria triggered}\label{alg:greedy:termination}
    \STATE{Break}
  \ENDIF
  \STATE\label{alg:greedy:update}{Update $\bs{c}$ via \eqref{eq:greedy-update}.}
\ENDWHILE
\RETURN $\bs{c}$
\end{algorithmic}
\end{algorithm}

Note that the bulk of the computational effort in Algorithm \ref{alg:greedy} corresponds to line \ref{alg:greedy:minimization} where the $\Omega$-global optimization problem \eqref{eq:global-minimization} must be solved, which can be of considerable expense at each iteration. We explain in Appendix \ref{app:poly-methods} how we accomplish this optimization for univariate polynomial spaces $V$. 

It is straightforward to establish that under a special kind of termination in Algorithm \ref{alg:greedy:update}, we obtain the solution to \eqref{eq:reduced-discrete}.
\begin{proposition}
  If Algorithm \ref{alg:greedy}, without any extra termination criteria, terminates after one only iteration of line \ref{alg:greedy:update}, then the output $\bs{c}$ is the solution to \eqref{eq:reduced-discrete}.
\end{proposition}
\begin{proof}
  Assume without loss that the input to algorithm \ref{alg:greedy} \an{$\bs{c}$ is not in $ C$}. By \eqref{eq:C-decomposition}, we have
  \begin{align*}
    \mathrm{dist} \left( \bs{c}, C \right) \geq \mathrm{dist} \left( \bs{c}, H_k(y) \right),
  \end{align*}
  for any $(y,k)$. Let $(y^\ast, k^\ast)$ be the solution to \eqref{eq:global-minimization}, and note that since $\bs{c} \not\in C$, 
  \begin{align*}
    \mathrm{dist} \left( \bs{c}, H_k(y) \right) = -\mathrm{sdist}(\bs{c}, H_{k^\ast}(y^\ast)) > 0.
  \end{align*}
  The assumption that Algorithm \ref{alg:greedy} terminates after one iteration implies that 
  \begin{align*}
    \bs{d} \coloneqq \bs{c} + \bs{\widehat{\ell}}_{k^\ast}(y^\ast) \mathrm{sdist}(\bs{c}, H_{k^\ast}(y^\ast)) \in C.
  \end{align*}
  Note $\bs{d}$ is returned by the algorithm. $\bs{c} \not\in C$, $\bs{d} \in C$, $\left\| \bs{\widehat{\ell}}_{k^\ast}(y^\ast) \right\|_2 = 1$, and that 
  \begin{align*}
    \mathrm{dist}\left(\bs{c}, C\right) \geq -\mathrm{sdist}(\bs{c}, H_{k^\ast}(y^\ast)),
  \end{align*}
  all imply that the above inequality is actually an equality, and thus $\bs{d}$ solves \eqref{eq:reduced-discrete}.
\end{proof}

In standard cyclic projection algorithms, it is well known that directly projecting onto each set in each iteration produces a suboptimal trajectory for the iterates. The greedy algorithm described in this section suffers from this as well, which we show in the numerical results section. An improvement that somewhat ameliorates this deficiency is accomplished by averaging these projections.

\subsection{Averaged projections}\label{ssec:alg-averaged}
A simple strategy to mitigate the oscillatory iteration trajectory produced by iterative greedy projections is via averaging. Precisely, given a current iterate $\bs{c}$, we identify the subset of $\Omega$ where our constraints are violated:
\begin{align}\label{eq:averaging-set}
  \omega_k^- &\coloneqq \left\{ y \in \omega_k \; \big|\; \mathrm{sdist}(\bs{c}, H_k(y)) < 0 \right\}.
\end{align}
Under mild assumptions on $V$, e.g., that it contains only piecewise continuous functions, $\omega_k^-$ is either the trivial (empty) set, or of positive Lebesgue measure. (In other words, it cannot be a discrete or nontrivial measure-0 set.) Assume for simplicity that $\omega_k^-$ has a positive Lebesgue measure for each $k$. We then produce an update by a normalized average of corrections corresponding to values of $y$ in $\omega_k^-$:
\begin{align}\label{eq:c-update-averaged}
  \bs{c} \gets \bs{c} + \sum_{k=1}^K \frac{1}{K |\omega_k^-|} \int_{\omega_k^-} \bs{\widehat{\ell}}_k(y) \mathrm{sdist}\left( \bs{c}, H_k(y) \right) \dx{y}.
\end{align}
Above, $|\omega_k^-|$ is the measure of $\omega_k^- \subset \Omega$. We again illustrate with an example that these quantities are computable.
\begin{example}
  Consider the positivity constraint setup of Example \ref{ex:positivity}. As we saw in Example \ref{eq:positivity-lambda}, the signed distance for our single constraint is given by \eqref{eq:sdist-1d}.
  Note that in this one-dimensional setup with finite-degree polynomials, the set $\omega_k^-$ is a finite union of subintervals of $[-1,1]$, and hence the measure $|\omega_k^-|$ is just the sum of the lengths of these subintervals. Then, the correction term on right-hand side of the update scheme \eqref{eq:c-update-averaged} is 
  \begin{align*}
    - \frac{1}{|\omega_k^-|} \int_{\omega_k^-} \bs{\widehat{\ell}}_{1}(y) \lambda(y) v(y) \dx{y} = - \frac{1}{|\omega_k^-|} \sum_{j=1}^N \bs{e}_j \int_{\omega_k^-} \lambda^2(y) v(y) v_j(y) \dx{y},
  \end{align*}
  where $\bs{e}_j$, $j \in [N]$ are the cardinal unit vectors in $\R^N$. Thus, the integrals that must be computed have smooth integrands and can be efficiently approximated by standard quadrature rules, assuming the endpoints of the subintervals defining $\omega_k^-$ can be identified.
\end{example}
A variation of Algorithm \ref{alg:greedy} that uses this averaging approach is nearly identical: the only change required is that the update of the coefficient vector $\bs{c}$ in line \ref{alg:greedy:update} should be replaced by the update in \eqref{eq:c-update-averaged}.

\Cref{fig:new4} visually depicts both the greedy and averaged projections idea where $V$ is a univariate space of polynomials and the constraint is positivity (i.e., Example \ref{ex:positivity}). In particular, the value $y^*$ that solves the greedy optimization problem \eqref{eq:global-minimization} is shown, along with the averaging set $\omega_1^-$ identified in \eqref{eq:averaging-set}. 

\begin{figure}[htbp]
  \centering
  \includegraphics[width=0.7\textwidth]{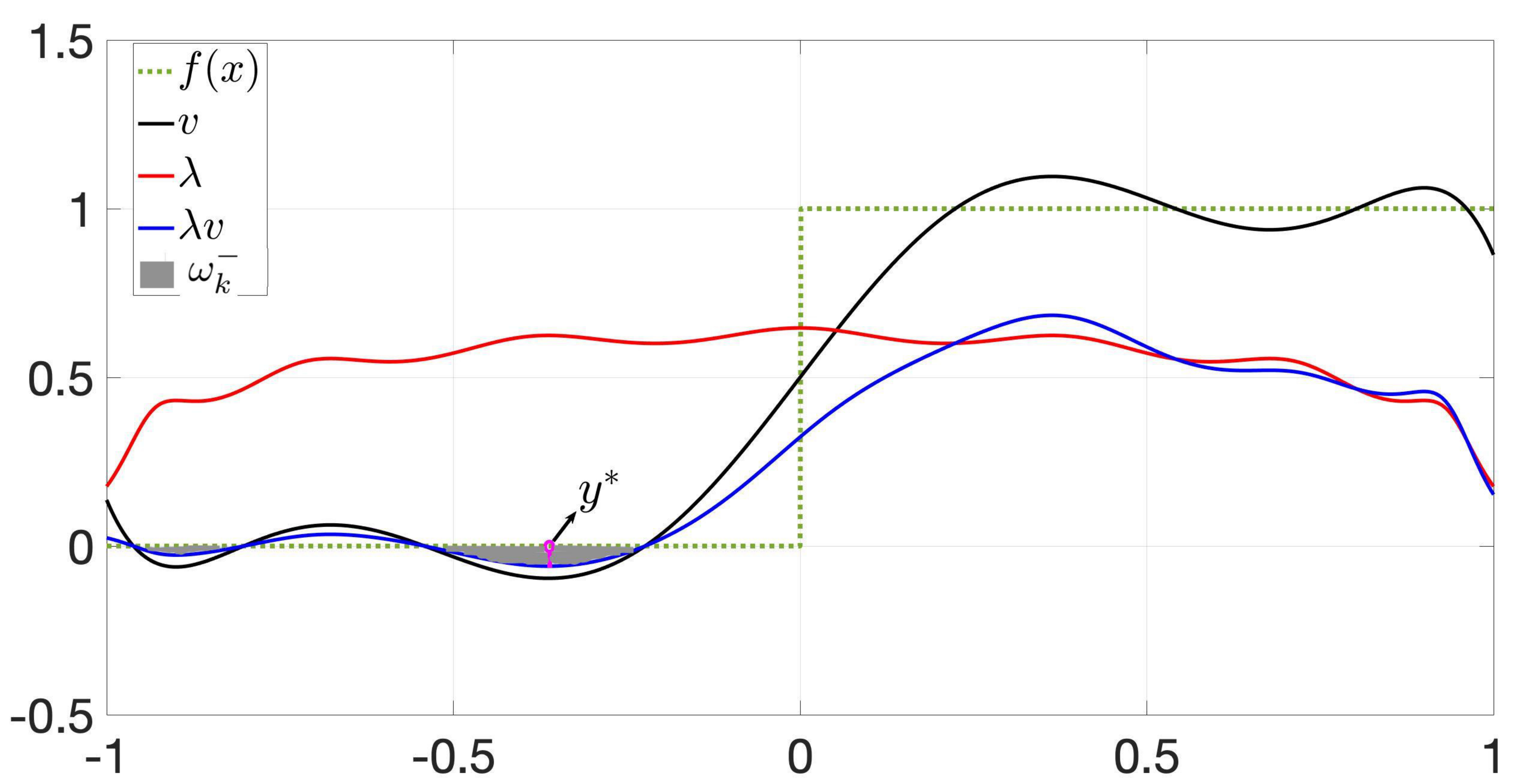}
  \caption{$v$ is the unconstrained $L^2([-1,1])$ projection of the step function $f(x)$ onto the space of degree-$7$ polynomials. For the positivity setup of Example \ref{ex:positivity}, the greedy point $y^\ast$ defined in \eqref{eq:global-minimization} is shown, and the averaging set $\omega_1^- \subset [-1,1]$ defined in \eqref{eq:averaging-set}. Also plotted is the signed distance $\lambda(y) v(y)$ of $v$ to $H_1(y)$.}\label{fig:new4}
\end{figure}

\subsection{Hybrid algorithms}
In experimentation, we have found that hybrid combinations of the greedy approach of Section \ref{ssec:alg-greedy} and the averaged approach of Section \ref{ssec:alg-averaged} work better than any algorithm alone. In particular, the greedy algorithm works well when $\bs{c}$ is ``close" to the solution, but the averaged algorithm works better for an iterate that is ``far" away. Thus, we utilize a standard switching procedure in optimization depending on the proximity to a basin of attraction. 

Through experimentation, we have found that the following switching mechanism works well:
We perform averaged projections until the norm of the correction \eqref{eq:c-update-averaged} reaches a certain tolerance. After a condition is met, we switch to greedy projections. The switching condition is the following: if $i$ is the iteration index, consider the ratio,
\begin{align*}
\alpha_i = \frac{\mathrm{sdist}(\bs{c}_i, H_{k^\ast_i}(y^\ast_i)}{\mathrm{sdist}(\bs{c}_{i-1}, H_{k^\ast_{i-1}}(y^\ast_{i-1})}.
\end{align*}
Our switching condition is triggered when $|\alpha_i - \alpha_{i-1}| \leq \epsilon$, for a user-specified $\epsilon$. At this point, we perform one more averaged update of the form \eqref{eq:c-update-averaged}, but multiply the right-hand side correction by $1/\alpha_i$. Subsequently, greedy projections as in \eqref{eq:greedy-update} are performed. While this procedure is quite \textit{ad hoc}, we have observed that it consistently performs better than other hybrid variants we have tried. 

\subsection{Algorithms for polynomial subspaces}

As described in previous sections, the main computational expense in our convex optimization algorithm is the minimization of the signed distance function in \eqref{eq:global-minimization} (for the greedy and hybrid algorithms) and identification and integration over the set $\omega_k^-$ in \eqref{eq:averaging-set} (for the averaged and hybrid algorithms). Such problems for \textit{general} function spaces are difficult to solve, and efficient algorithms will likely depend on what kinds of functions the subspace $V$ contains. 

When $V$ contains univariate polynomials, all the tasks in the algorithm can be reduced to the problem of computing roots of polynomials, and hence are feasible in principle. We accomplish this computationally by computing the spectrum of a confederate matrix, \an{although more sophisticated and practically effective methods are known}. We describe this formulation and details of the approach in Appendix \ref{app:poly-methods}.

\begin{figure}[htbp]
  \resizebox{\textwidth}{!}{
    \includegraphics[trim=11cm 1cm 7cm 1cm,clip=true,width=0.33\textwidth]{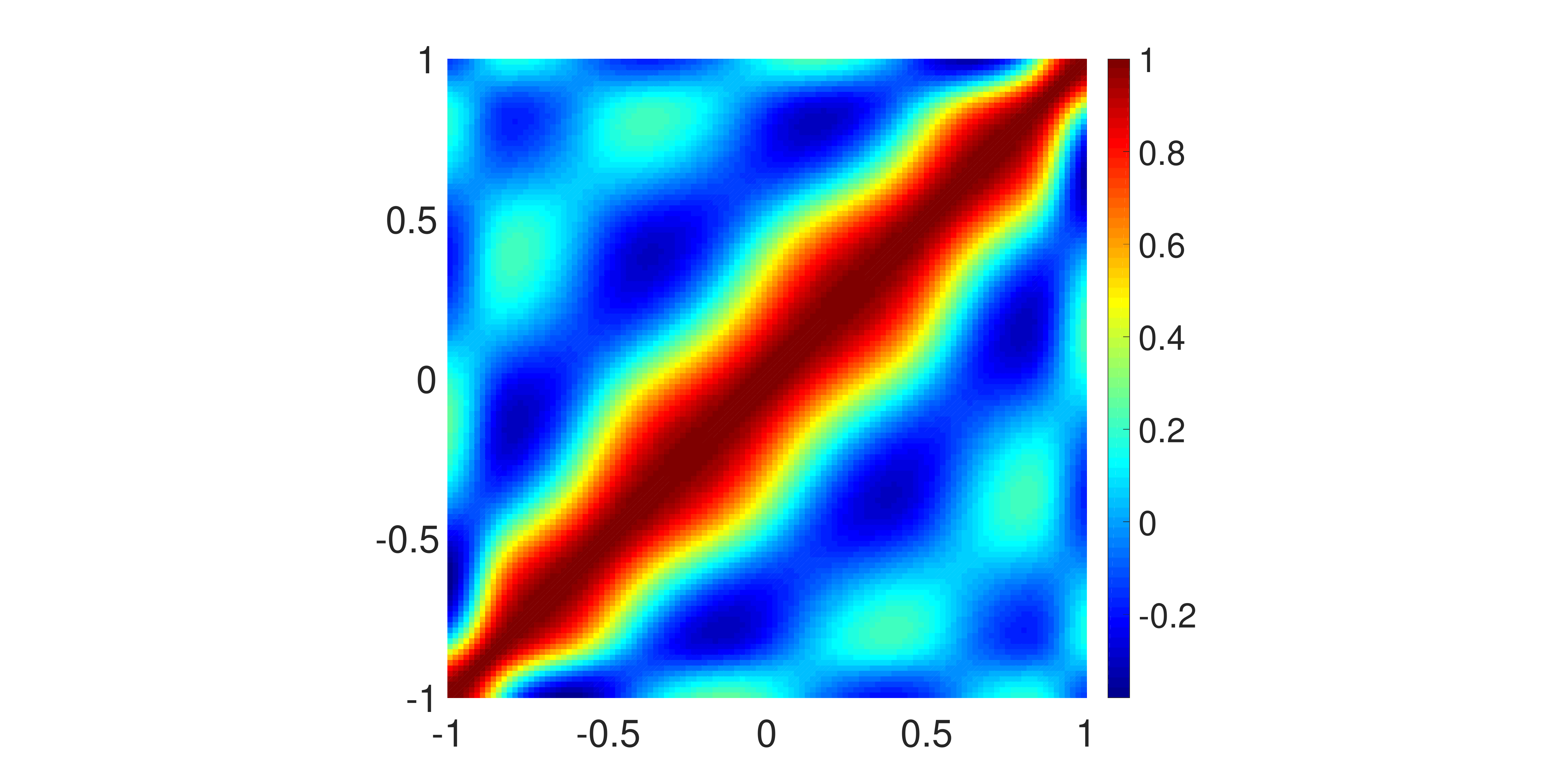}
    \includegraphics[trim=11cm 1cm 7cm 1cm,clip=true,width=0.33\textwidth]{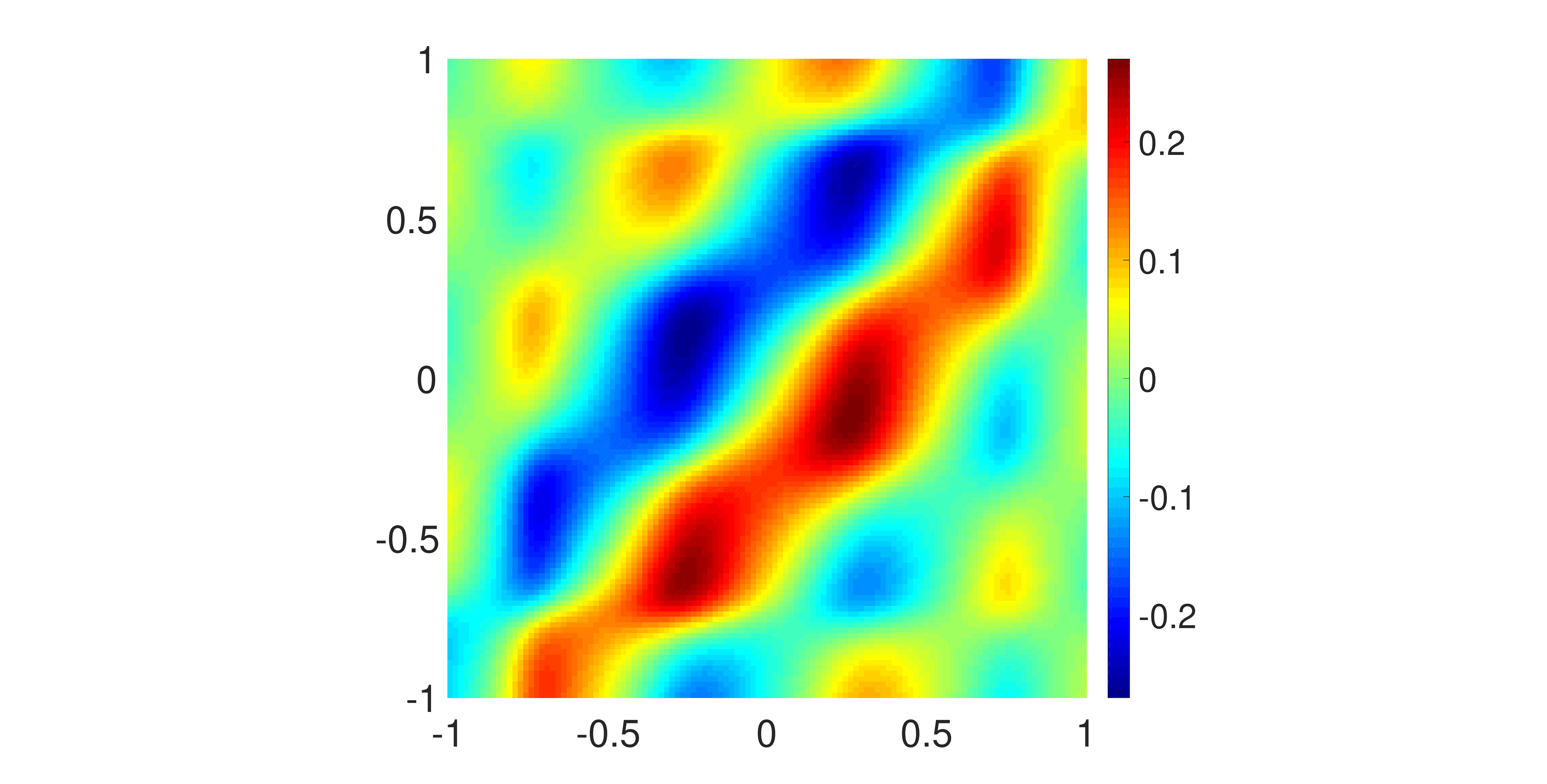}
    \includegraphics[trim=11cm 1cm 7cm 1cm,clip=true,width=0.33\textwidth]{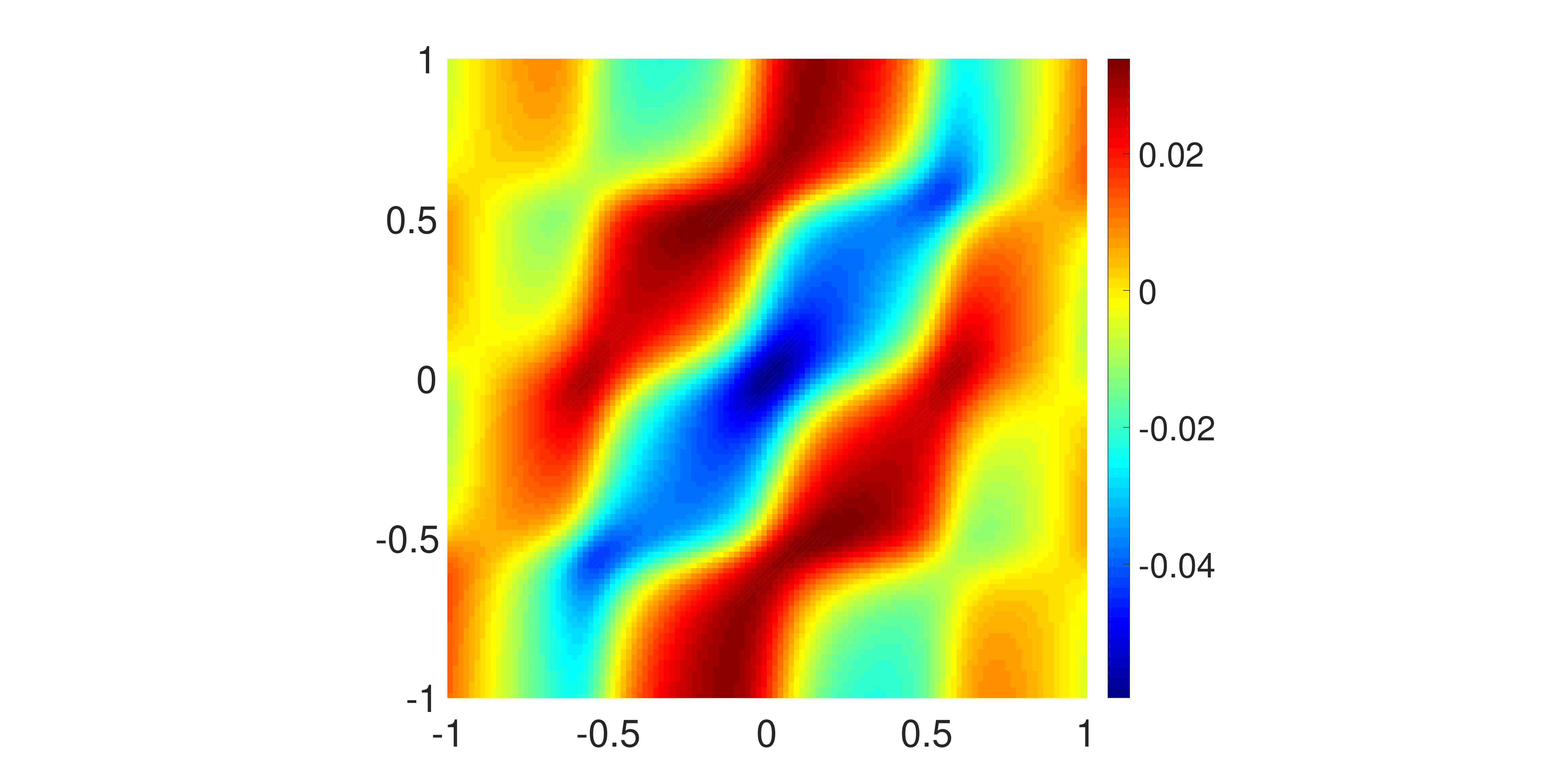}
  }\\
  \resizebox{\textwidth}{!}{
    \includegraphics[trim=11cm 1cm 7cm 1cm,clip=true,width=0.33\textwidth]{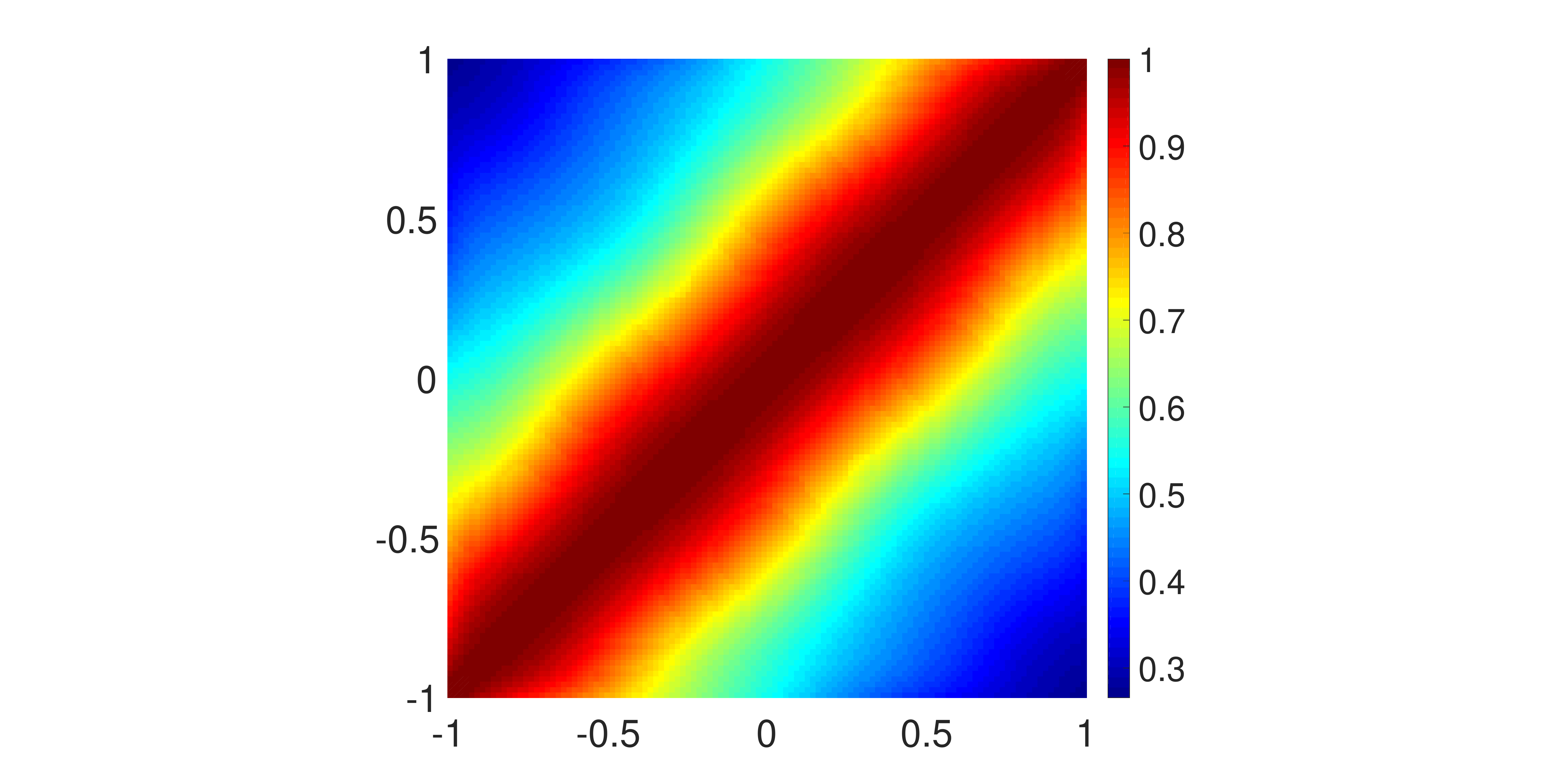}
    \includegraphics[trim=11cm 1cm 7cm 1cm,clip=true,width=0.33\textwidth]{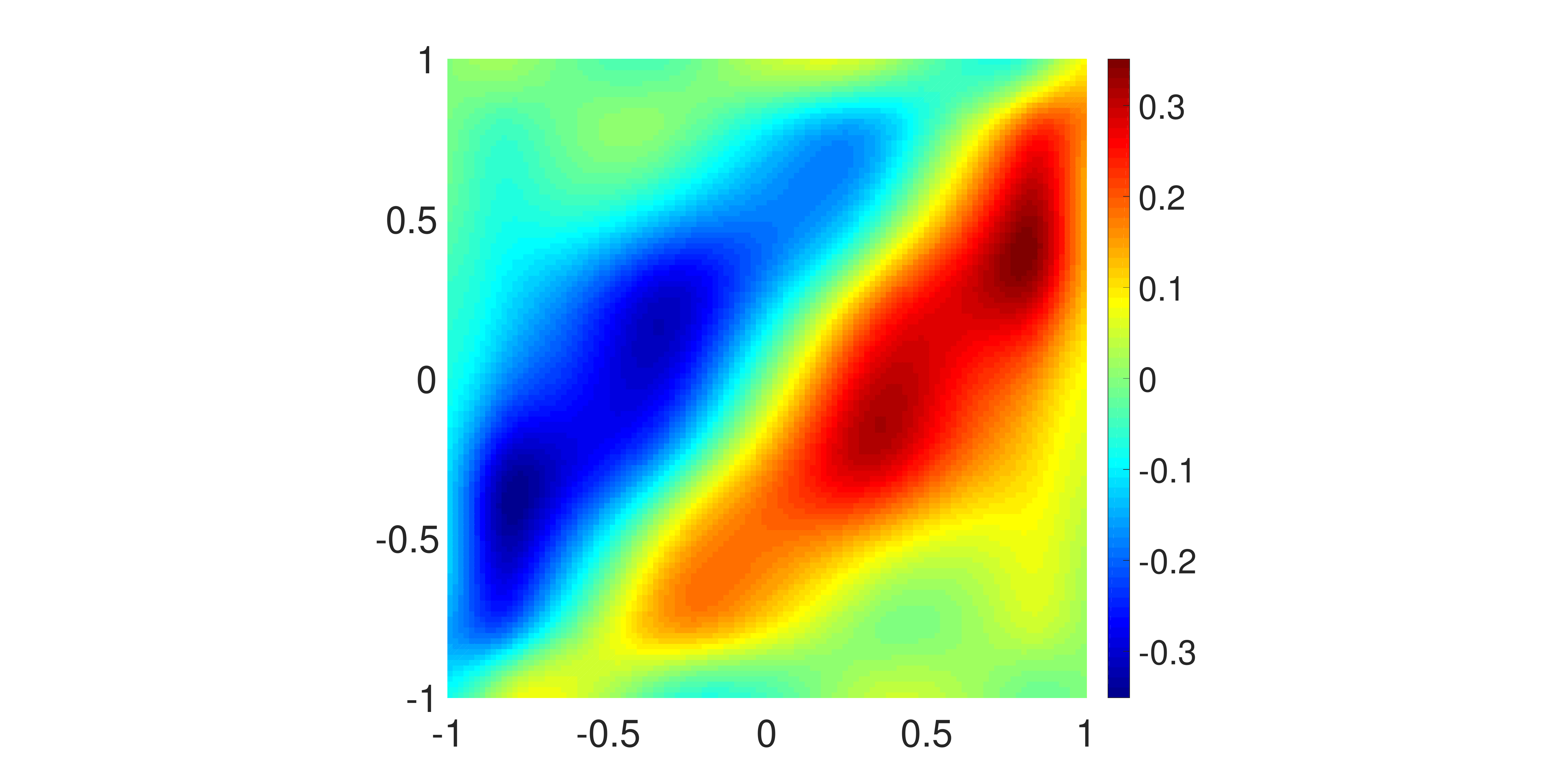}
    \includegraphics[trim=11cm 1cm 7cm 1cm,clip=true,width=0.33\textwidth]{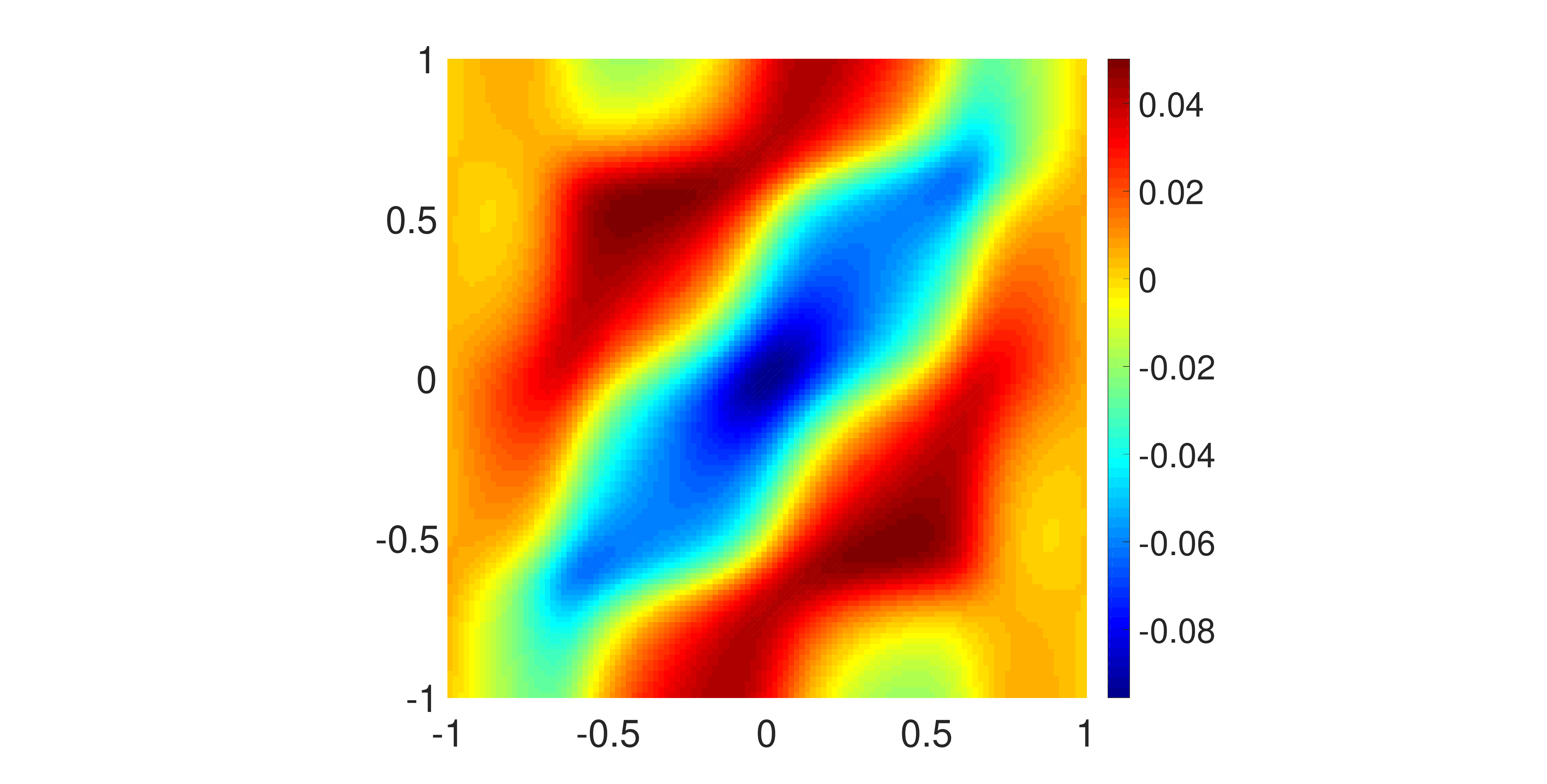}
  }\\
  \resizebox{\textwidth}{!}{
    \includegraphics[trim=11cm 1cm 7cm 1cm,clip=true,width=0.33\textwidth]{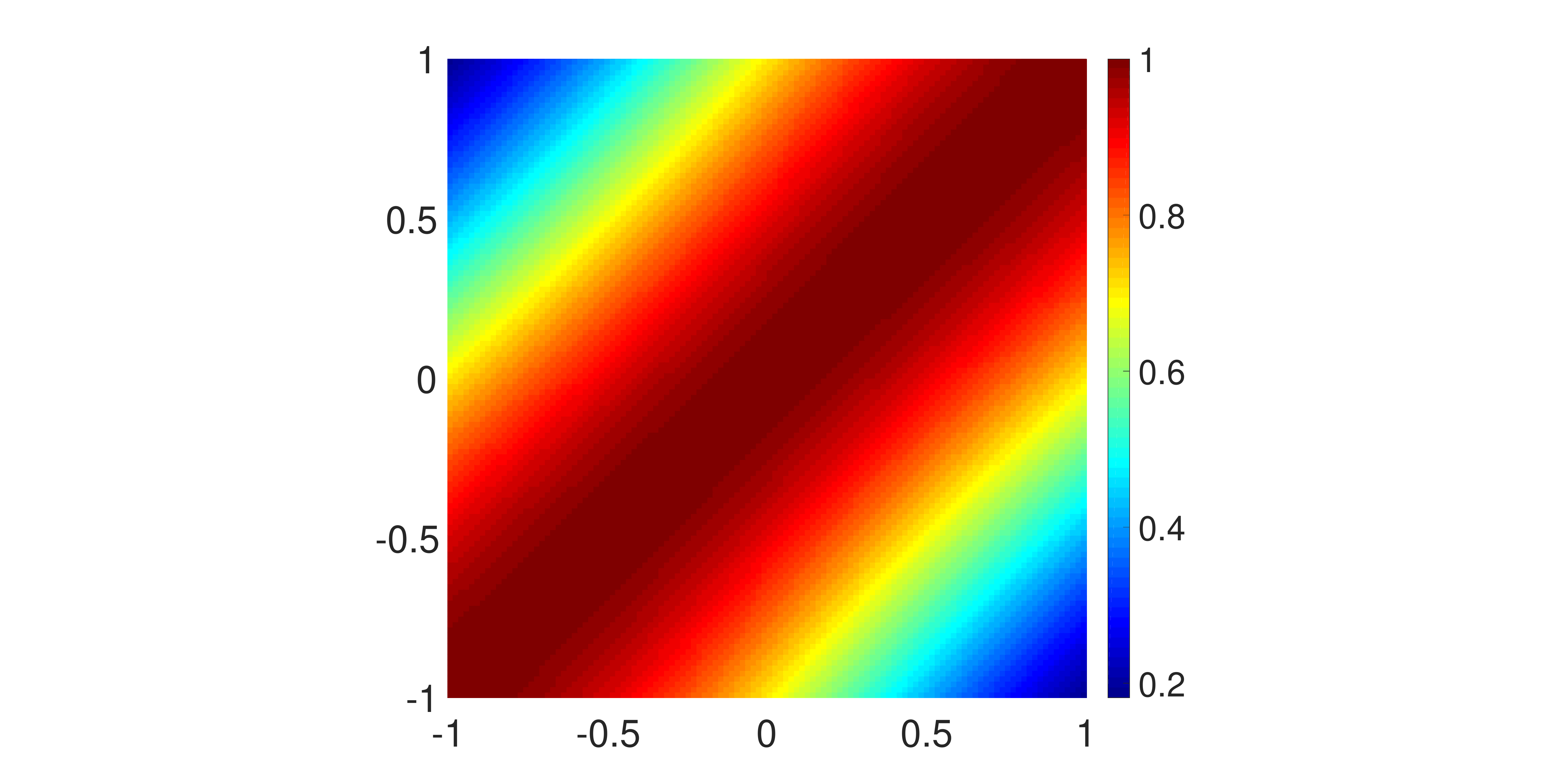}
    \includegraphics[trim=11cm 1cm 7cm 1cm,clip=true,width=0.33\textwidth]{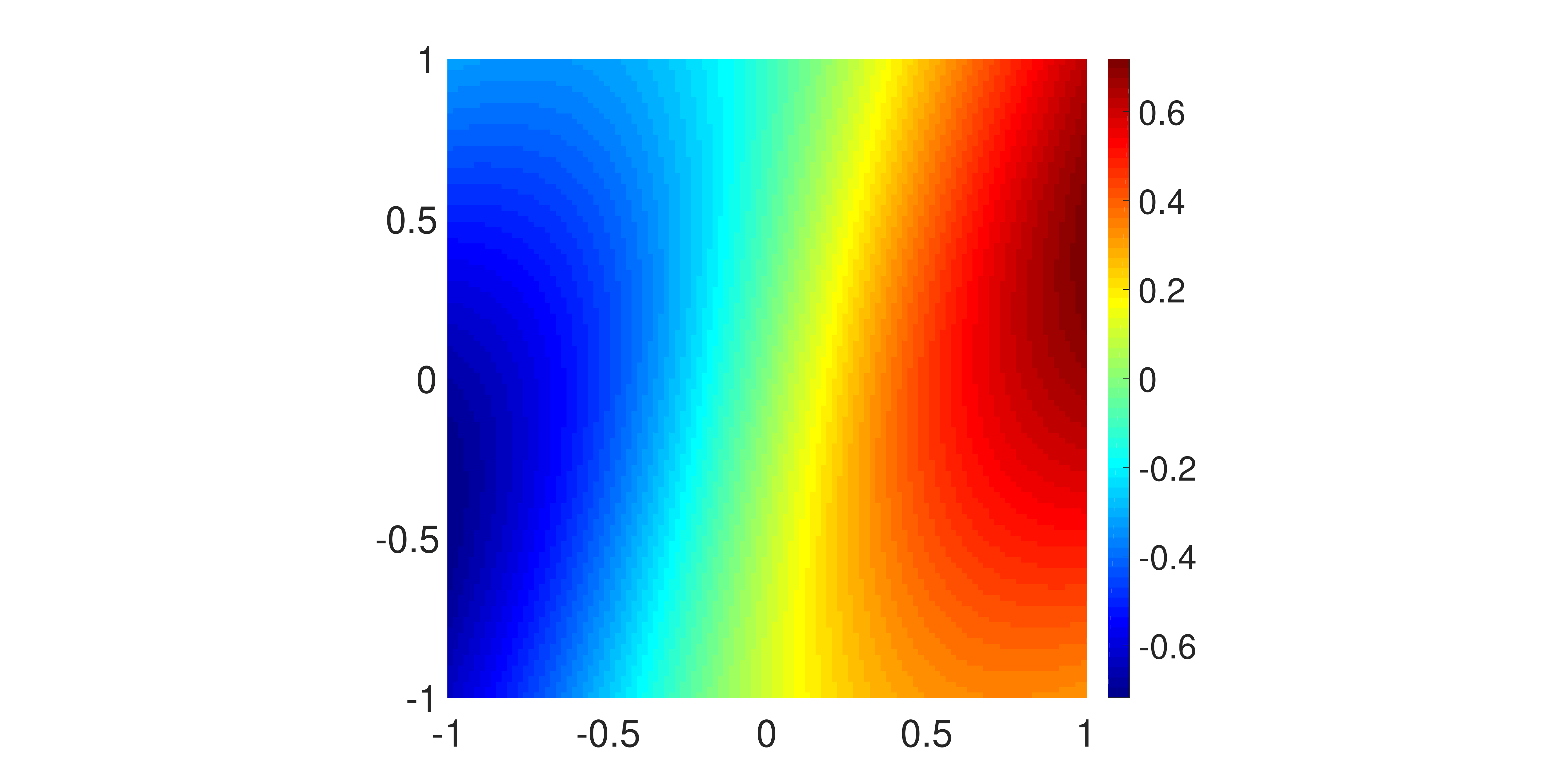}
    \includegraphics[trim=11cm 1cm 7cm 1cm,clip=true,width=0.33\textwidth]{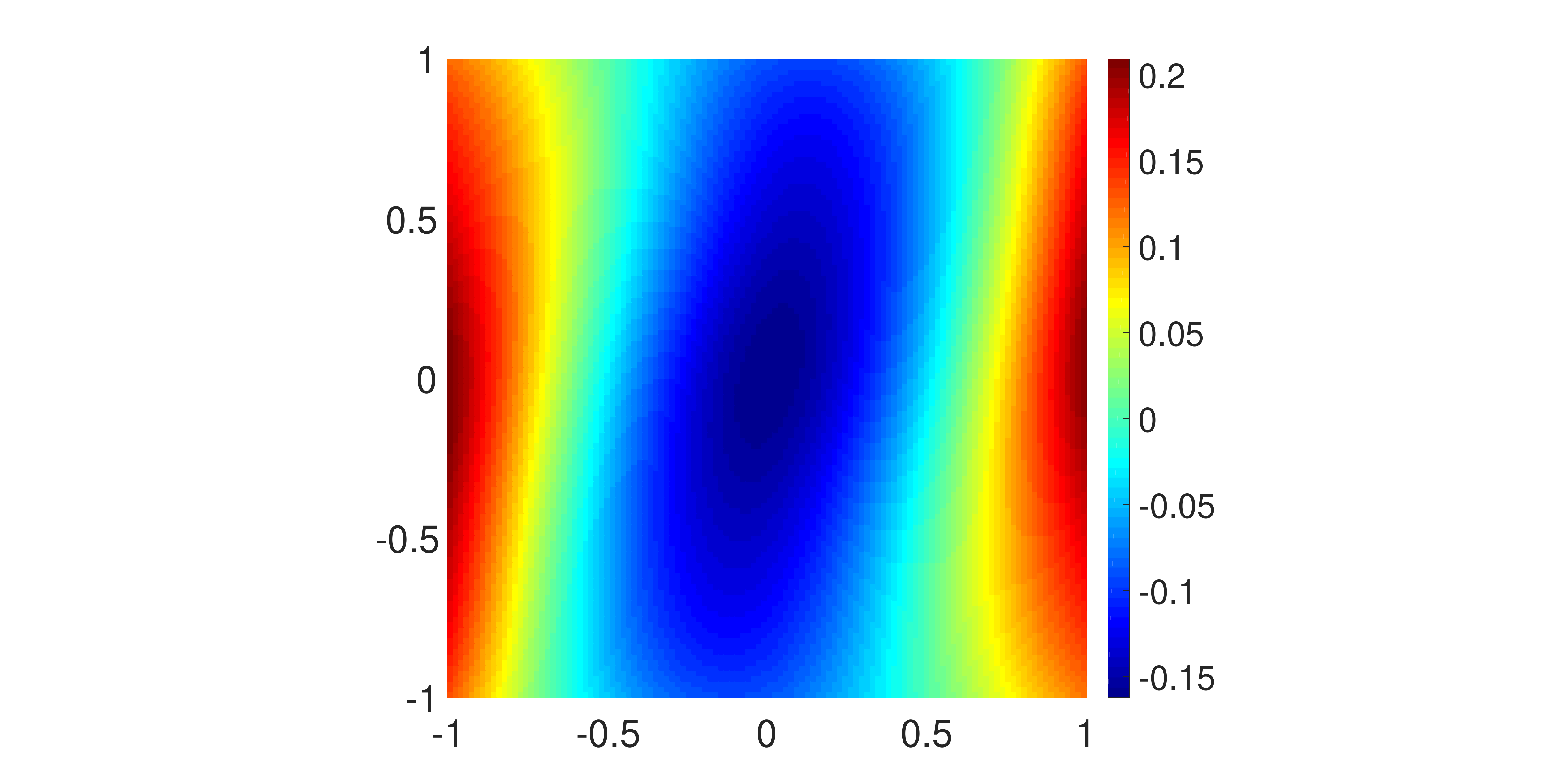}
  }
  \caption{Correction functions for degree-$5$ polynomial approximation. Plots of $\ell_k(y)(x)$ are shown as functions of $(x,y)$ for various constraints enforcing positivity of the $k$th derivative (rows) and ambient Hilbert spaces (columns). Top: $k=0$ positivity; middle: $k=1$ monotonicity; bottom: $k=2$ convexity. Left: $L^2([-1,1])$; middle; $H^1([-1,1])$; bottom: $H^2([-1,1])$.}\label{fig:surf5}
\end{figure}

\begin{figure}[htbp]
  \resizebox{\textwidth}{!}{
    \includegraphics[trim=11cm 1cm 7cm 1cm,clip=true,width=0.33\textwidth]{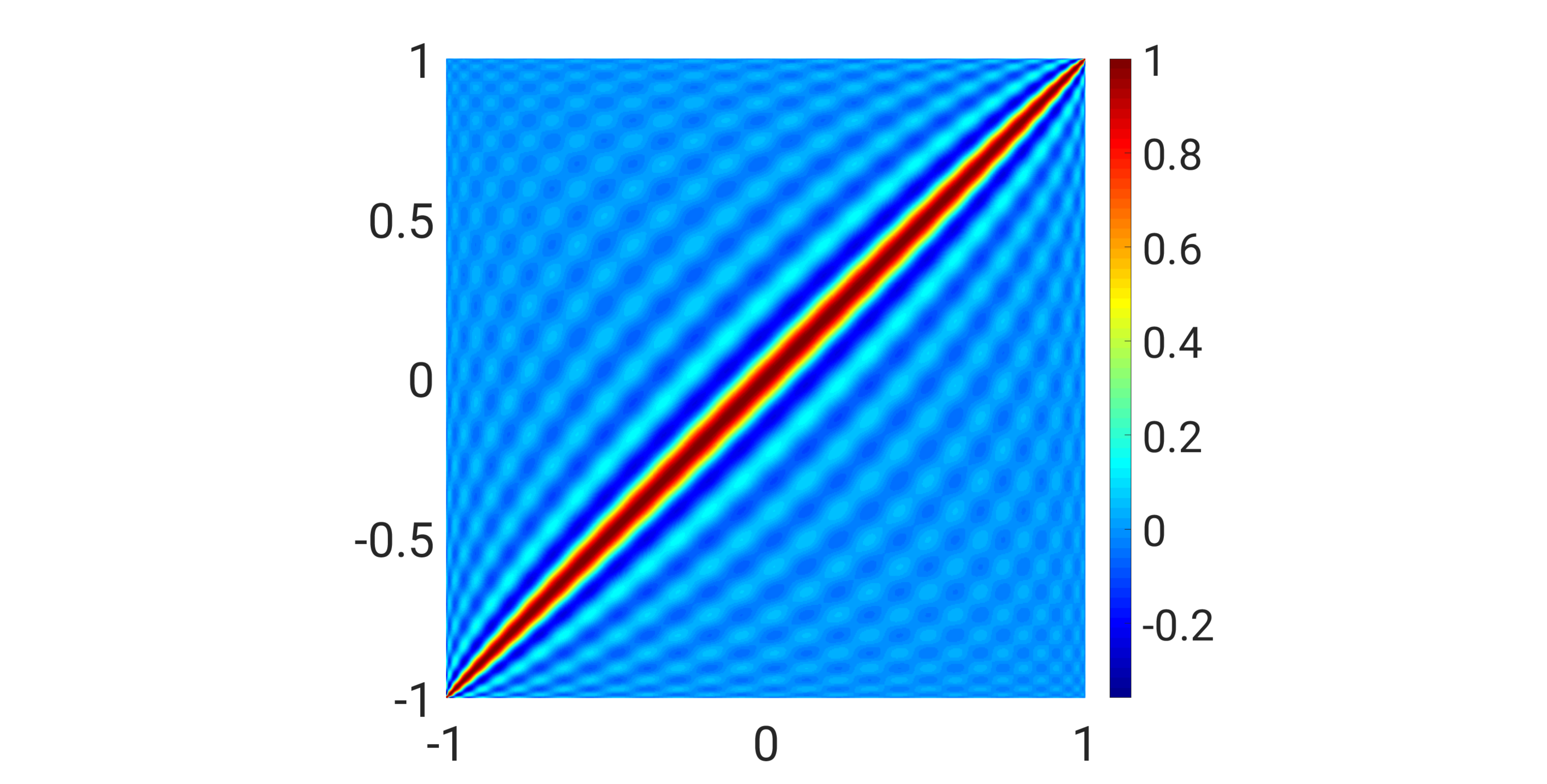}
    \includegraphics[trim=11cm 1cm 7cm 1cm,clip=true,width=0.33\textwidth]{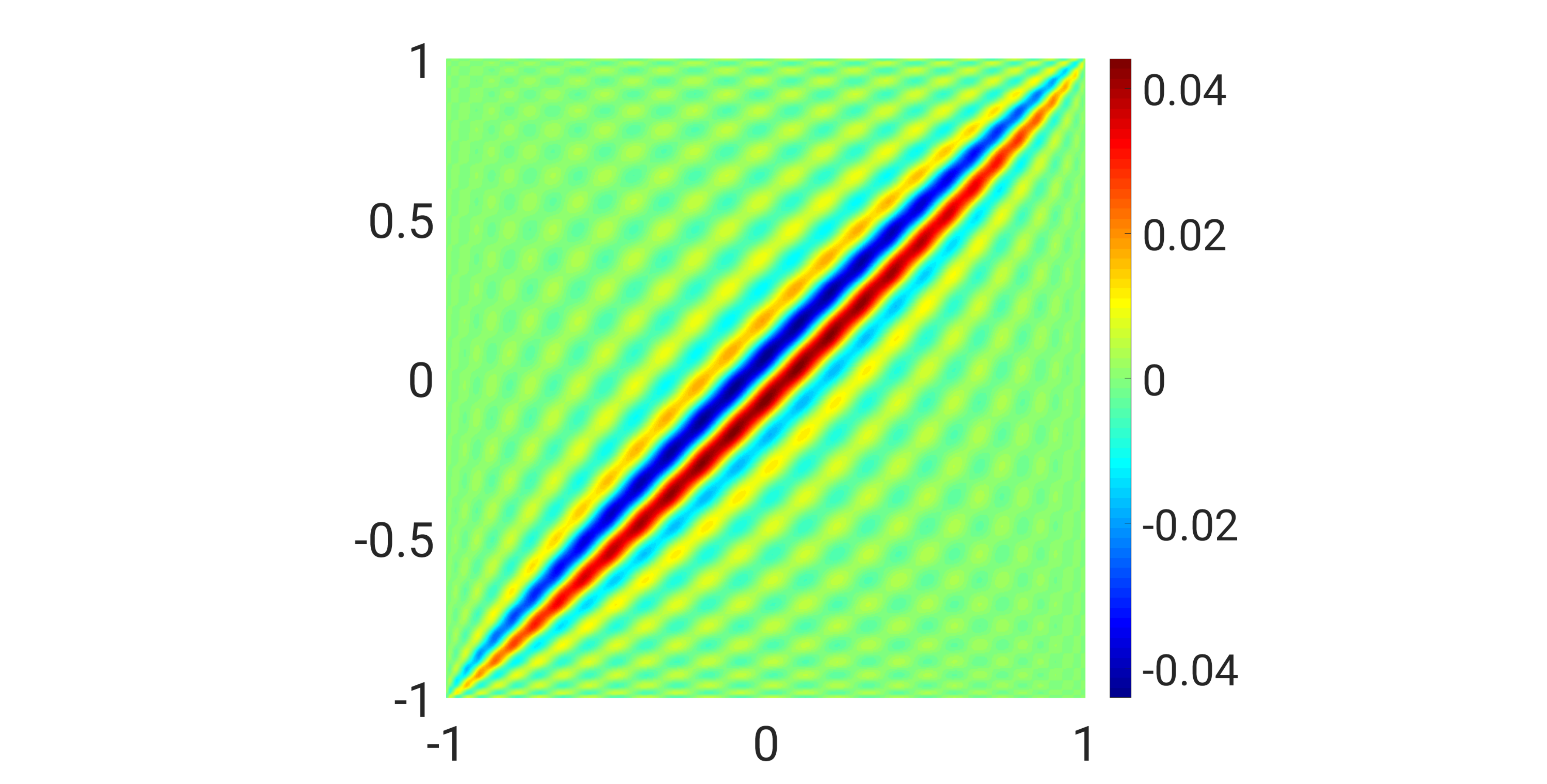}
    \includegraphics[trim=11cm 1cm 7cm 1cm,clip=true,width=0.33\textwidth]{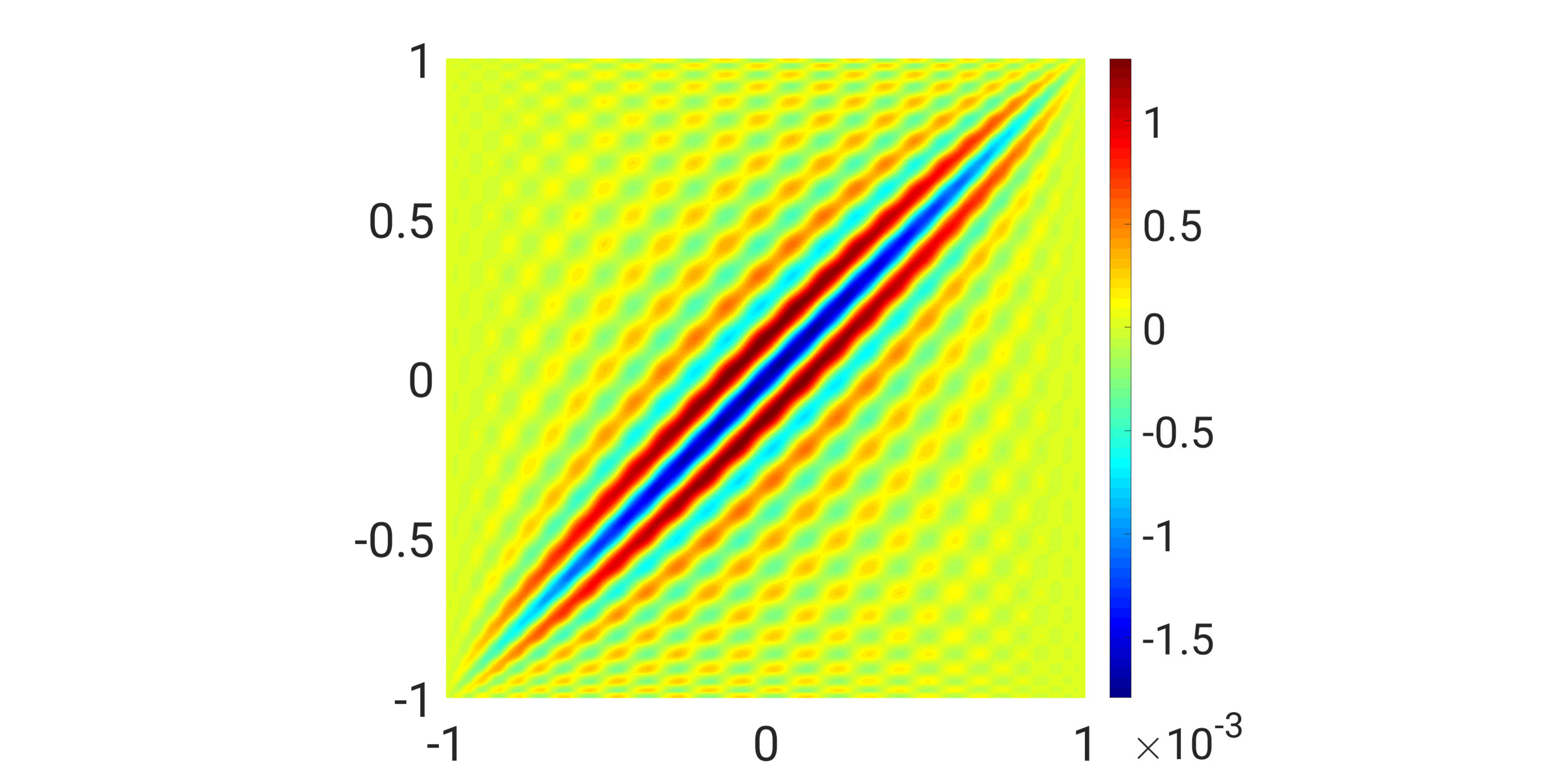}
  }\\
  \resizebox{\textwidth}{!}{
    \includegraphics[trim=11cm 1cm 7cm 1cm,clip=true,width=0.33\textwidth]{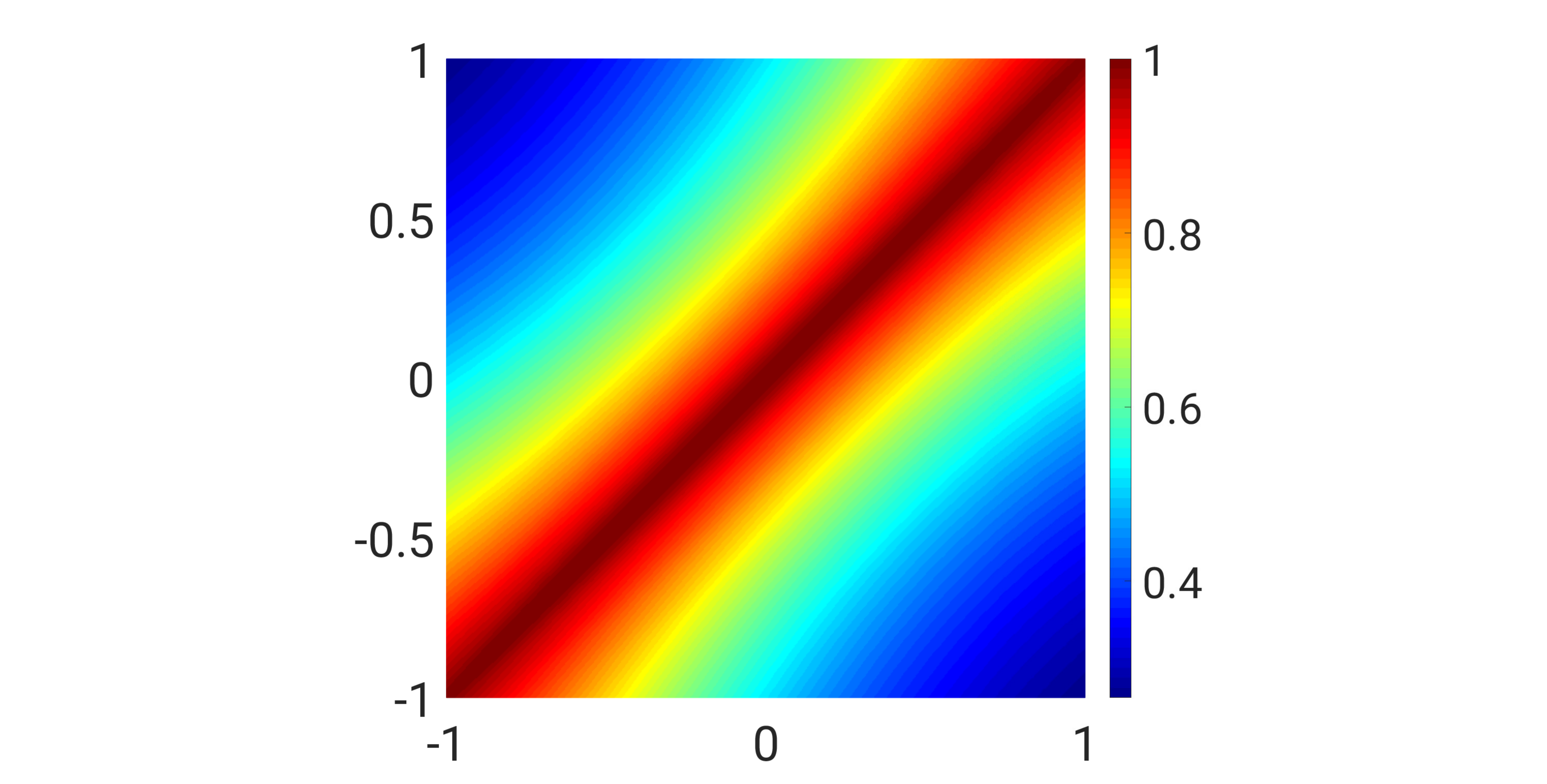}
    \includegraphics[trim=11cm 1cm 7cm 1cm,clip=true,width=0.33\textwidth]{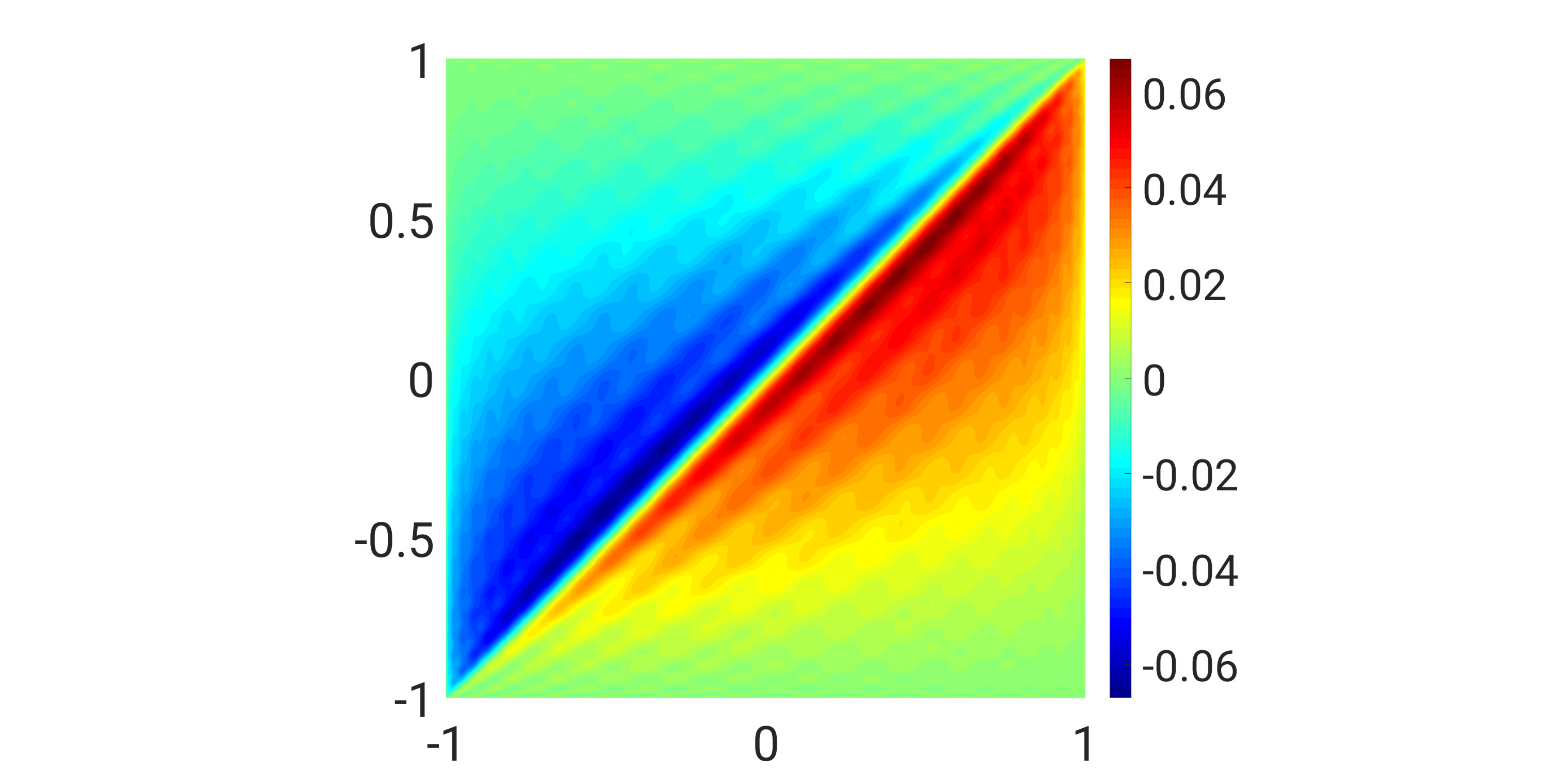}
    \includegraphics[trim=11cm 1cm 7cm 1cm,clip=true,width=0.33\textwidth]{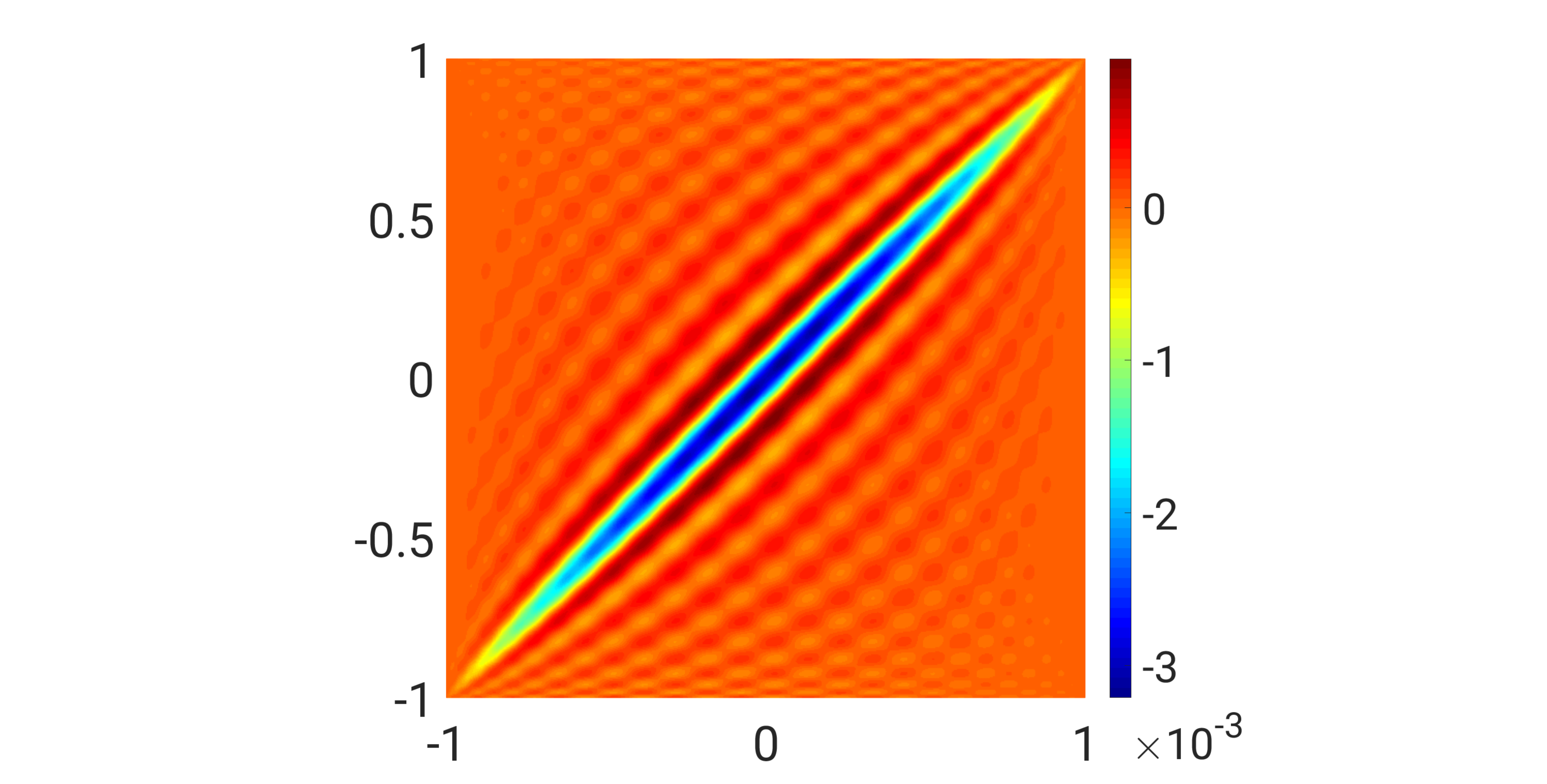}
  }\\
  \resizebox{\textwidth}{!}{
    \includegraphics[trim=11cm 1cm 7cm 1cm,clip=true,width=0.33\textwidth]{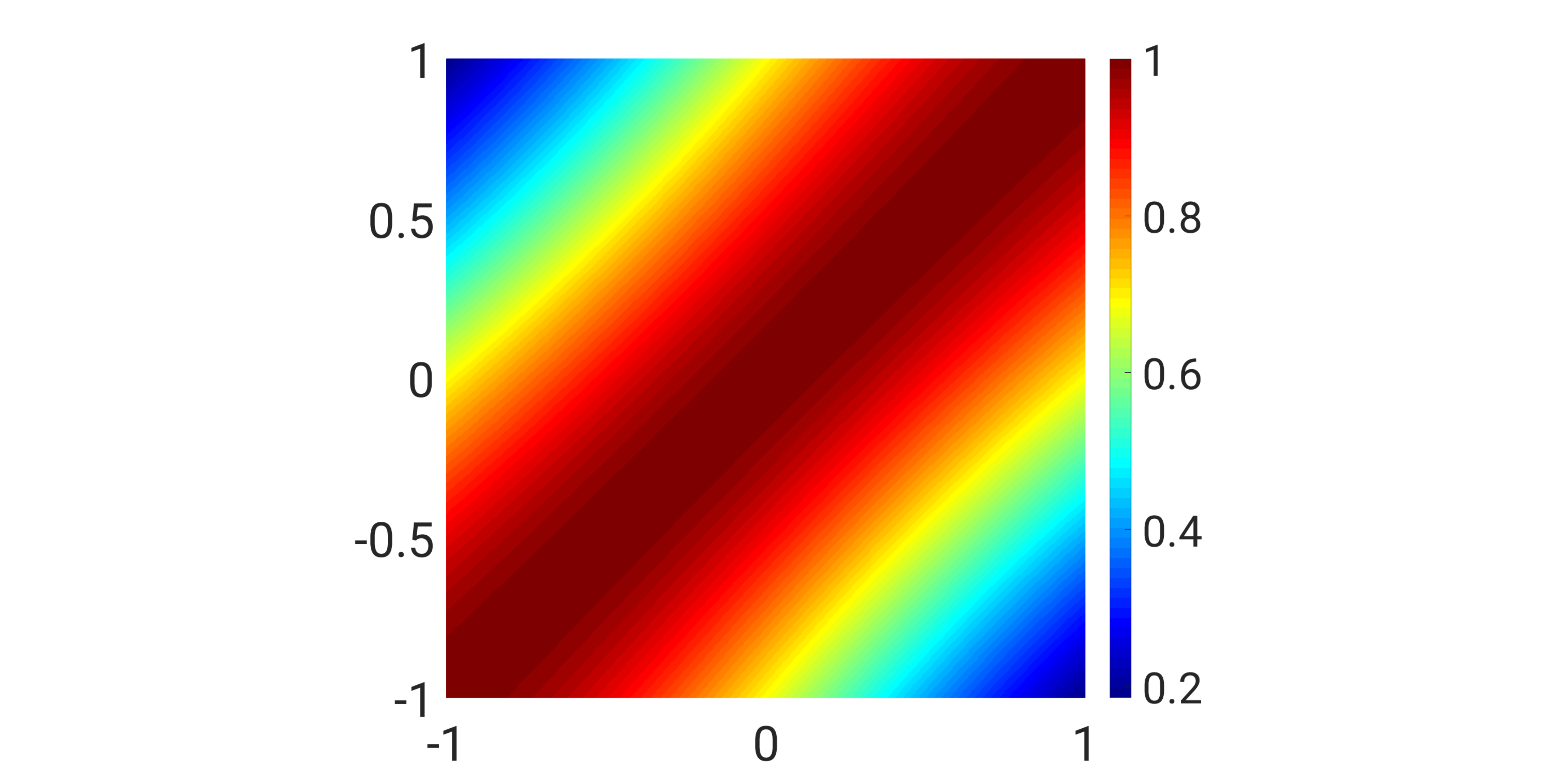}
    \includegraphics[trim=11cm 1cm 7cm 1cm,clip=true,width=0.33\textwidth]{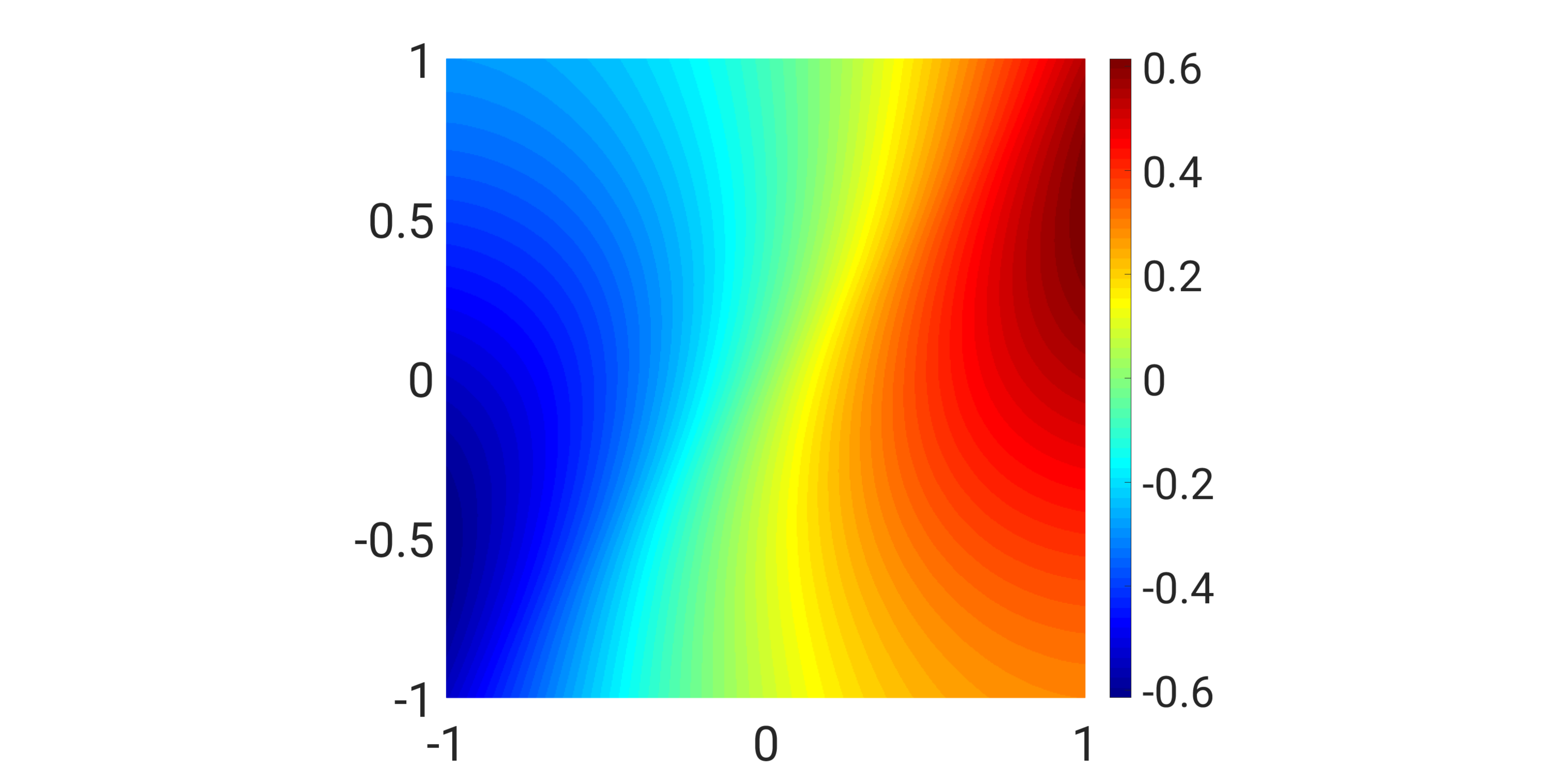}
    \includegraphics[trim=11cm 1cm 7cm 1cm,clip=true,width=0.33\textwidth]{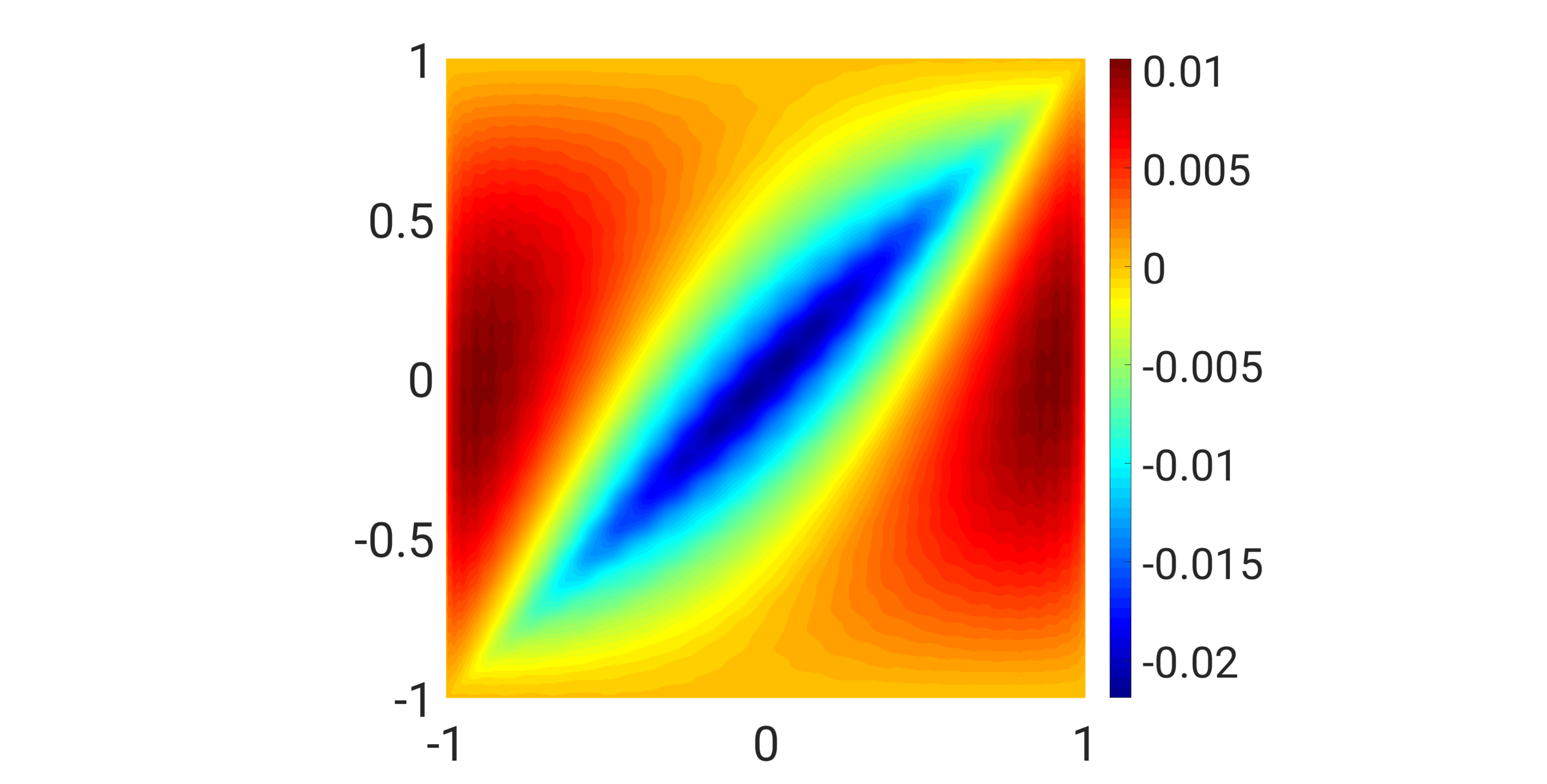}
  }
  \caption{Correction functions for degree-$30$ polynomial approximation. Plots of $\ell_k(y)(x)$ are shown as functions of $(x,y)$ for various constraints enforcing positivity of the $k$th derivative (rows) and ambient Hilbert spaces (columns). Top: $k=0$ positivity; middle: $k=1$ monotonicity; bottom: $k=2$ convexity. Left: $L^2([-1,1])$; middle; $H^1([-1,1])$; right: $H^2([-1,1])$.}\label{fig:surf30}
\end{figure}

\subsection{Nonidentity matrices $\bs{A}$}\label{ssec:AnotI}
The optimization problem we seek to solve is \eqref{eq:constopt-discrete}; the algorithms in this section have proceeded under the assumption that $\bs{A} = \bs{I}$. When this is not the case, we must first solve \eqref{eq:transformed-opt}, so that the full solution is \eqref{eq:transformed-opt-full}. Thus, we focus on the problem
\begin{align}\label{eq:mapped-optimization}
  \argmin_{\bs{z} \in \bs{\Sigma} \bs{V}^\ast C} \left\| \bs{z} - \bs{U}^\ast \bs{b} \right\|_2.
\end{align}
Note that the only difference between this optimization and the simplified version \eqref{eq:reduced-discrete} is that the feasible set is $\bs{\Sigma}\bs{V}^\ast C$ instead of $C$ so that we need only address the presence of the linear map $\bs{\Sigma} \bs{V}^\ast$. Since $C$ is closed and convex, then $\bs{\Sigma}\bs{V}^\ast C$ is also closed and convex, and in particular is defined as the intersection of closed, conic, convex sets $\widetilde{C}_k$:
\begin{align*}
  \bs{\Sigma}\bs{V}^\ast C \eqqcolon \widetilde{C} = \bigcap_{k=1}^K \widetilde{C}_k \coloneqq \bigcap_{k=1}^K \bs{\Sigma}\bs{V}^\ast \bs{C}_k.
\end{align*}
Thus, all our previous algorithms apply, except that we need to only transform $(L_k, r_k, \omega_k)$ for $C_k$ into the appropriate quantities for $\widetilde{C}_k$. These transformations are straightforward but technical, so we omit showing them explicitly.


\section{Numerical results}
\label{sec:results}

In all that follows, $f$ is a given function in a Hilbert space $H$. Given a finite-dimensional space $V \subset H$, the function $v$ is the $H$-best projection onto $V$, which does not in general satisfy any structural constraints. (Note from discussion in Section \ref{ssec:AnotI} that extensions to, e.g., collocation-based approximations, are straightforward.) The function $\tilde{v}$ is the output of the constrained optimization procedure. 

With the univariate Sobolev spaces,
\begin{align*} 
  H^q([-1,1]) &\coloneqq \left\{ f: [-1,1] \rightarrow \R \; \big|\; \|f\|_{H^2} < \infty  \right\}, & \|f\|_{H^q}^2 \coloneqq \sum_{j=0}^q \int_{-1}^1 \left[ f^{(j)}(x) \right]^2 \dx{x},
\end{align*}
our examples will consider the ambient Hilbert space $H$ as $H^0 (= L^2)$, $H^1$, or $H^2$. The subspace $V$ in all our experiments is the space of polynomials up to degree $N-1$:
\begin{align*}
  V = \left\{ p: [-1,1] \rightarrow \R \; \big|\; \deg p \leq N \right\}.
\end{align*}
Our test functions $f_j$ are defined iteratively for $j \geq 1$ as,
\begin{align*}
  f_{j+1}(x) &= c_{j+1} \int_{-1}^x f_j(t) \dx{t}, & f_0(x) = \left\{\begin{array}{rl} 0, & x \leq 0, \\ 1, & x > 0 \end{array}\right.,
\end{align*}
where $c_{j+1}$ are normalization constants chosen so that $f_{j+1}(1) = 1$. Thus, $f_j$ has $j$ weak $L^2$ derivatives. Finally, most of our results will consider intersections of the following four types of constraint sets in $V$:
\begin{itemize}
  \item (Positivity) $F_0 \coloneqq \{ f \in H \; \big|\; f(x) \geq 0 \;\forall\, x \in [-1,1] \}$
  \item (Boundedness) $G_0 \coloneqq \{ f \in H \; \big|\; f(x) \leq 1 \;\forall\, x \in [-1,1] \}$
  \item (Monotonicity) $F_1 \coloneqq \{ f \in H \; \big|\; f'(x) \geq 0 \;\forall\, x \in [-1,1] \}$
  \item (Convexity) $F_2 \coloneqq \{ f \in H \; \big|\; f''(x) \geq 0 \;\forall\, x \in [-1,1] \}$
\end{itemize}
Our final example considers a slightly more exotic set of constraints, which we discuss later.

In order to understand how much our algorithms ``change" the input $v$ when producing constrained approximation $\tilde{v}$, we measure the following quantity:
\begin{equation}\label{eq:err}
  \eta \coloneqq \frac{\| v - \tilde{v}\|_H}{\|f - v\|_H}.
\end{equation}
Since $f - v$ is $H$-orthogonal to $V$, then 
\begin{align*}
  \left\| f - \tilde{v} \right\|_H^2 = (1 + \eta^2) \| f - v\|^2_H.
\end{align*}
Thus, $\sqrt{1 + \eta^2}$ measures the error in the constrained approximation relative to the (best) unconstrained approximation. Values on the order of 1 imply that this optimization problem commits an additional error that is approximately the same as the error committed by the best (unconstrained) approximation.

Algorithm \ref{alg:greedy} is the greedy algorithm, but it is the template for the averaging and hybrid algorithms as well. For example, a hybrid algorithm needs to replace only line \ref{alg:greedy:update} in that algorithm by the update \eqref{eq:c-update-averaged}. However, we have left some details of the termination criterion in line \ref{alg:greedy:termination} unexplained. For example, we do not actually enforce $\mathrm{sdist}(\bs{c}, H_{k^\ast}(y^\ast)) \leq 0$ as stated due to finite precision. Instead, we enforce
\begin{align}\label{eq:tolerance}
  \mathrm{sdist}(\bs{c}, H_{k^\ast}(y^\ast)) &\leq \delta, & \delta &> 0,
\end{align}
where we set $\delta = 10^{-10}$ and have implemented the procedure in double precision. In addition, the number of iterations $I$ required before termination will also be reported.

\subsection{Algorithm comparison}
A short summary of all the experiments investigating the hybrid approaches and their comparison with the greedy and the averaging methods is given in the \cref{tab:sum}.

\begin{table}[thbp]
\vspace{-2mm}
\centering
 \resizebox{0.6\textwidth}{!}{
 \begin{tabular}{l r r l l r r l l} 
 & \multicolumn{4}{c}{$N = 6$} & \multicolumn{4}{c}{$N = 31$} \\ \cmidrule(l){2-5}\cmidrule(l){6-9}
 & \multicolumn{2}{c}{$I$}  & \multicolumn{2}{c}{$\eta$} & \multicolumn{2}{c}{$I$}  & \multicolumn{2}{c}{$\eta$}  \\ \cmidrule(l){2-3} \cmidrule(l){4-5} \cmidrule(l){6-7} \cmidrule(l){8-9}
   \multicolumn{1}{c}{$\epsilon$} & 
   \multicolumn{1}{c}{$10^{-3}$} & 
   \multicolumn{1}{c}{$10^{-5}$} & 
   \multicolumn{1}{c}{$10^{-3}$} & 
   \multicolumn{1}{c}{$10^{-5}$} & 
   \multicolumn{1}{c}{$10^{-3}$} & 
   \multicolumn{1}{c}{$10^{-5}$} & 
   \multicolumn{1}{c}{$10^{-3}$} & 
   \multicolumn{1}{c}{$10^{-5}$} \\
   \cmidrule(l){1-1} \cmidrule(l){2-5} \cmidrule(l){6-9}
 Greedy & 20 & 20&         1.147&         1.147& 23 & 23&         0.986&         0.986\\ 
 Averaging & 36 &36&       1.148&       1.148& 383&  383&    0.985&    0.985\\
 Hybrid & 4 & 16&1.1464& 1.148& 2 &3& 1.142&  1.054\\ [1ex] 
 \hline
 \end{tabular}
  }
  \caption{Performance summary of three proposed algorithms on the test function $f = f_2$ for different values of $\epsilon$, where $\epsilon$ is as described in Section 4.3. The constraint set is $E = F_0$.}\label{tab:sum}
\end{table}

\subsection{Function approximation examples}
We present two examples of function approximation to preserve structure in this section. The first example takes $H = H^0$ and the test function $f = f_0$, which is a step (discontinuous) function. We present results for different $N$ (the dimension of $V$) and different constraints. Figure \ref{fig:constrained-step} illustrates the results of the greedy algorithm. We compare medium-degree polynomial approximation $N = 6$ with high-degree polynomial approximation $N = 31$. The three kinds of constraints are (a) positivity, (b) positivity and boundedness, and (c) positivity, boundedness, and monotonicity. We observe that both the positivity and monotonicity constraints accomplish what is desired: the approximation $\tilde{v}$ satisfies the desired constraints, but still features Gibbs'-type oscillations. However, enforcing monotonicity as well results in a nonoscillatory approximation. All computed values of $\eta < 1$ show that the constrained approximation commits an error that is comparable to that of the $H$-best approximation.

\begin{figure}[H]
  \begin{center}
    \includegraphics[width=0.32\textwidth]{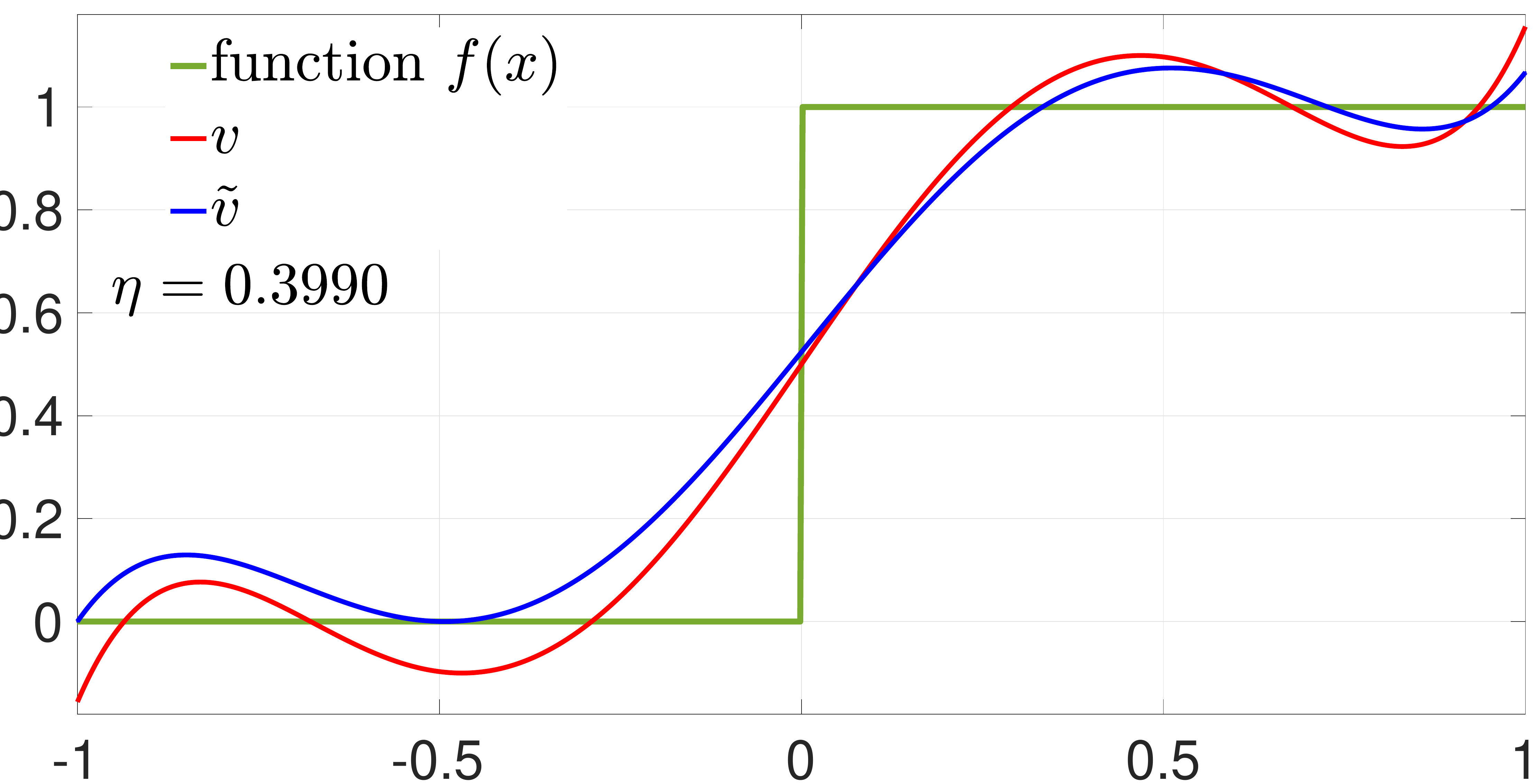}
    \includegraphics[width=0.32\textwidth]{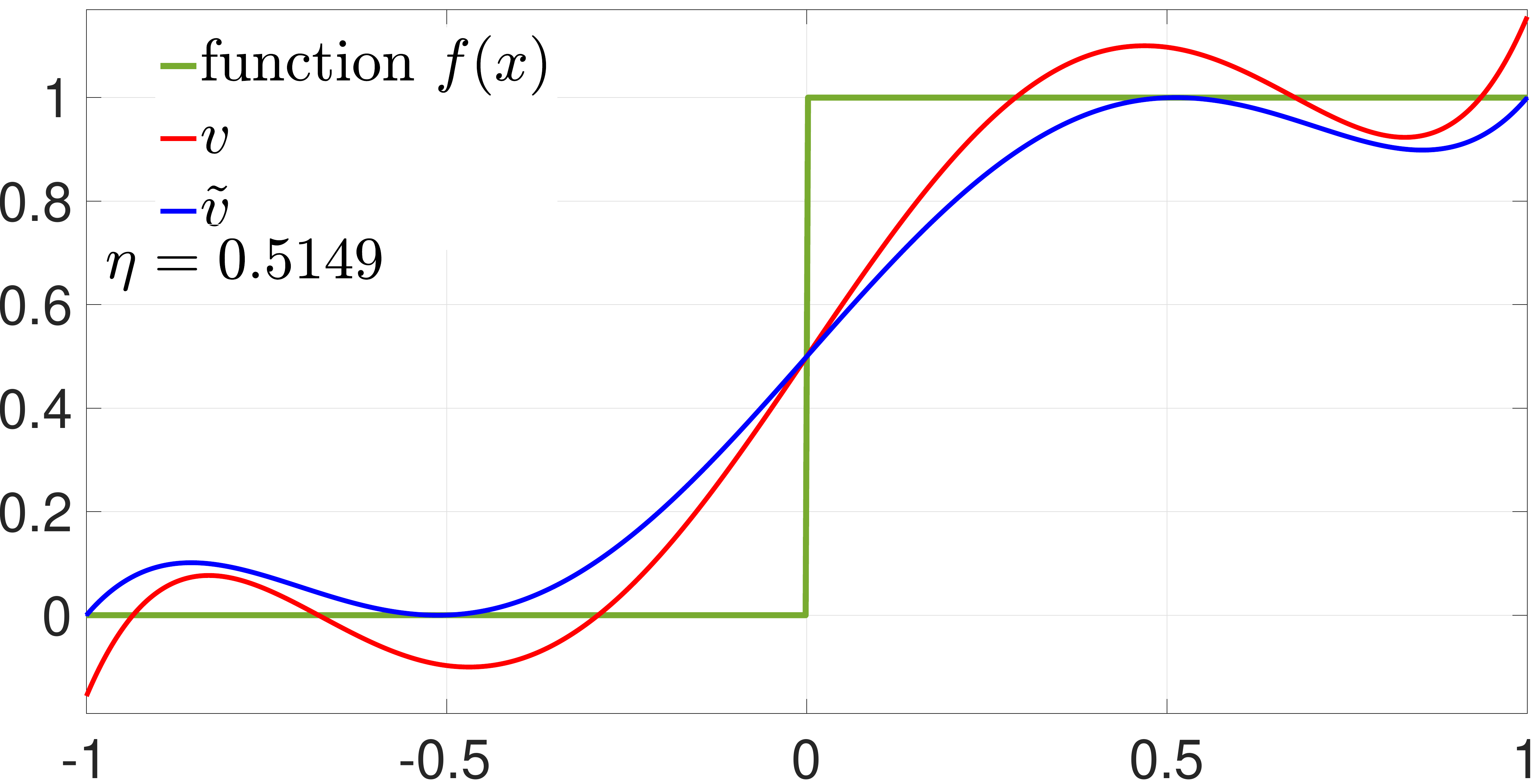}
    \includegraphics[width=0.32\textwidth]{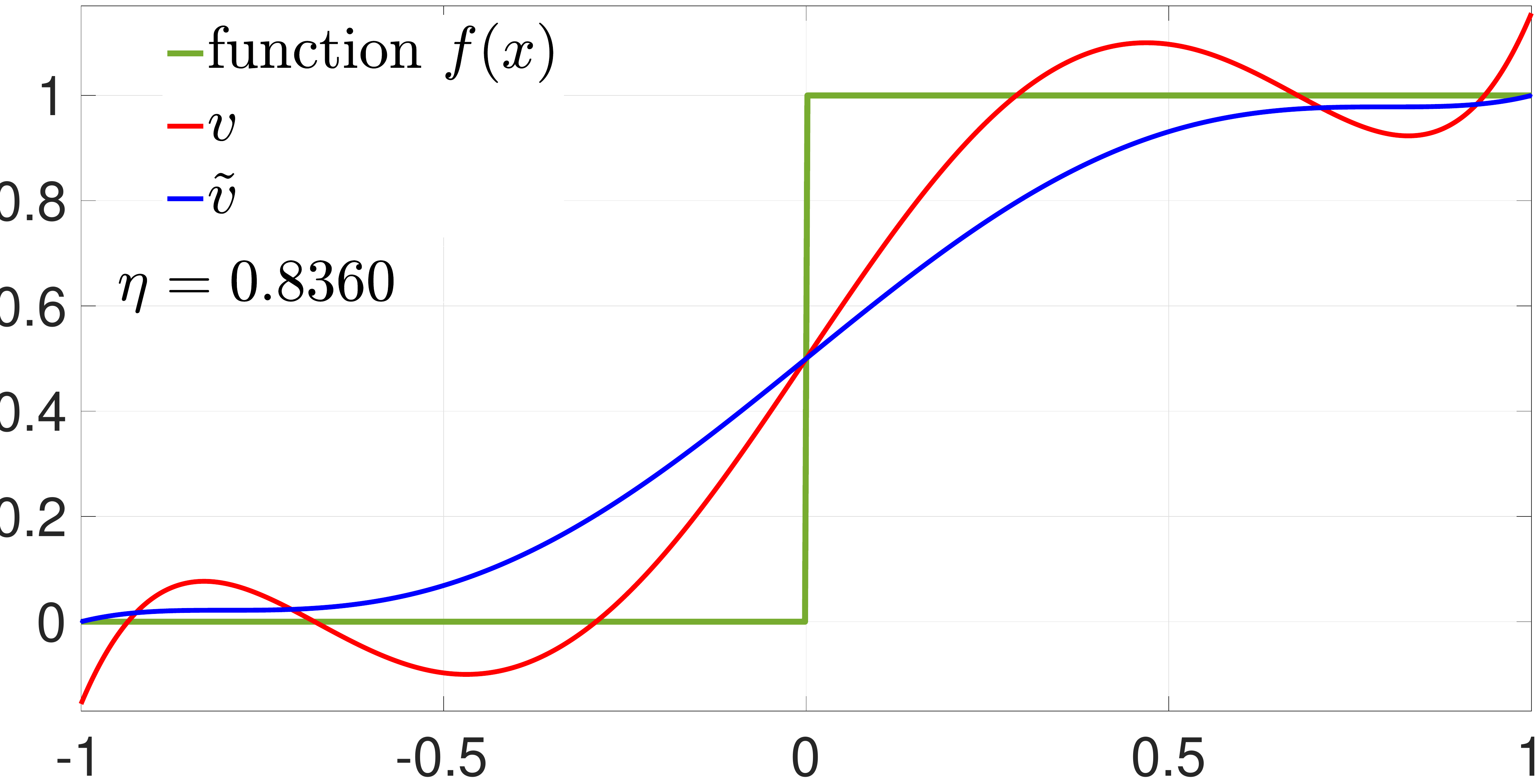} \\
    \includegraphics[width=0.32\textwidth]{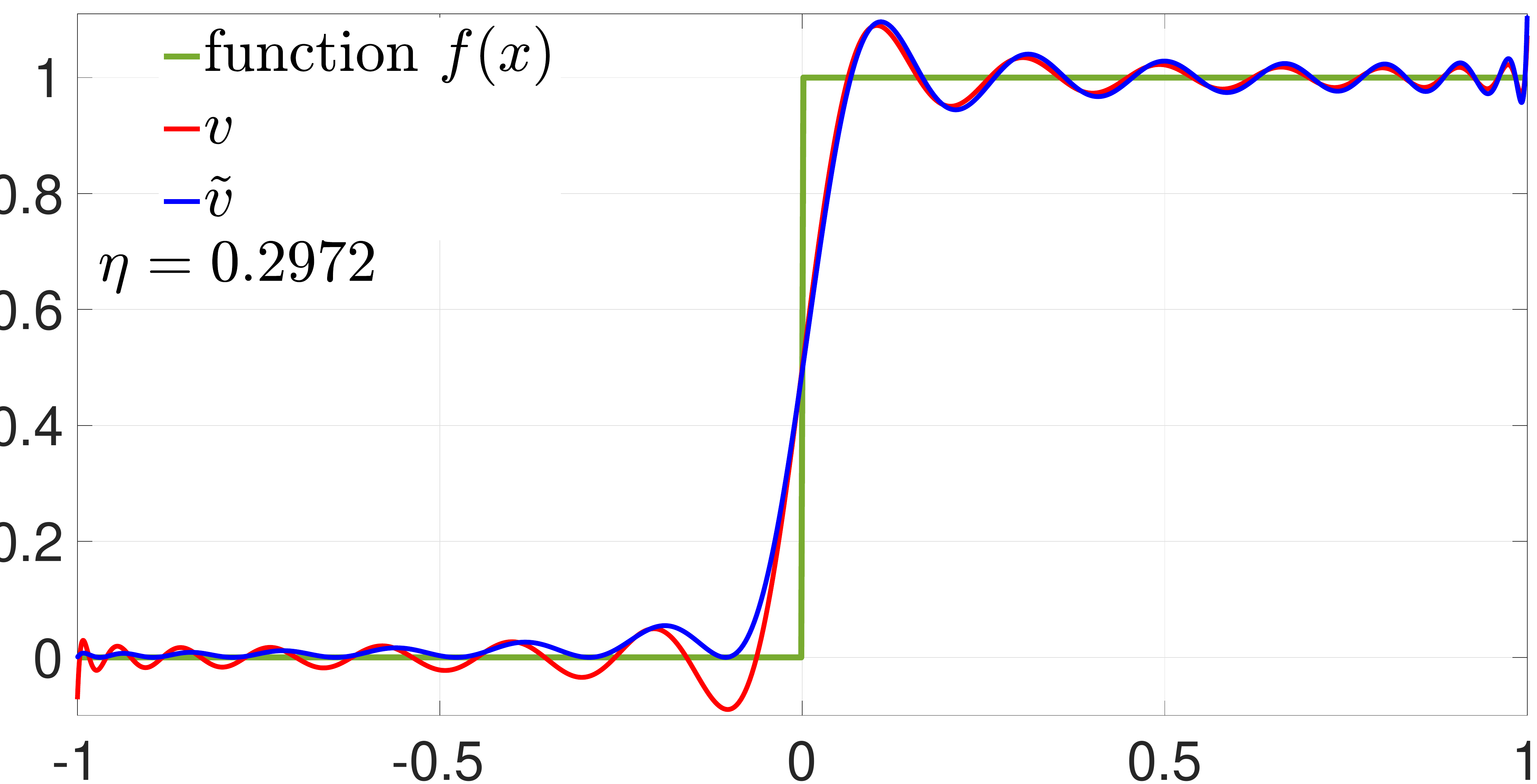}
    \includegraphics[width=0.32\textwidth]{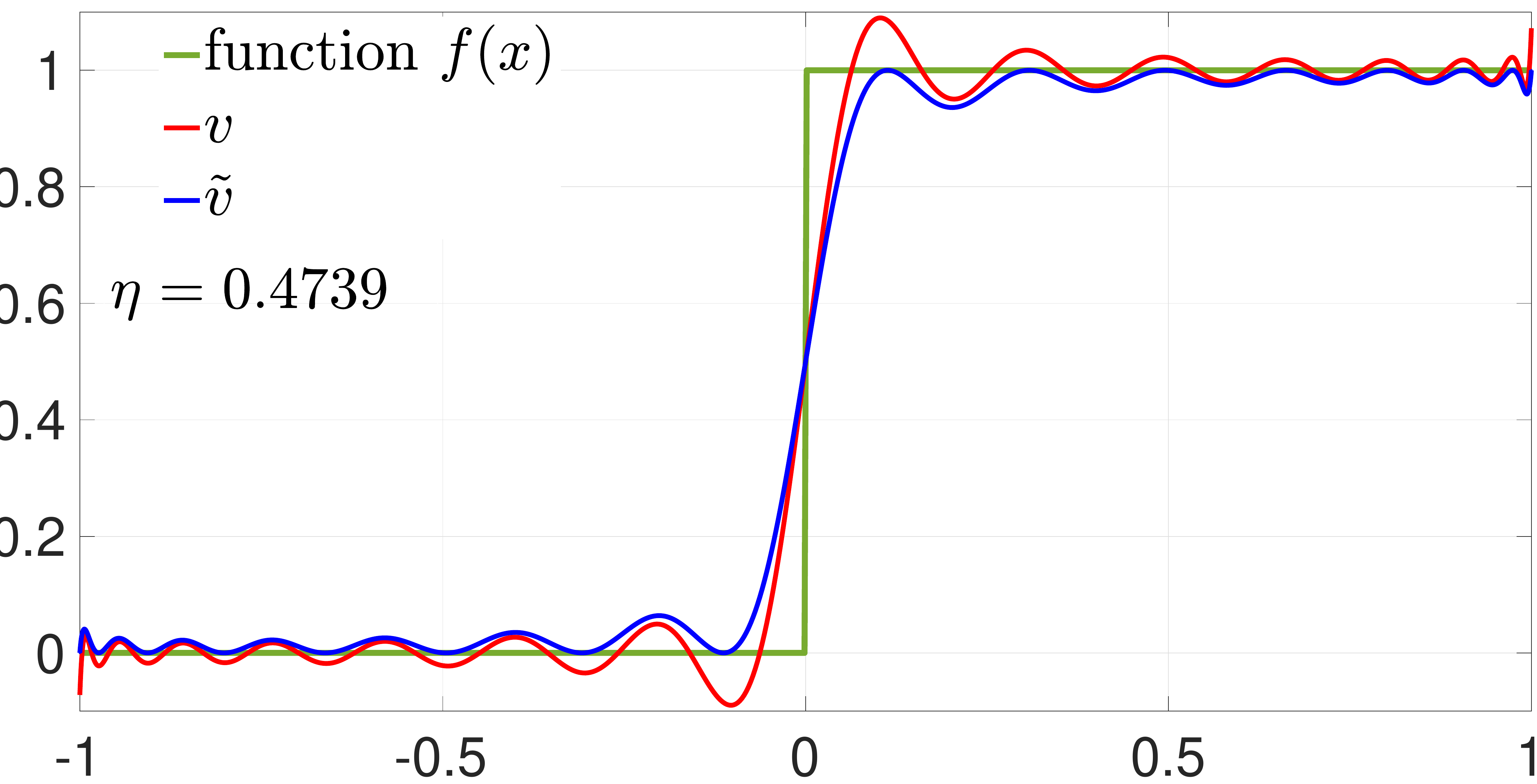}
    \includegraphics[width=0.32\textwidth]{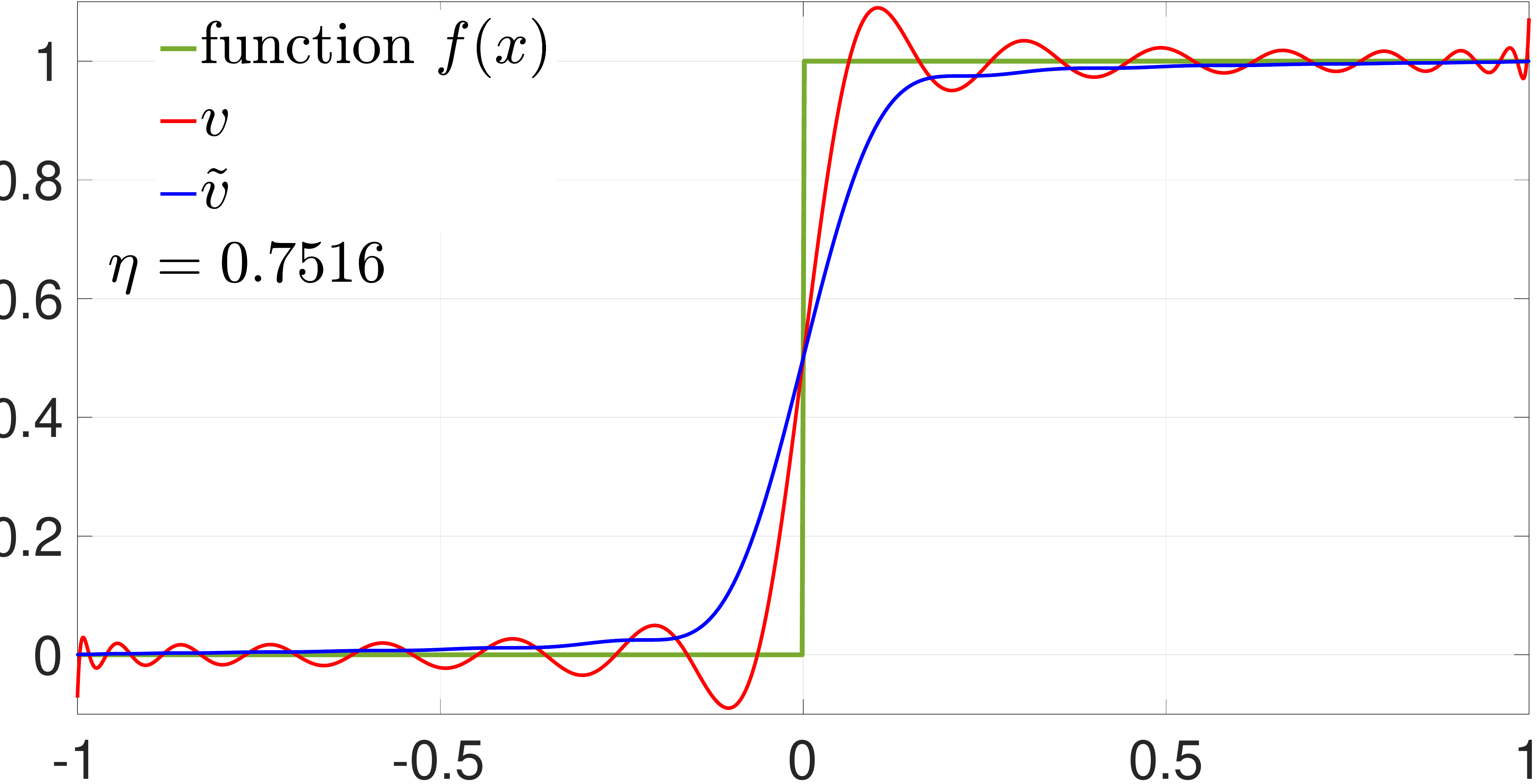} \\
  \end{center}
  \caption{Greedy algorithm results: Test function $f_0$ for different constraint sets $E$ and polynomial spaces $V$. Top: $N = \dim V = 6$, bottom: $N = \dim V = 31$. Left: Constraint $E = F_0$. Center: Constraint $E = F_0 \cap G_0$. Right: Constraint $E = F_0 \cap G_0 \cap F_1$. }\label{fig:constrained-step}
\end{figure}

Our second experiment uses the test function $f  = f_2$, which has a piecewise-constant second derivative. We use a fixed constraint: positivity, monotonicity, and convexity. Using again $N = 6$ and $N = 31$, we investigate the approximation for different ambient spaces $H = H^0$, $H^1$, and $H^2$. Results are displayed in Figure \ref{fig:constrained-f2}. We observe much larger values of $\eta$ in this experiment, but note that the values of $\eta$ decrease as the order of the Sobolev space increases. We also observe that the visual discrepancy between the constrained approximation and the underlying function is also considerably larger in this experiment. However, the approximation quality still appears good for the larger value of $N = 31$.
\begin{figure}[htbp]
\includegraphics[width=0.32\textwidth]{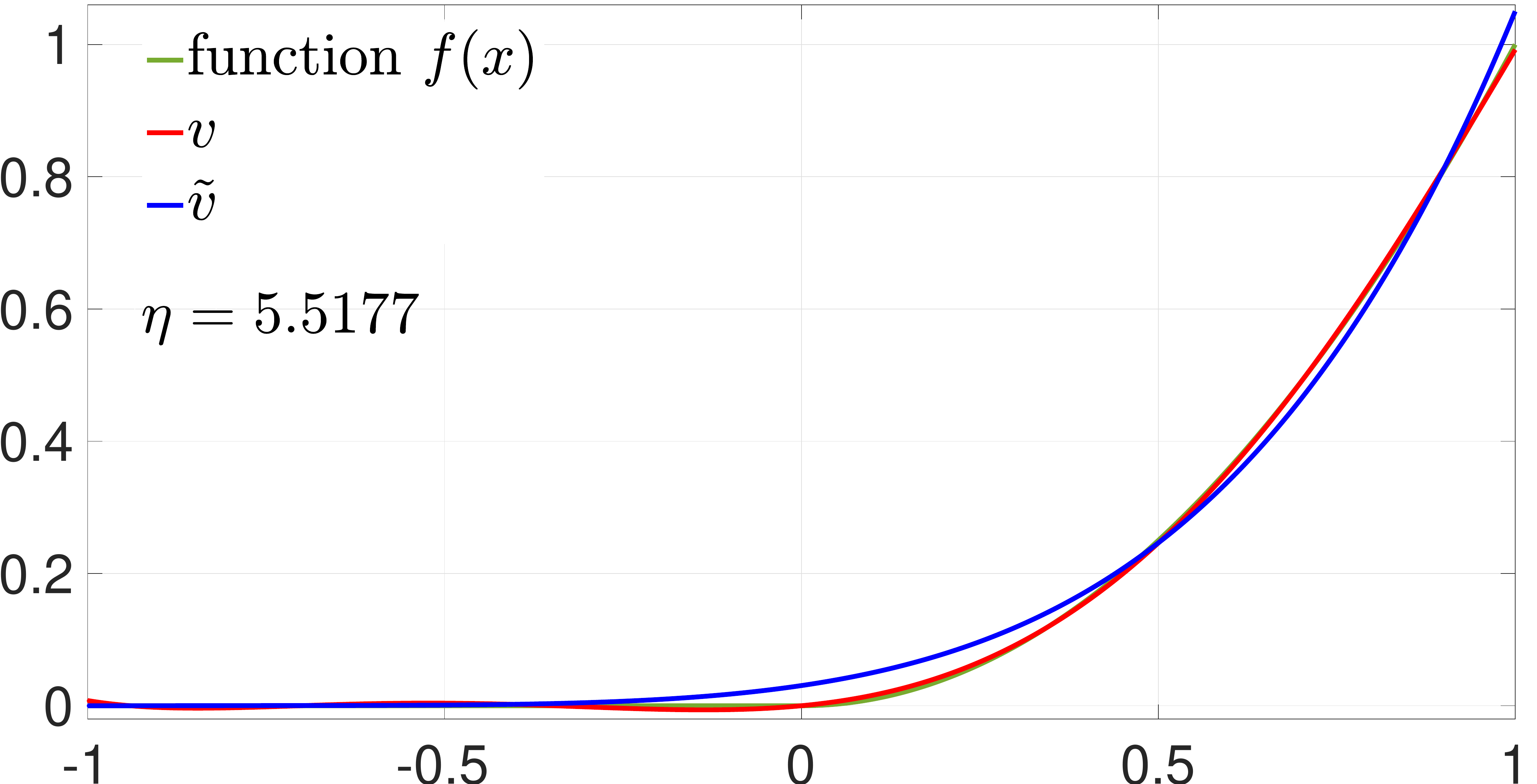}   
\includegraphics[width=0.32\textwidth]{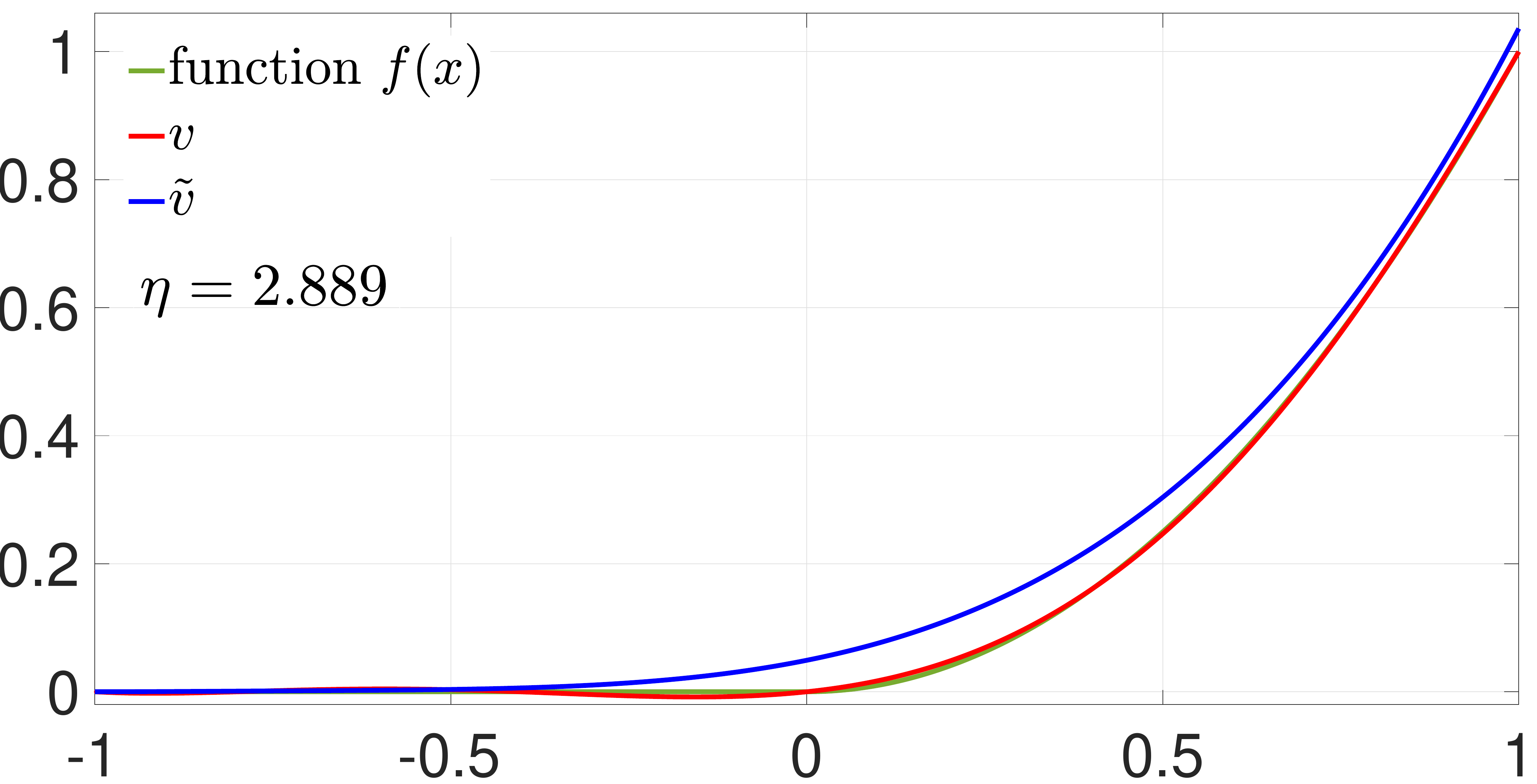}   
\includegraphics[width=0.32\textwidth]{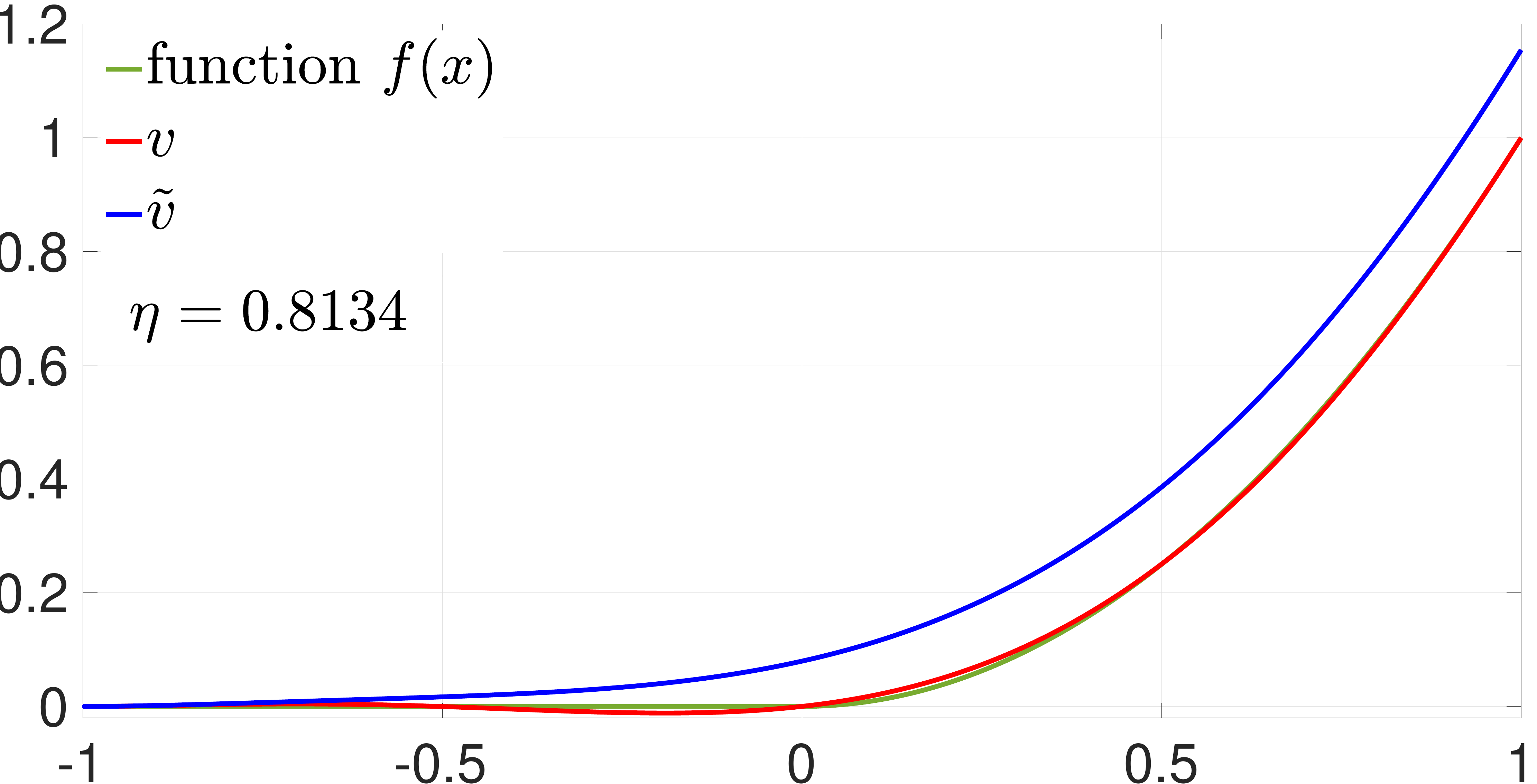}\\ 
\includegraphics[width=0.32\textwidth]{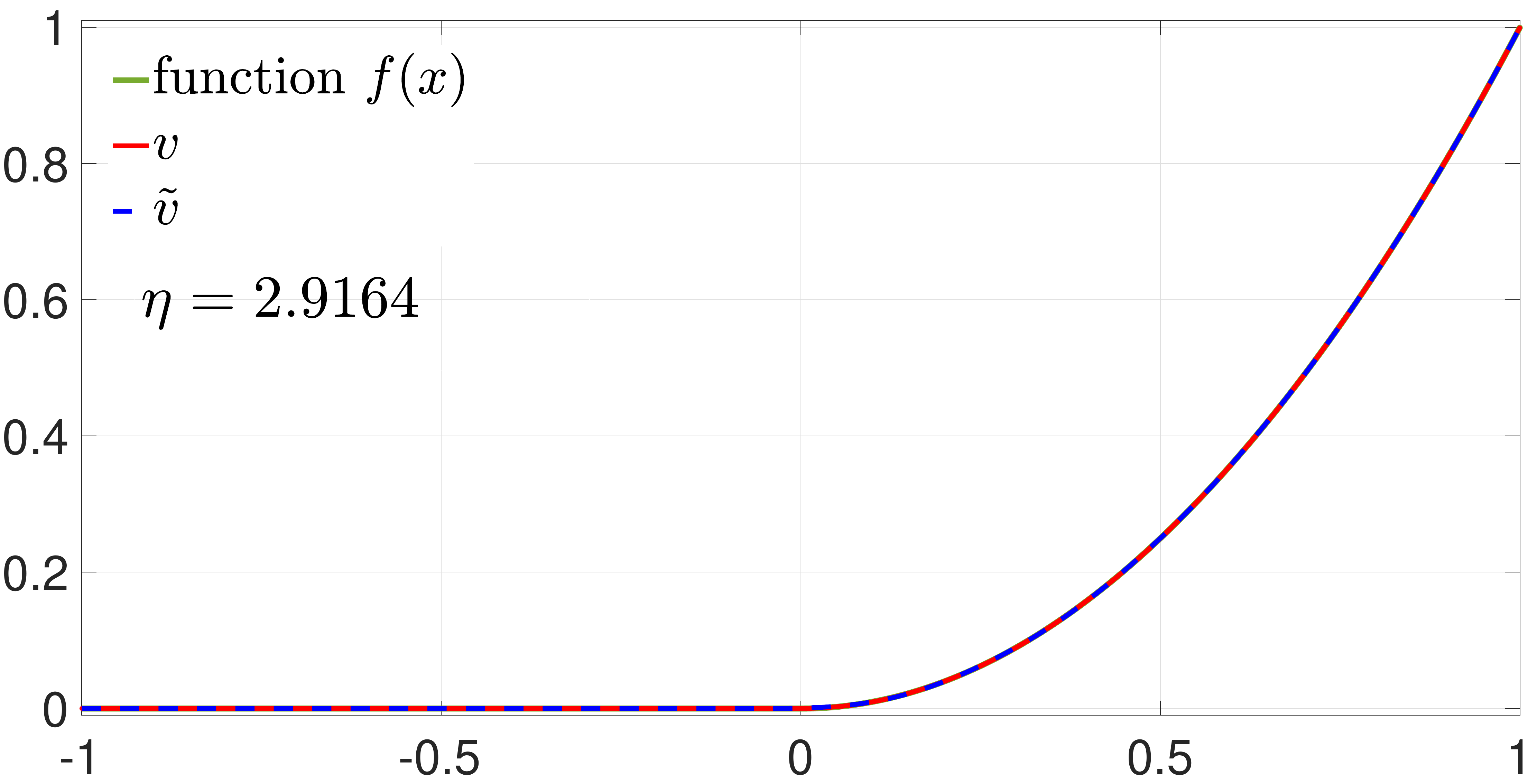}  
\includegraphics[width=0.32\textwidth]{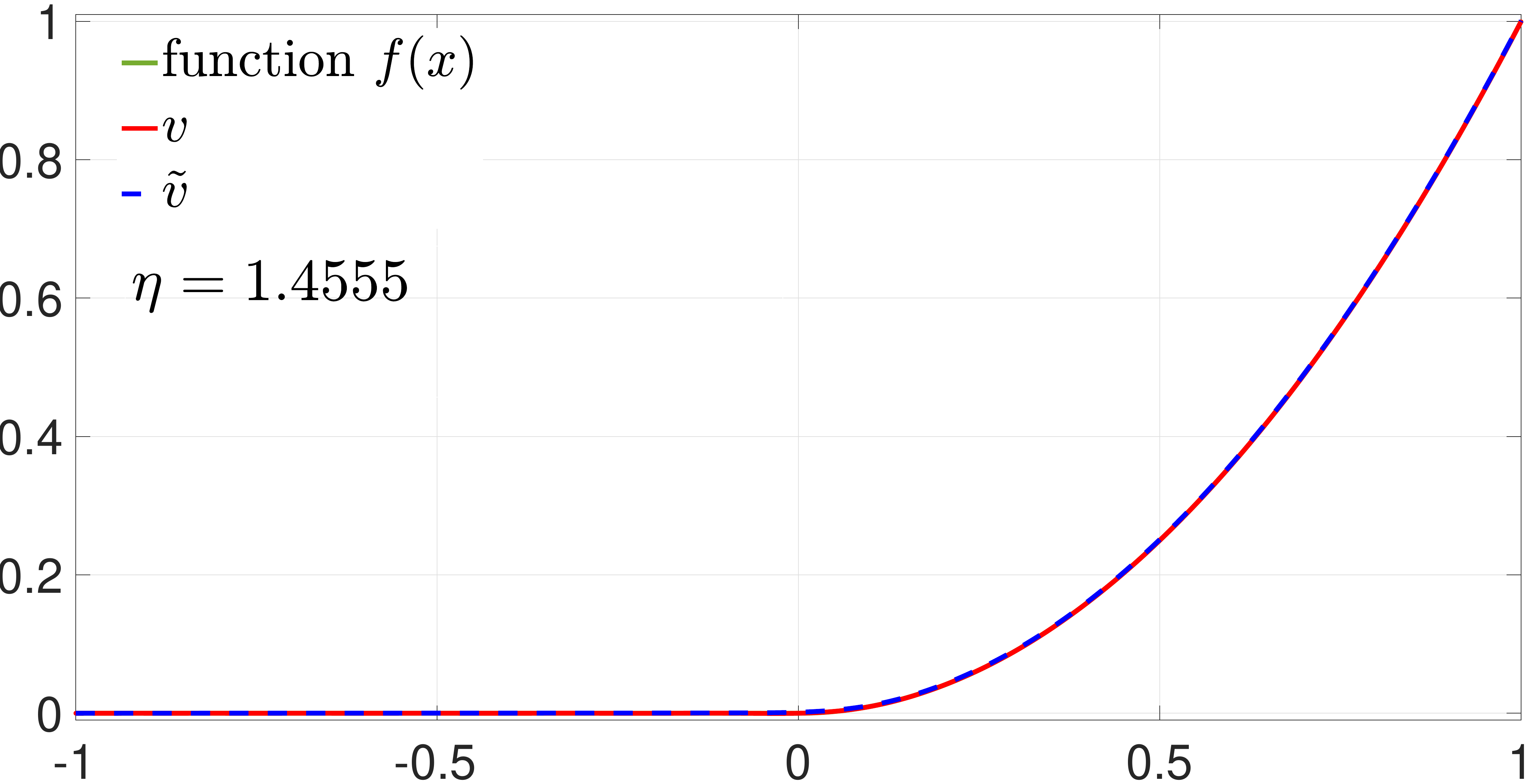}  
\includegraphics[width=0.32\textwidth]{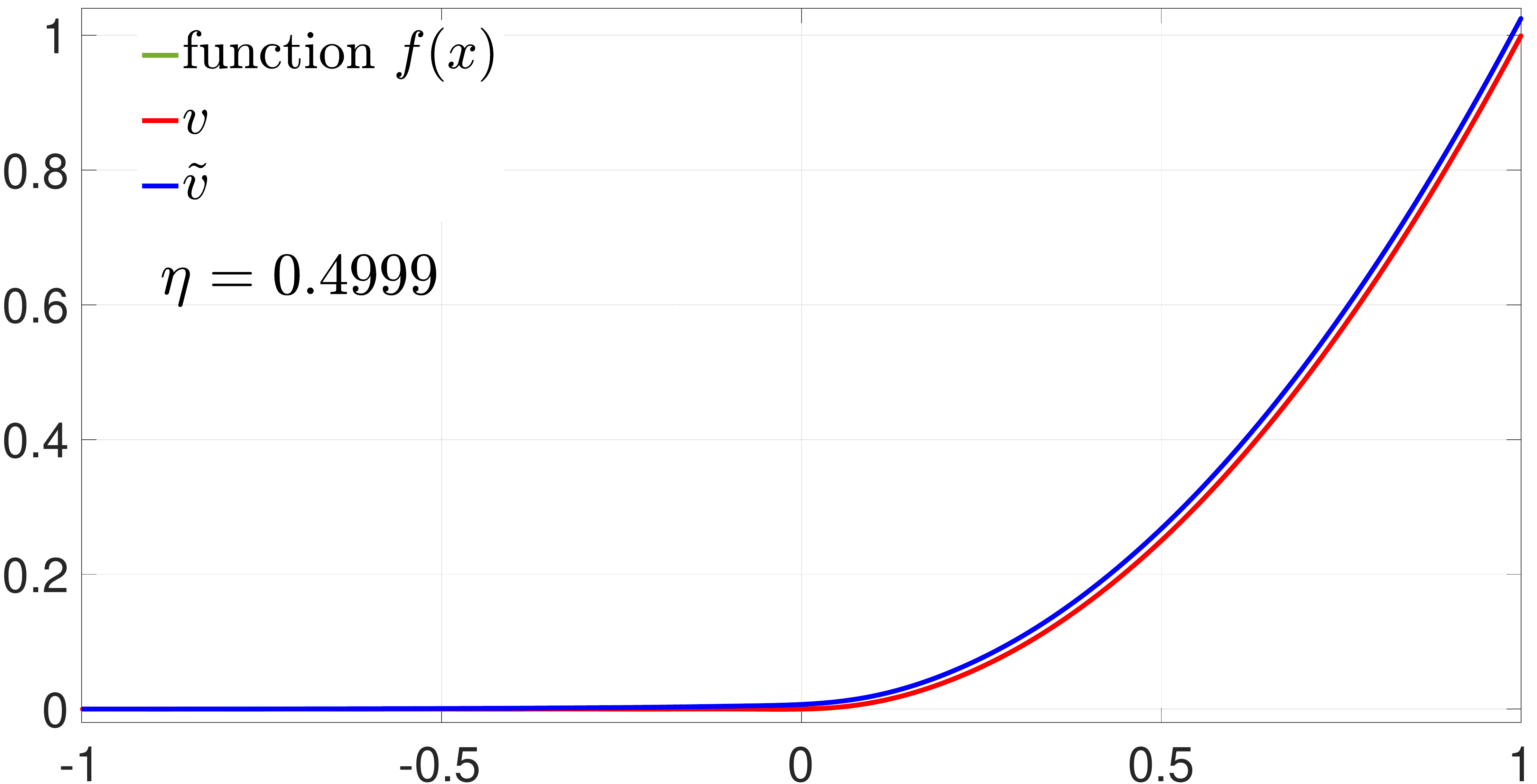}  
\caption{Test function $f_2$ for different polynomial spaces $V$ and ambient spaces $H$. The constraint is $E = F_0 \cap F_1 \cap F_2$. Top: $N = 6$, bottom: $N = 31$. Left: $H = H^0$. Center: $H = H^1$. Right: $H = H^2$. }
  \label{fig:constrained-f2}
\end{figure}

\subsection{Constrained approximation as a nonlinear filter}
The right-hand panels in Figure \ref{fig:constrained-step} show that the monotonicity constraint removes oscillations in the approximation. These empirical results suggest that the constrained optimization procedure is a type of spectral filter. There is a stronger theoretical motivation for this observation as well.
\begin{proposition}
  Let $E \subset V$ be a nonempty, closed, convex set in $H$. Given some $v \in V$, let $\tilde{v}$ be the solution to \eqref{eq:constopt-continuous} (i.e., also the solution to \eqref{eq:constopt-discrete}). If $0 \in E$, then, $\|\tilde{v}\| \leq \|v\|$.
\end{proposition}
\begin{proof}
  Projections onto closed convex sets in Hilbert spaces are nonexpansive \cite{cheney_proximity_1959}. I.e., $\| \tilde{v} - P(0) \| \leq \| v - 0\|$, where $P: V \rightarrow E$ is the projection operator from $V$ to $E$. Since $0 \in E$, then $P(0) = 0$.
\end{proof}
In general, the assumption that $E$ is closed and convex is automatically satisfied from our apparatus in Sections \ref{sec:setup} and \ref{sec:method}. The only nontrivial requirement is that $v = 0$ is a member of the constraint set $E$. All the examples in Figures \ref{fig:constrained-step} and \ref{fig:constrained-f2} satisfy $0 \in E$, and thus we expect that the optimization problem decreases the norm of the function, just as a standard linear filter would. Note, however, that our ``filter'' (optimization) is a nonlinear map.

To illustrate this filter interpretation, we compare in Figures \ref{fig:spectrum-step} and \ref{fig:spectrum-f2} the magnitude of the before-optimization and after-optimization expansion coefficients. These figures correspond to the experiments in Figures \ref{fig:constrained-step} and \ref{fig:constrained-f2}, respectively.

For the step function example shown in Figure \ref{fig:spectrum-step}, we see that when monotonicity is enforced, there is a steeper decay of the higher order coefficients in the constrained approximation. The stronger decay of coefficients is also observed when only positivity/boundedness is enforced, but the increase in decay is less pronounced. All these observations are qualitatively consistent with Figure \ref{fig:constrained-step}. We emphasize that this constrained optimization procedure is nonlinear, so that our approximation cannot easily be written in coefficient space as a standard (linear) spectral filter.



\begin{figure}[htbp]
\begin{center}
  \includegraphics[width=0.32\textwidth]{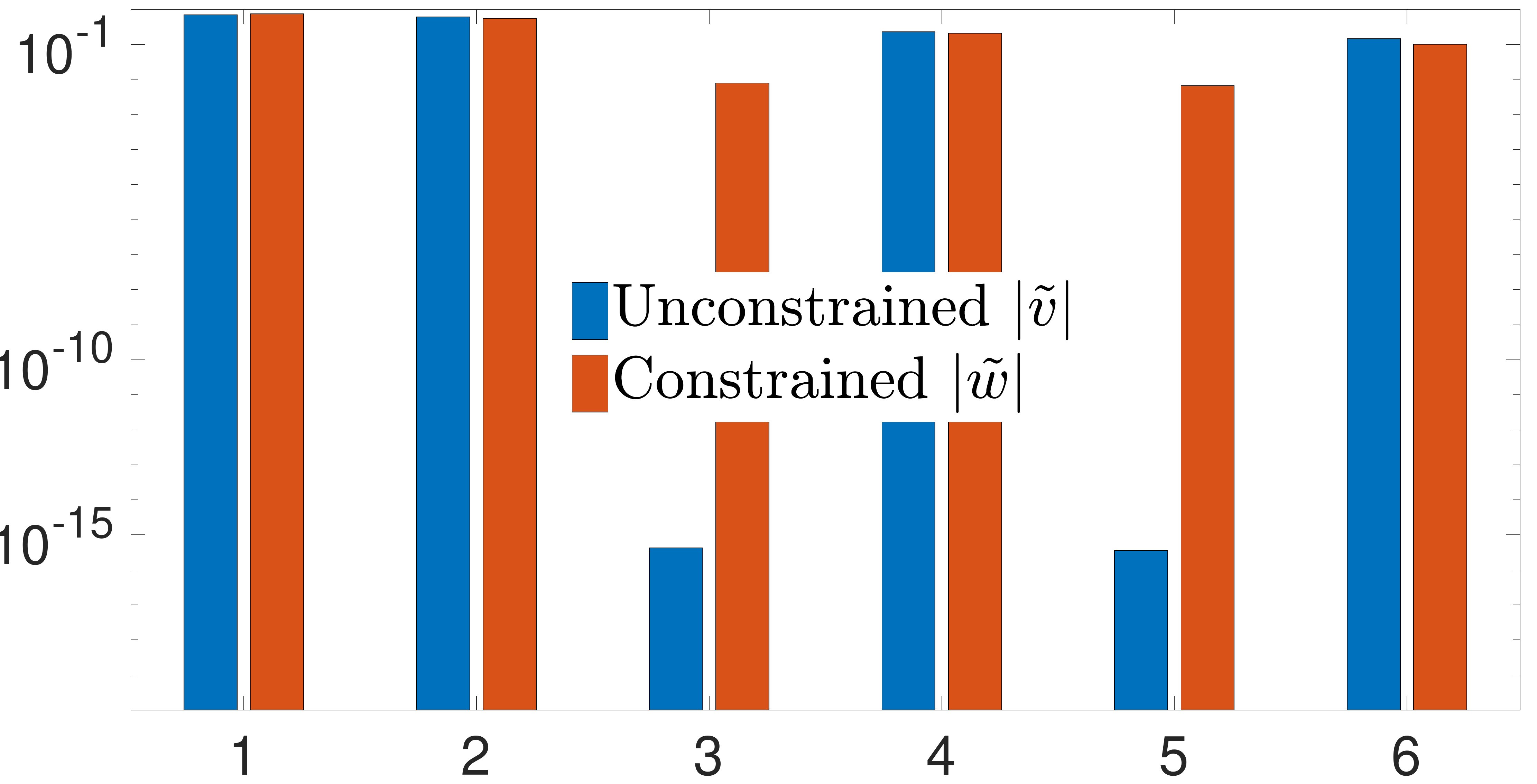}
  \includegraphics[width=0.32\textwidth]{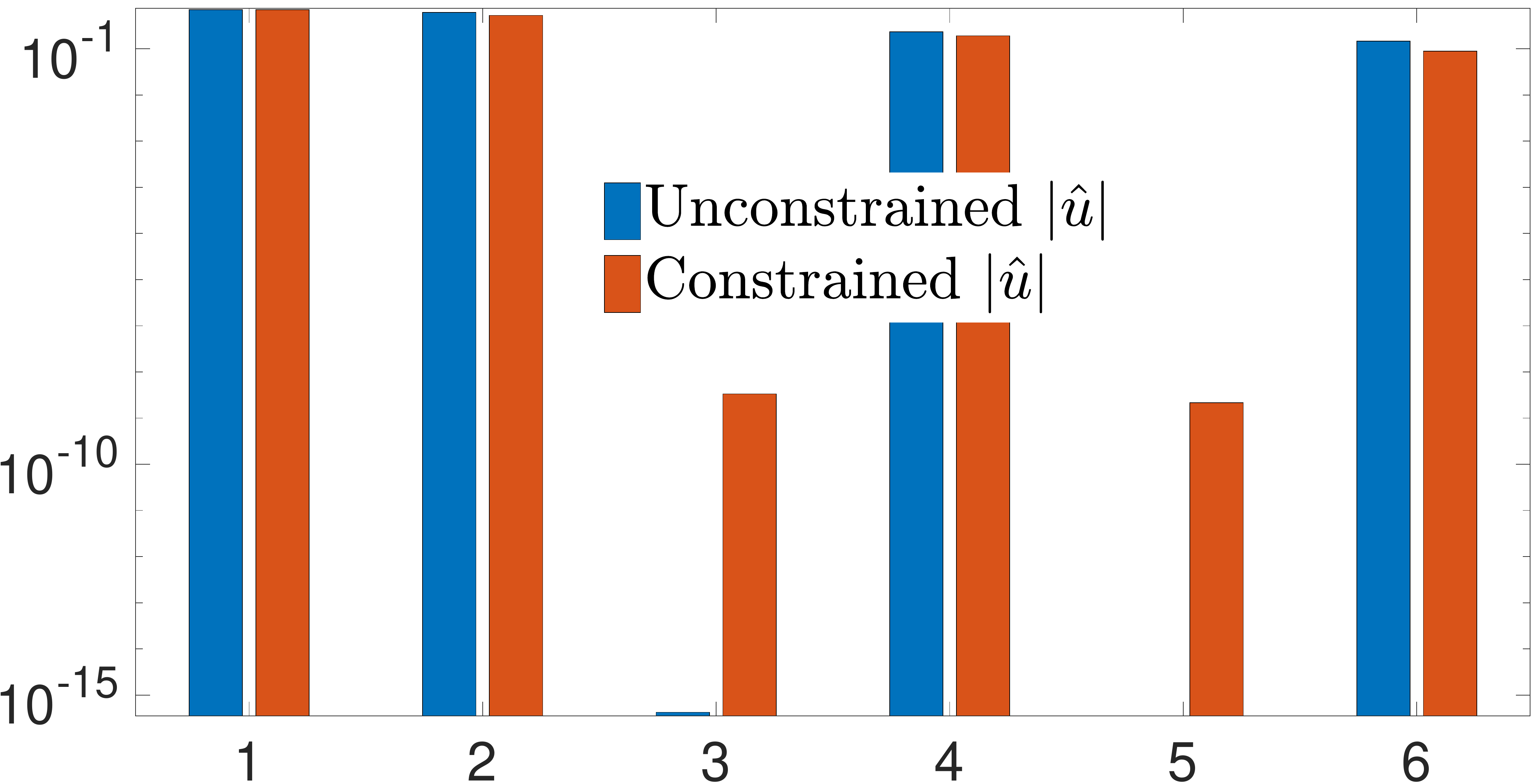}
  \includegraphics[width=0.32\textwidth]{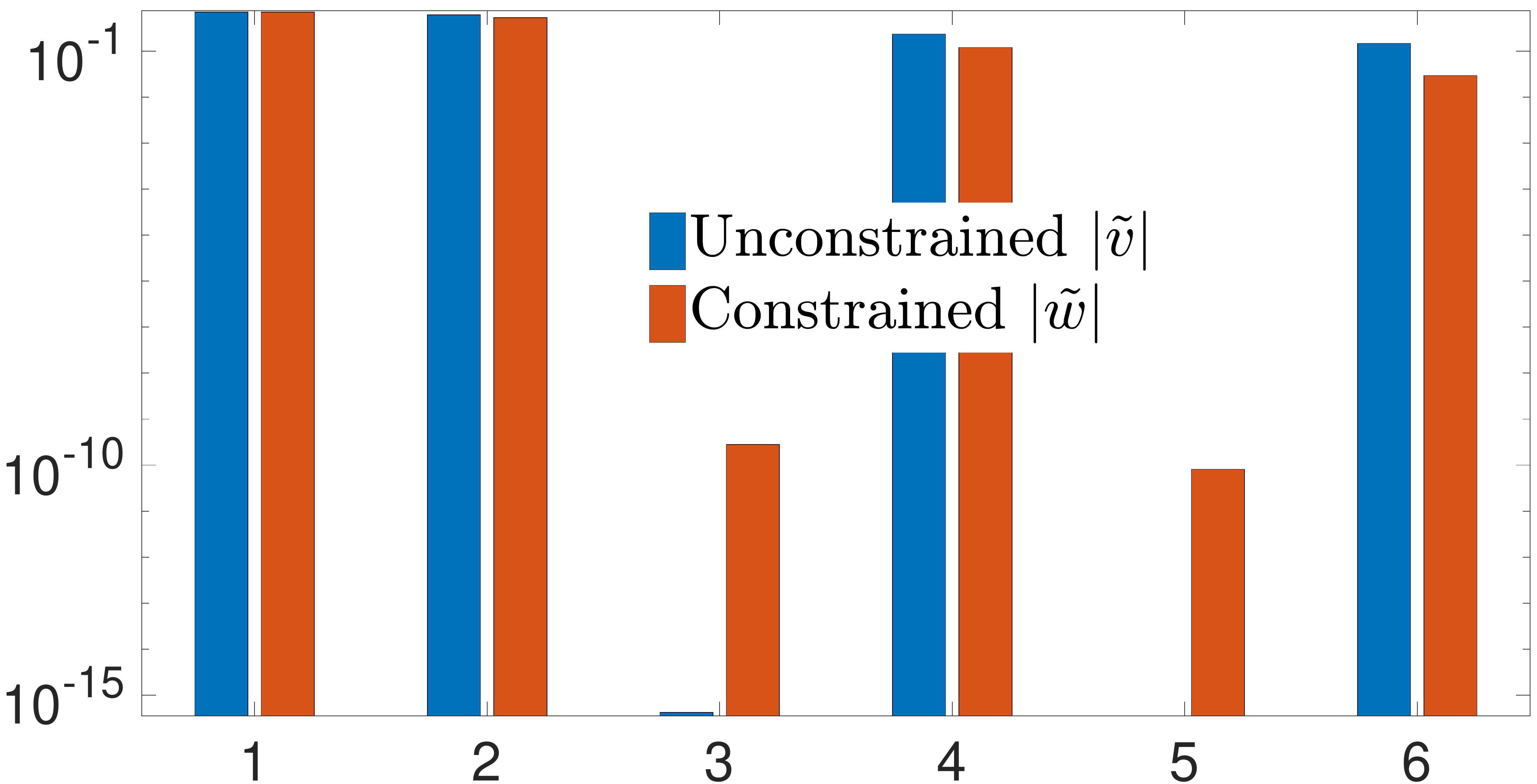} \\
  \includegraphics[width=0.32\textwidth]{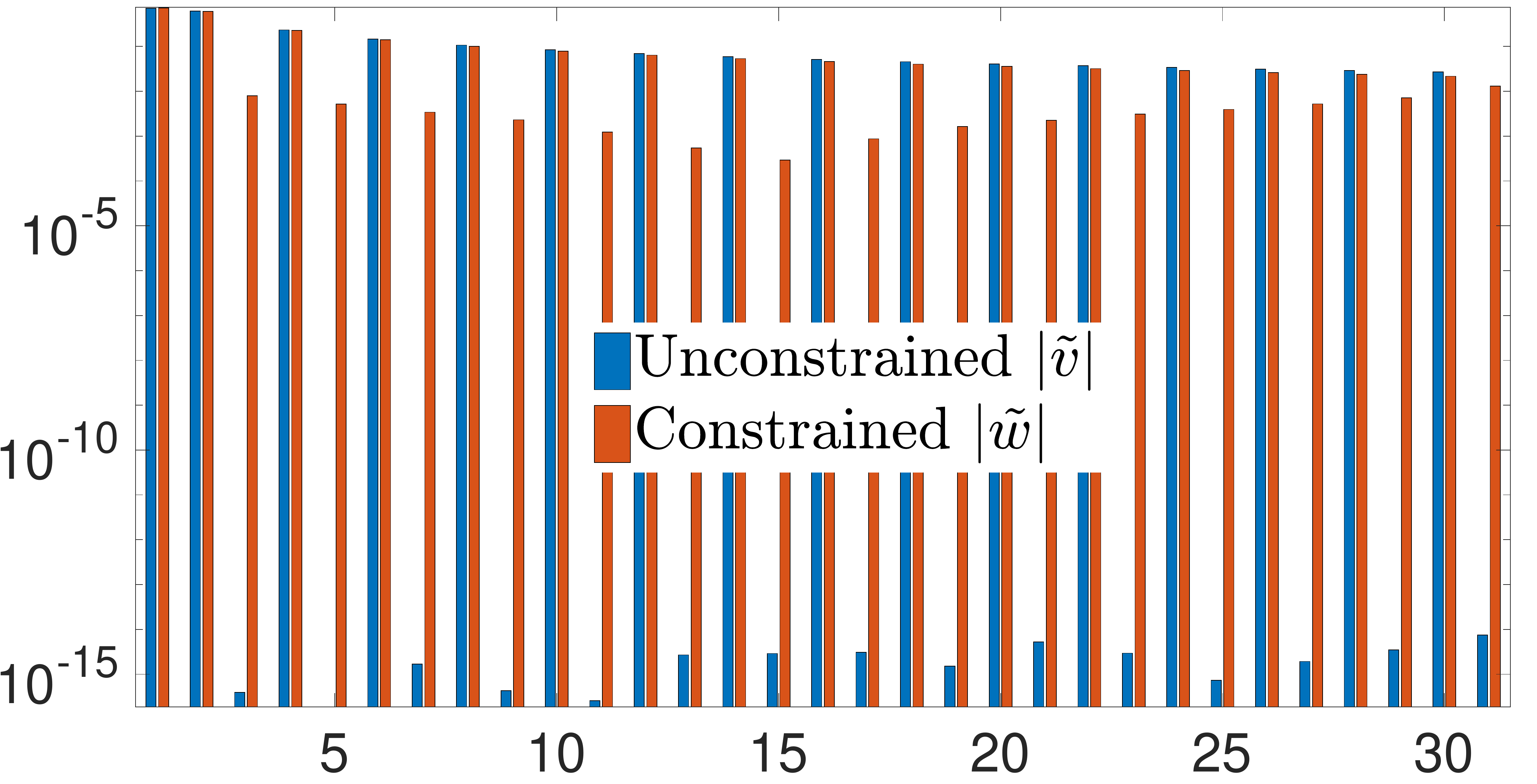}
  \includegraphics[width=0.32\textwidth]{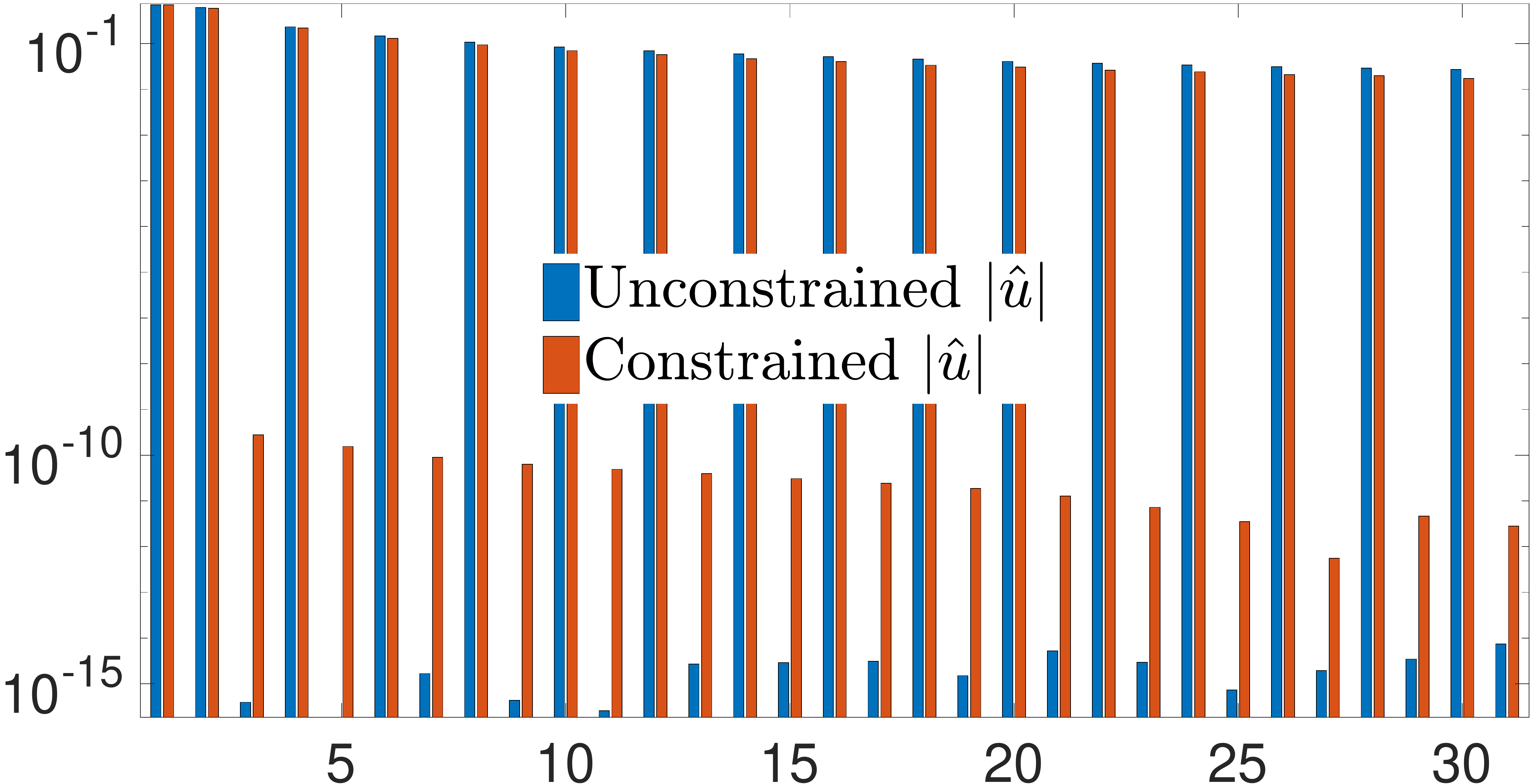}
  \includegraphics[width=0.32\textwidth]{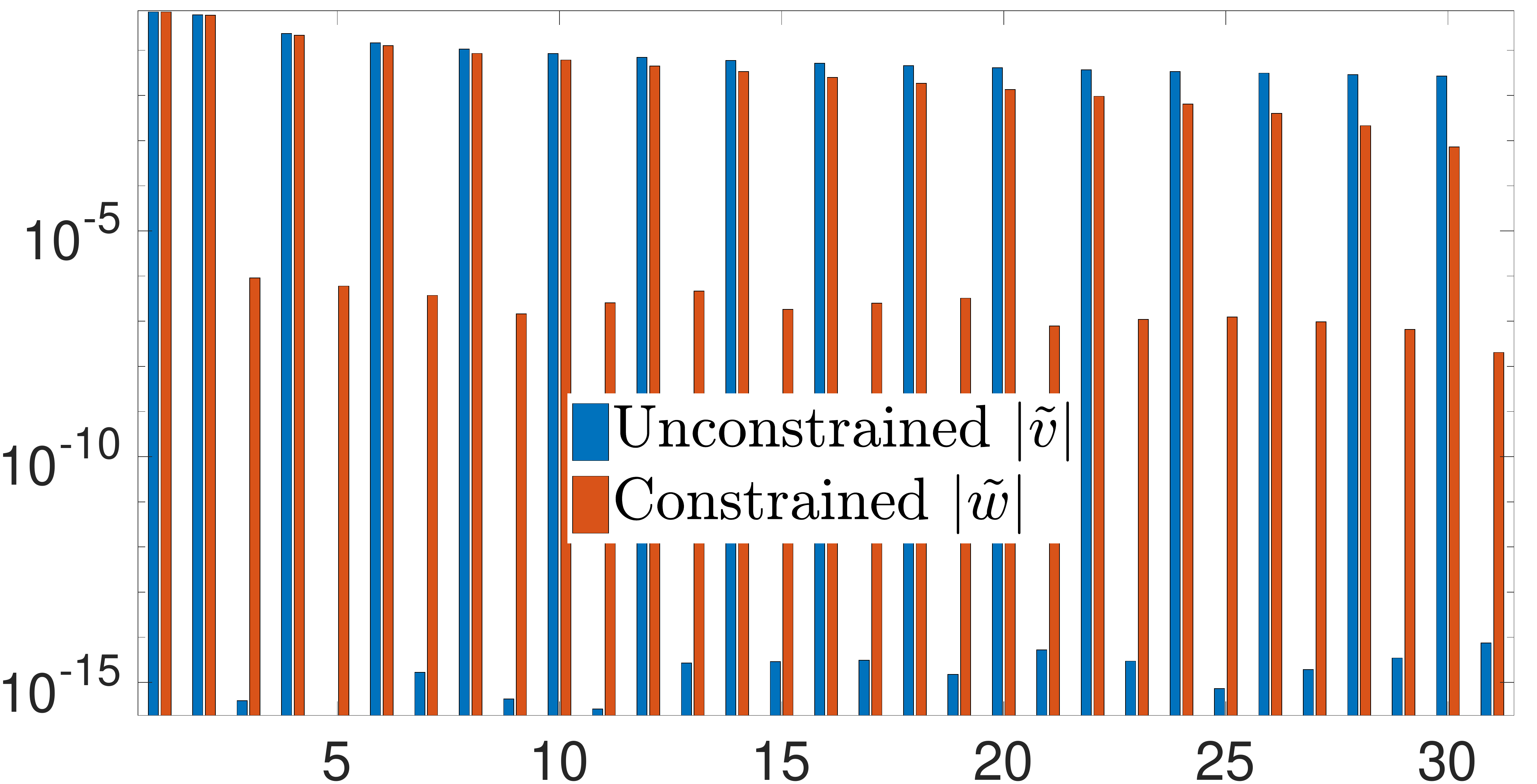}
\end{center}
\caption{
  Companion to Figure \ref{fig:constrained-step}. Bar plot showing unconstrained projection coefficients magnitude $|\widetilde{v}_j|$ vs various constrained projection coefficients magnitude $|\widetilde{w}_j|$. Top: $N = 6$. Bottom: $N=31$. Left: Constraint $E = F_0$. Center: Constraint $E = F_0 \cap G_0$. Right: Constraint $E = F_0 \cap G_0 \cap F_1$. 
}\label{fig:spectrum-step}
\end{figure}

\begin{figure}[htbp]
  \begin{center}
    \includegraphics[width=0.32\textwidth]{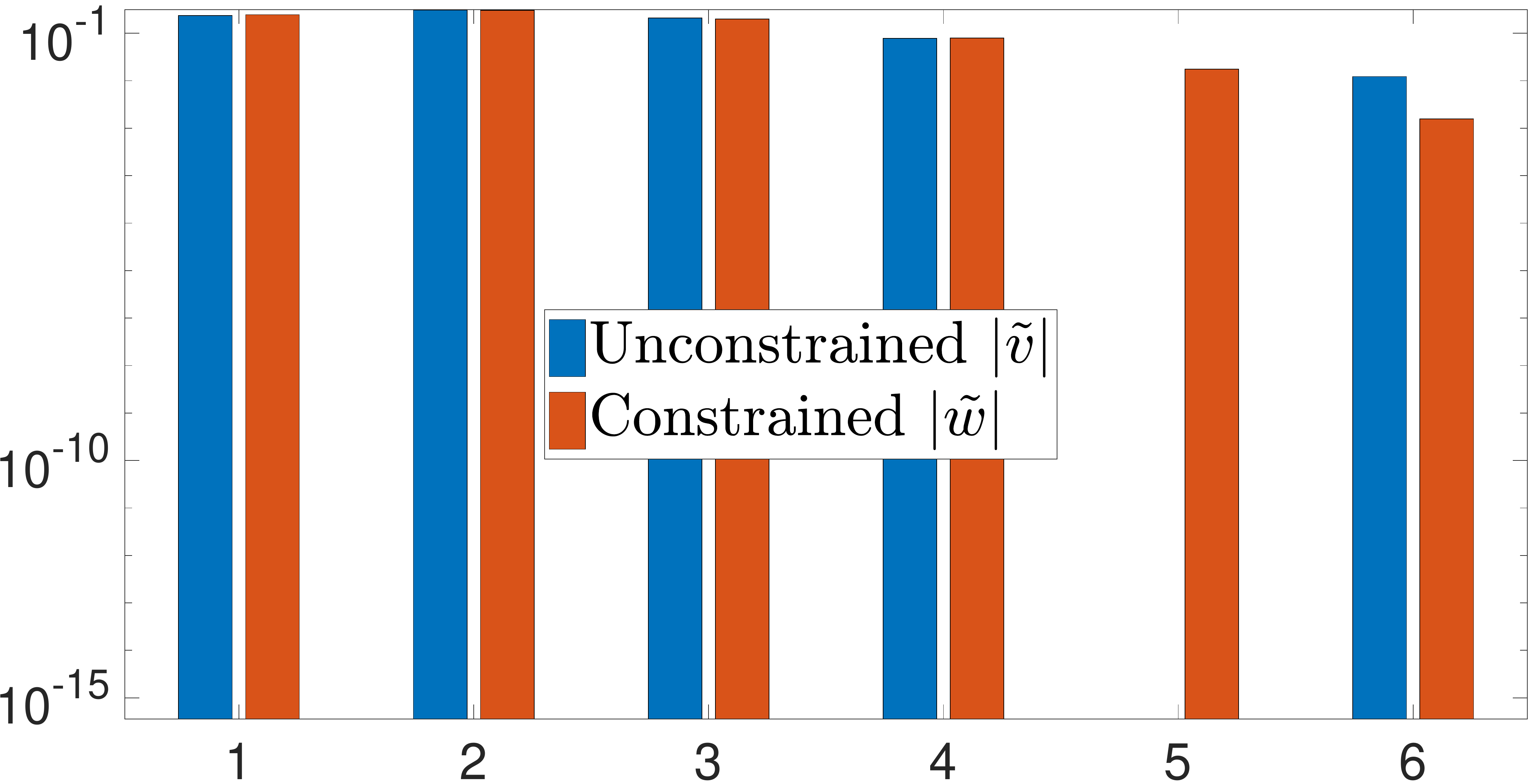}
    \includegraphics[width=0.32\textwidth]{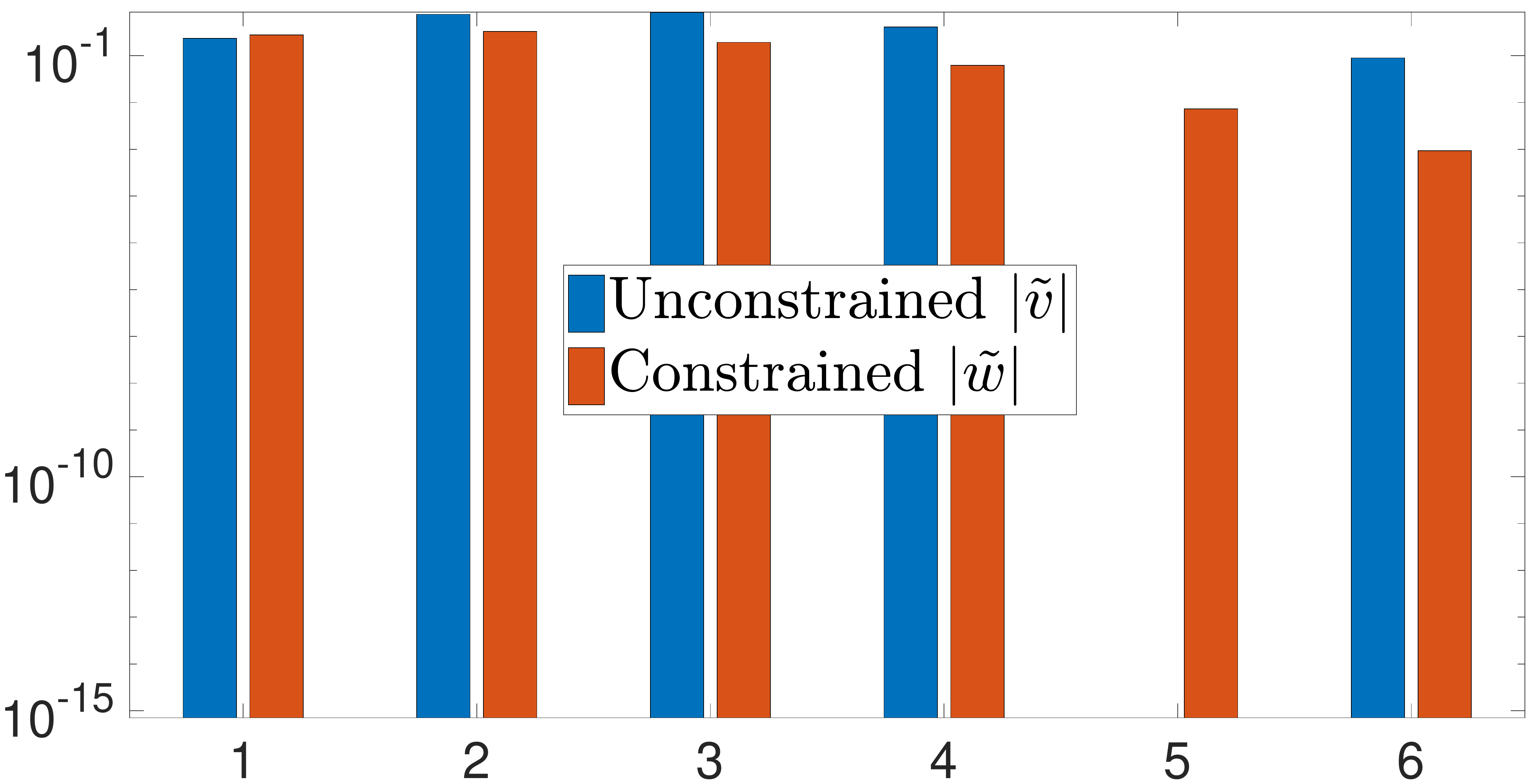}
    \includegraphics[width=0.32\textwidth]{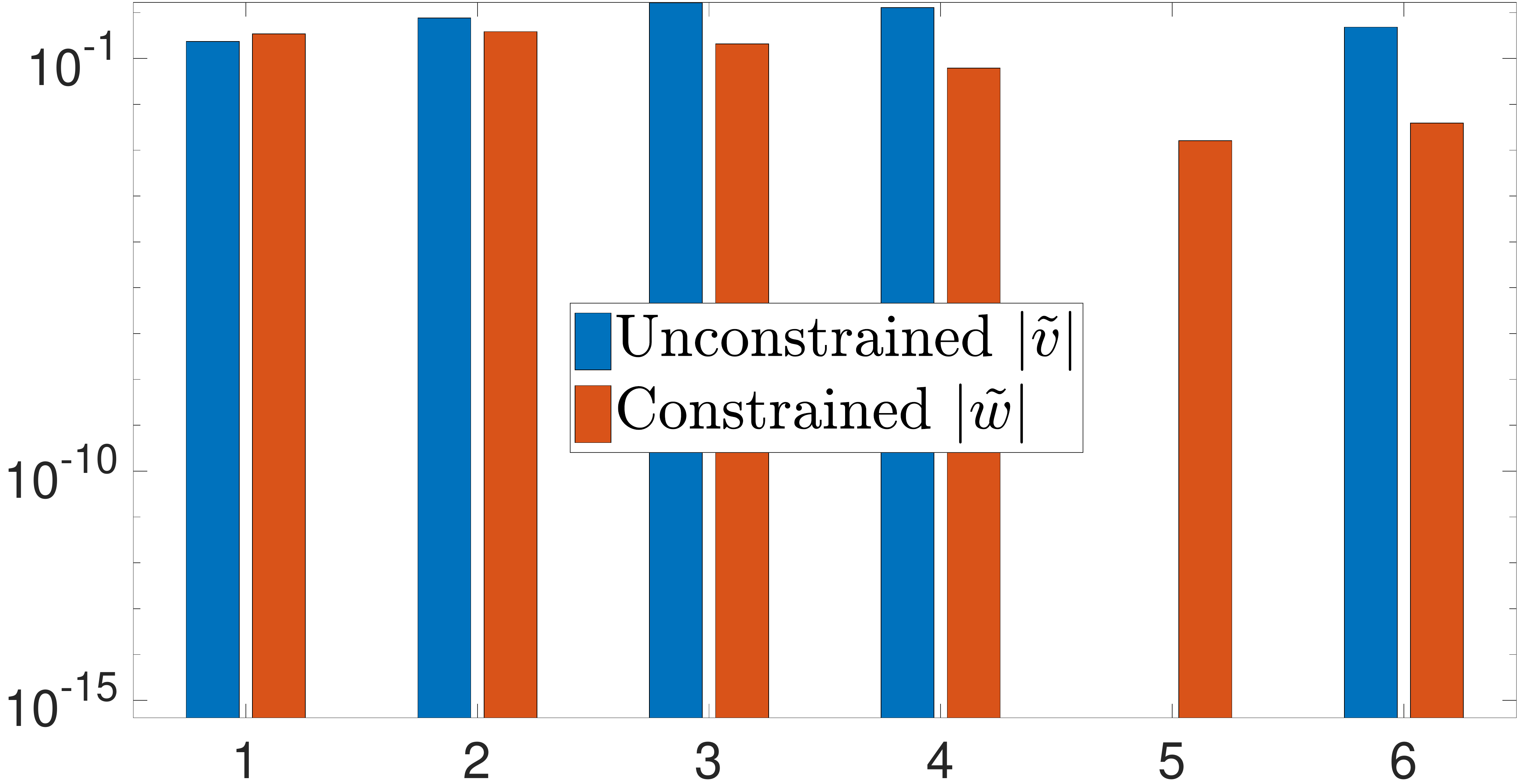} \\
    \includegraphics[width=0.32\textwidth]{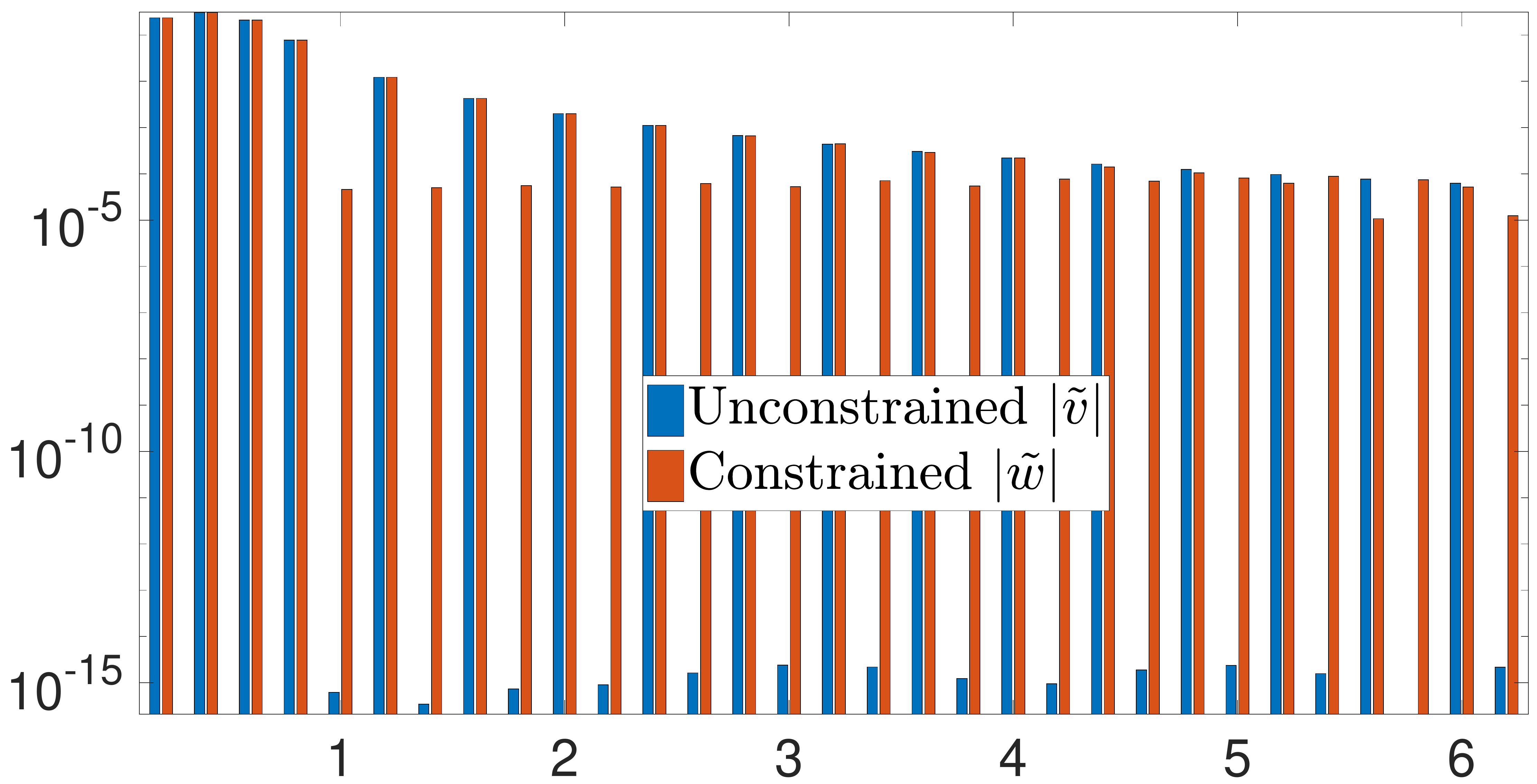}
    \includegraphics[width=0.32\textwidth]{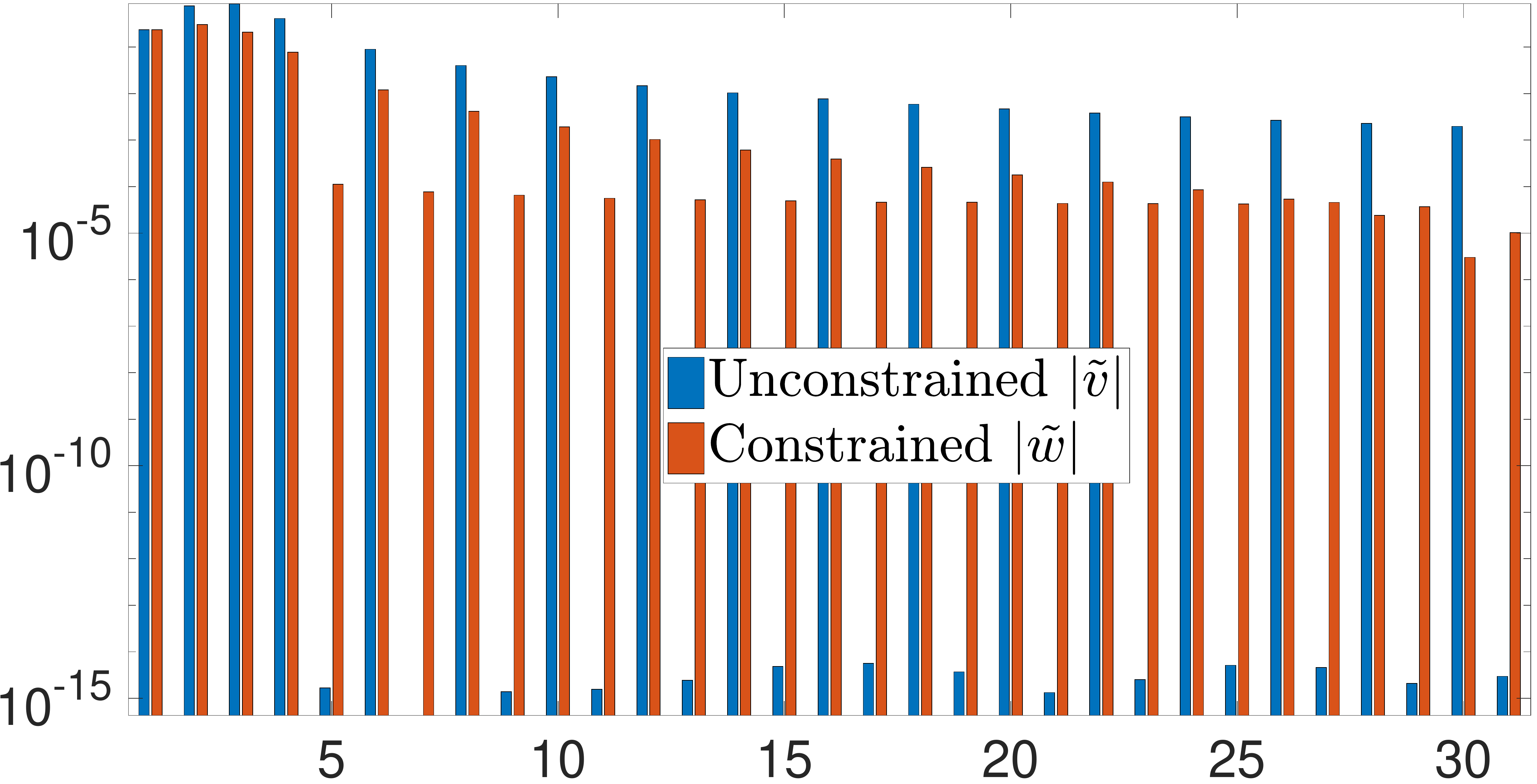}
    \includegraphics[width=0.32\textwidth]{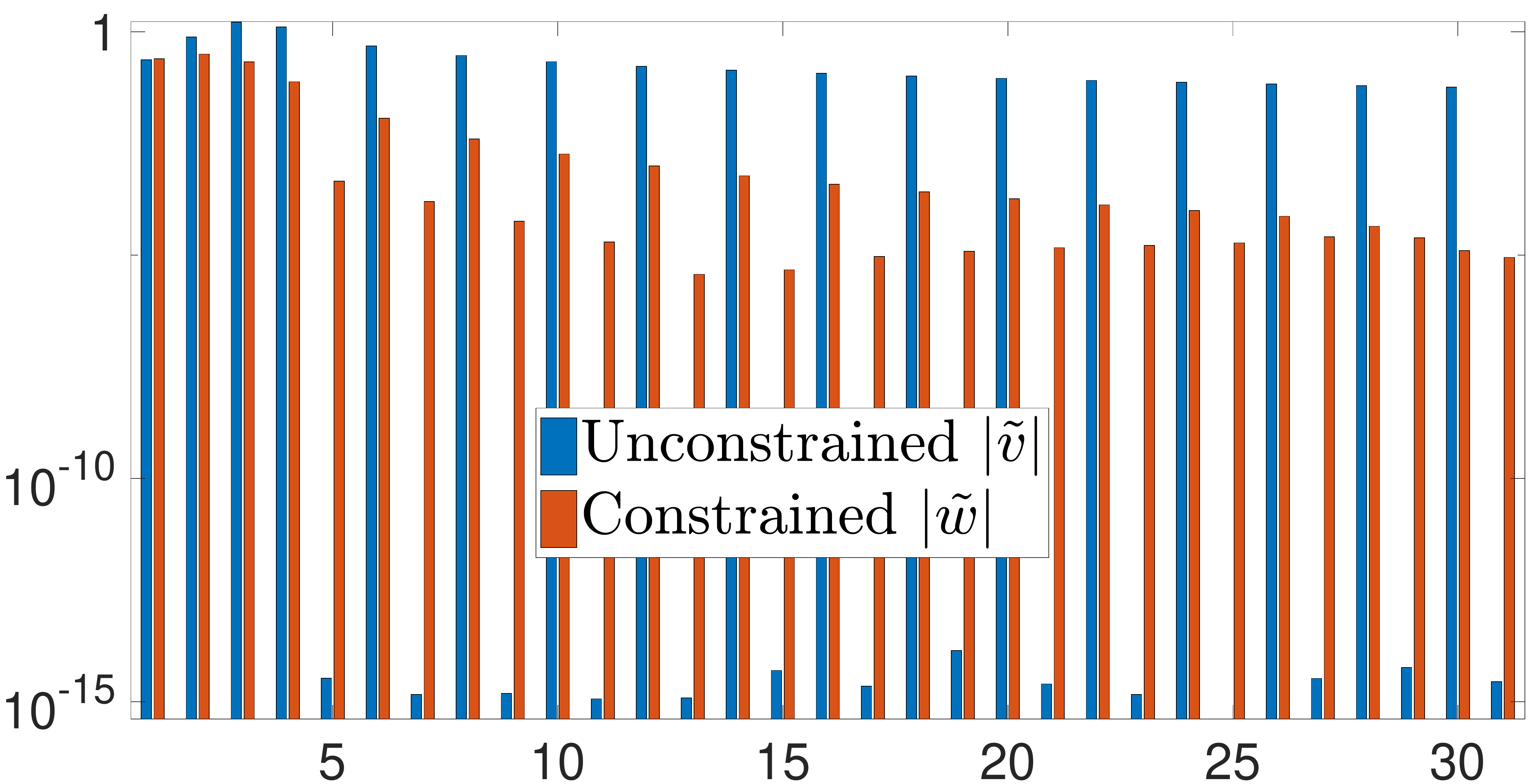}
  \end{center}
  \caption{Companion to Figure \ref{fig:constrained-f2}. Bar plot showing unconstrained projection coefficients magnitude $|\widetilde{v}_j|$ vs various constrained projection coefficients magnitude $|\widetilde{w}_j|$. Top: $N = 6$, bottom: $N = 31$. Left: $H = H^0$. Center: $H = H^1$. Right: $H = H^2$.}  \label{fig:spectrum-f2}
\end{figure}

\subsection{Convergence rates}
Optimal Hilbert space projections of smooth functions onto polynomial spaces converge at a rate commensurate with the function smoothness. We investigate in this section whether the corresponding \textit{constrained} projections have similar convergence rates. In Figure \ref{fig:conv-plot-f} we show convergence of $H = L^2$-optimal (unconstrained) polynomial projections versus the output from our constrained optimization procedure. Our constrained approximations are less accurate, but the convergence \textit{rates} are unchanged.

\begin{figure}[htbp]
\begin{center}
  \includegraphics[width=0.32\textwidth]{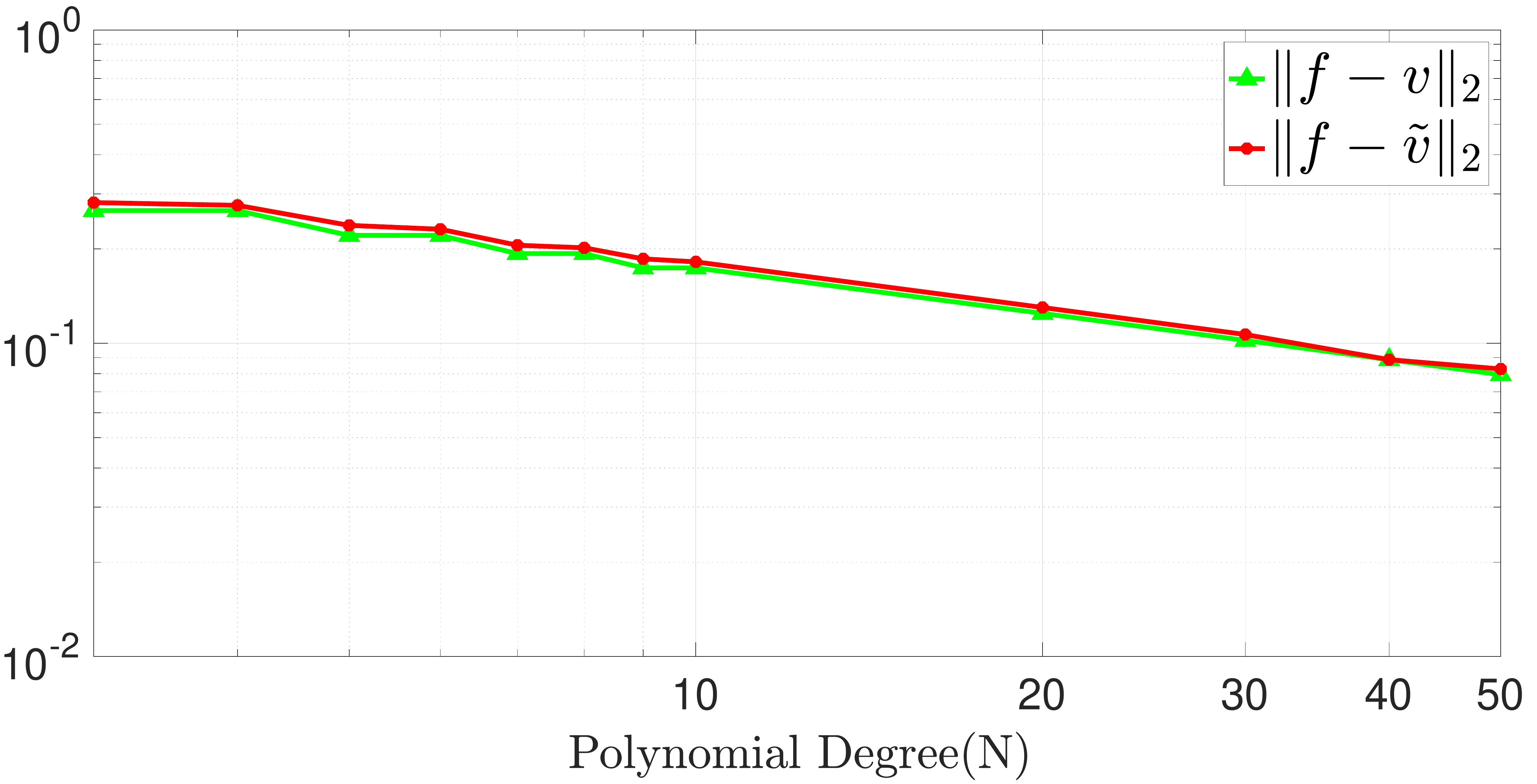}
  \includegraphics[width=0.32\textwidth]{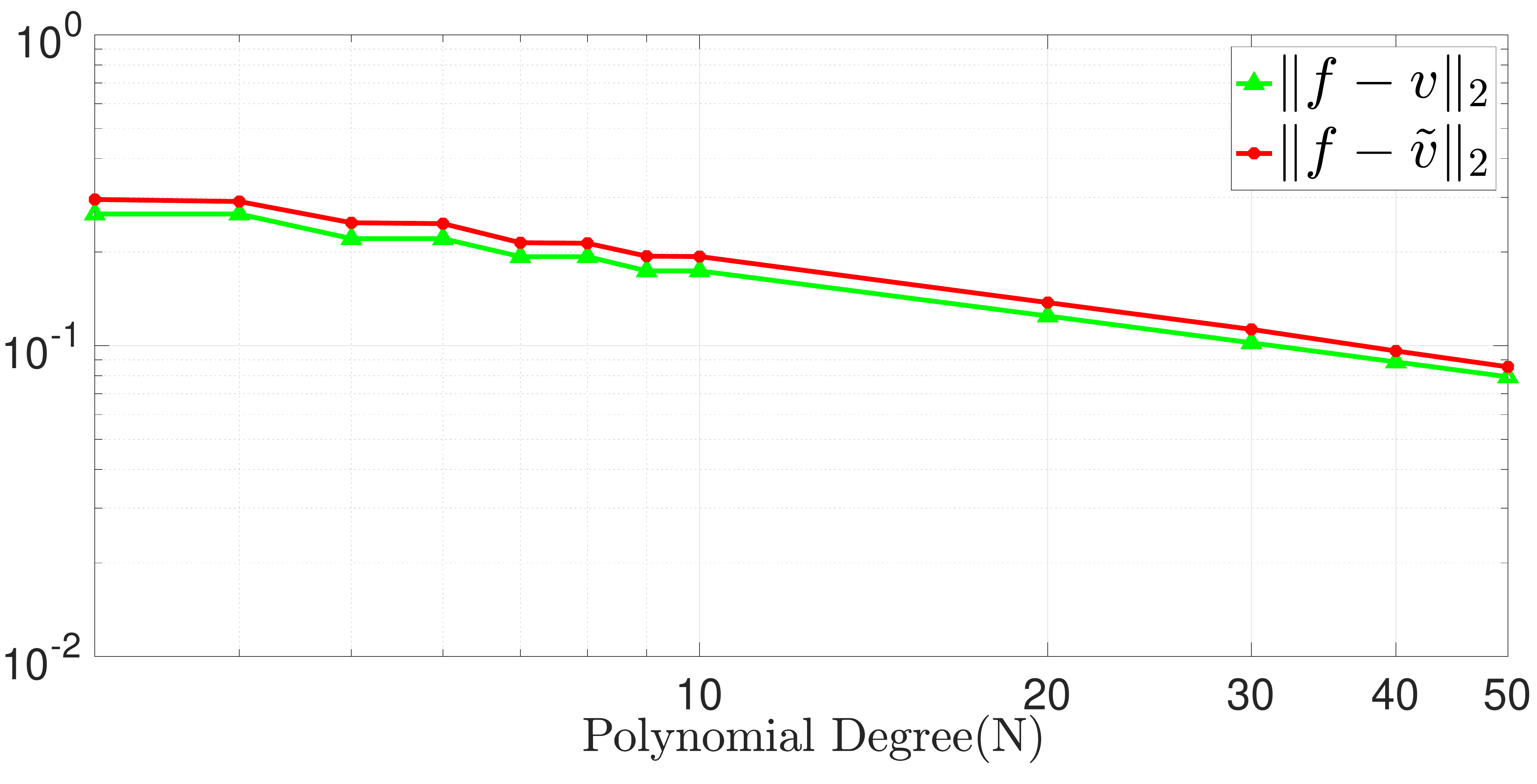}
  \includegraphics[width=0.32\textwidth]{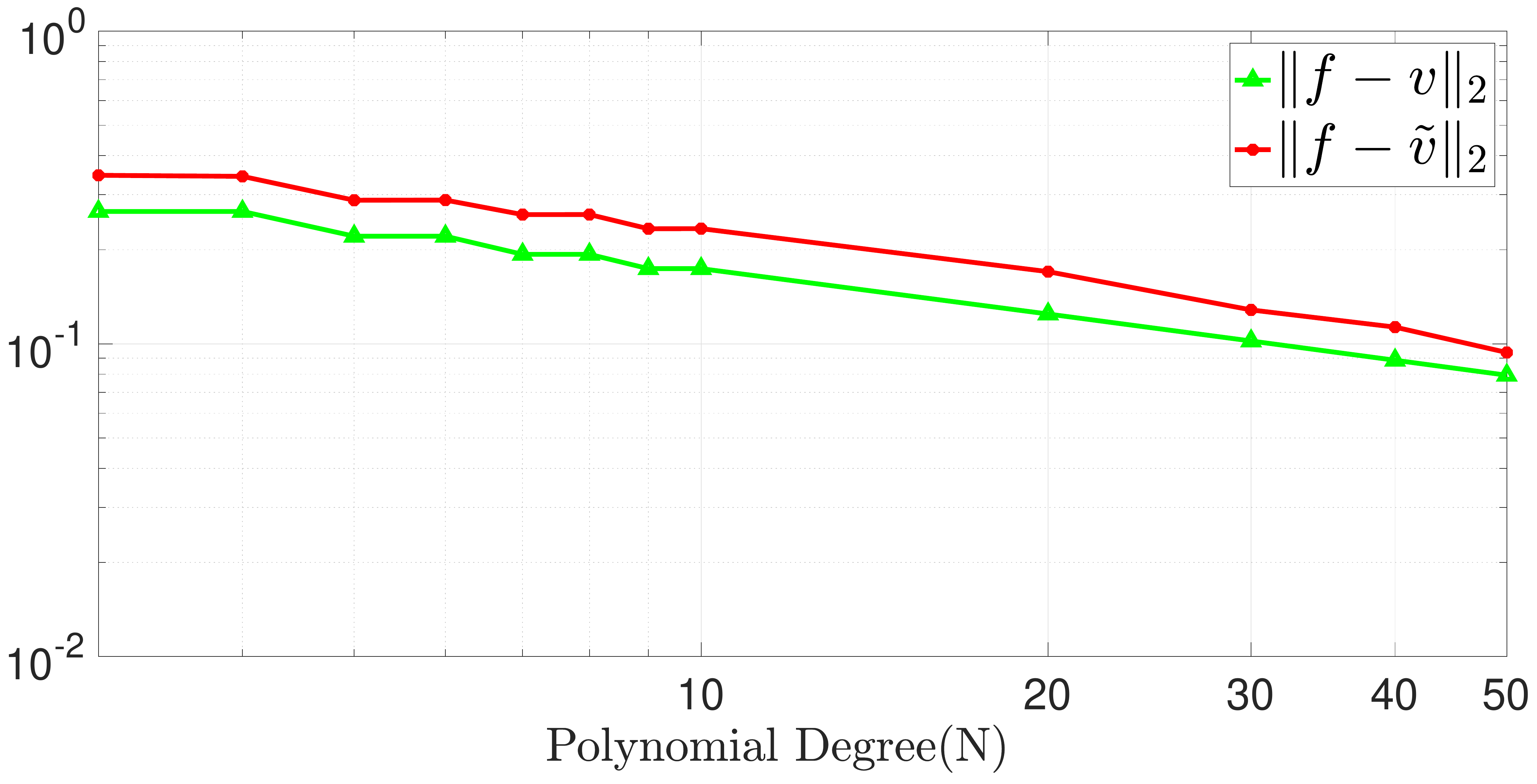} \\
  \includegraphics[width=0.32\textwidth]{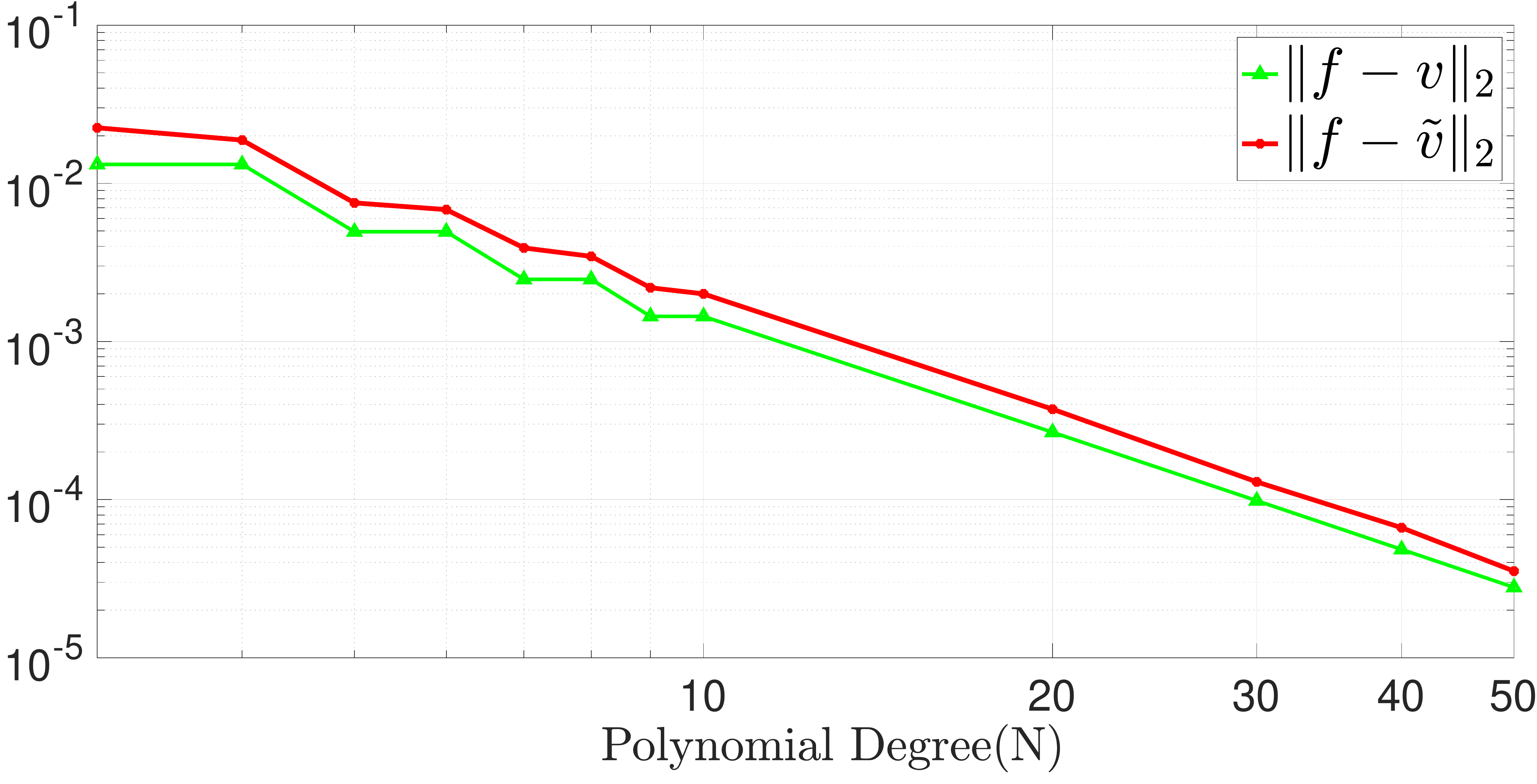}
  \includegraphics[width=0.32\textwidth]{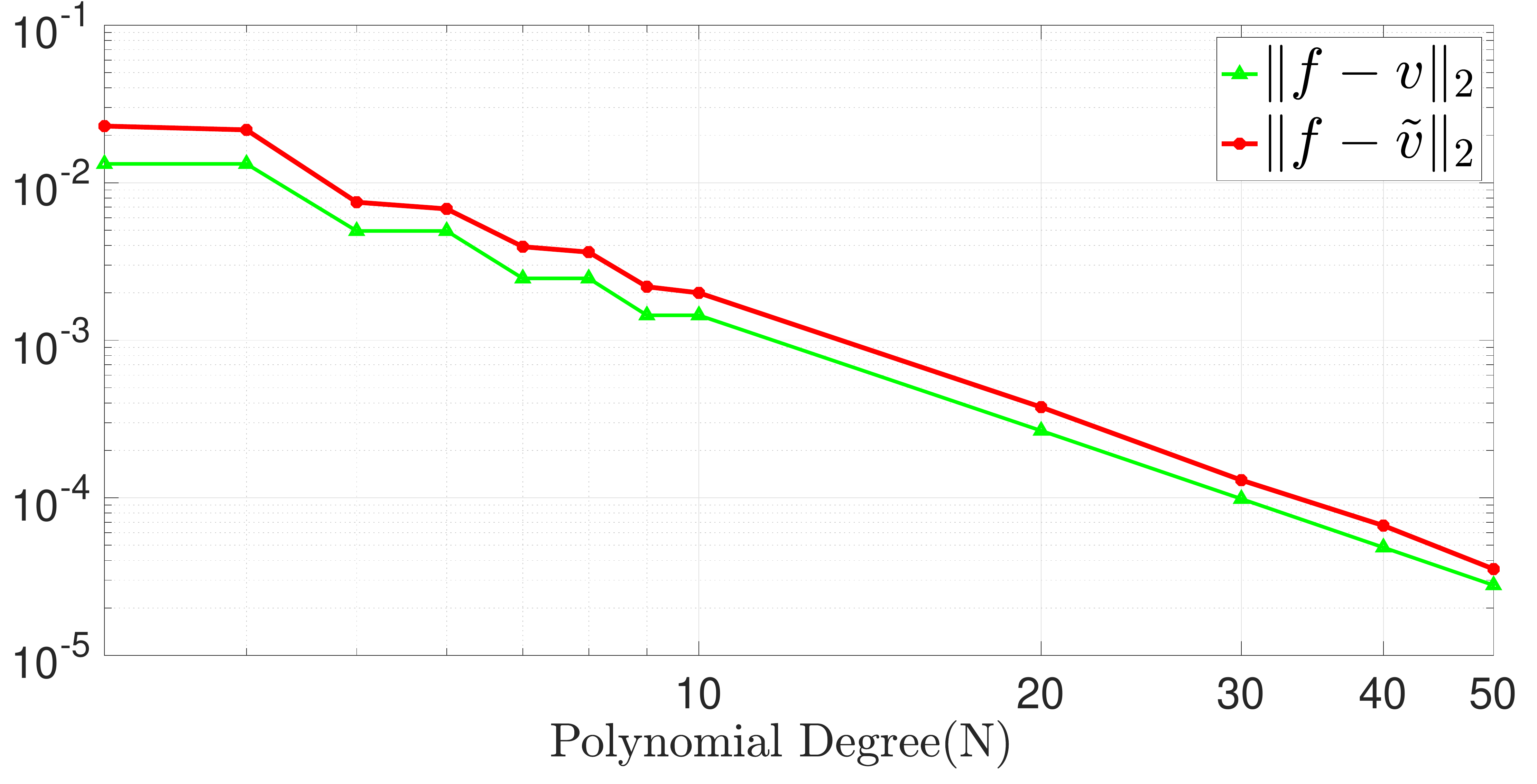}
  \includegraphics[width=0.32\textwidth]{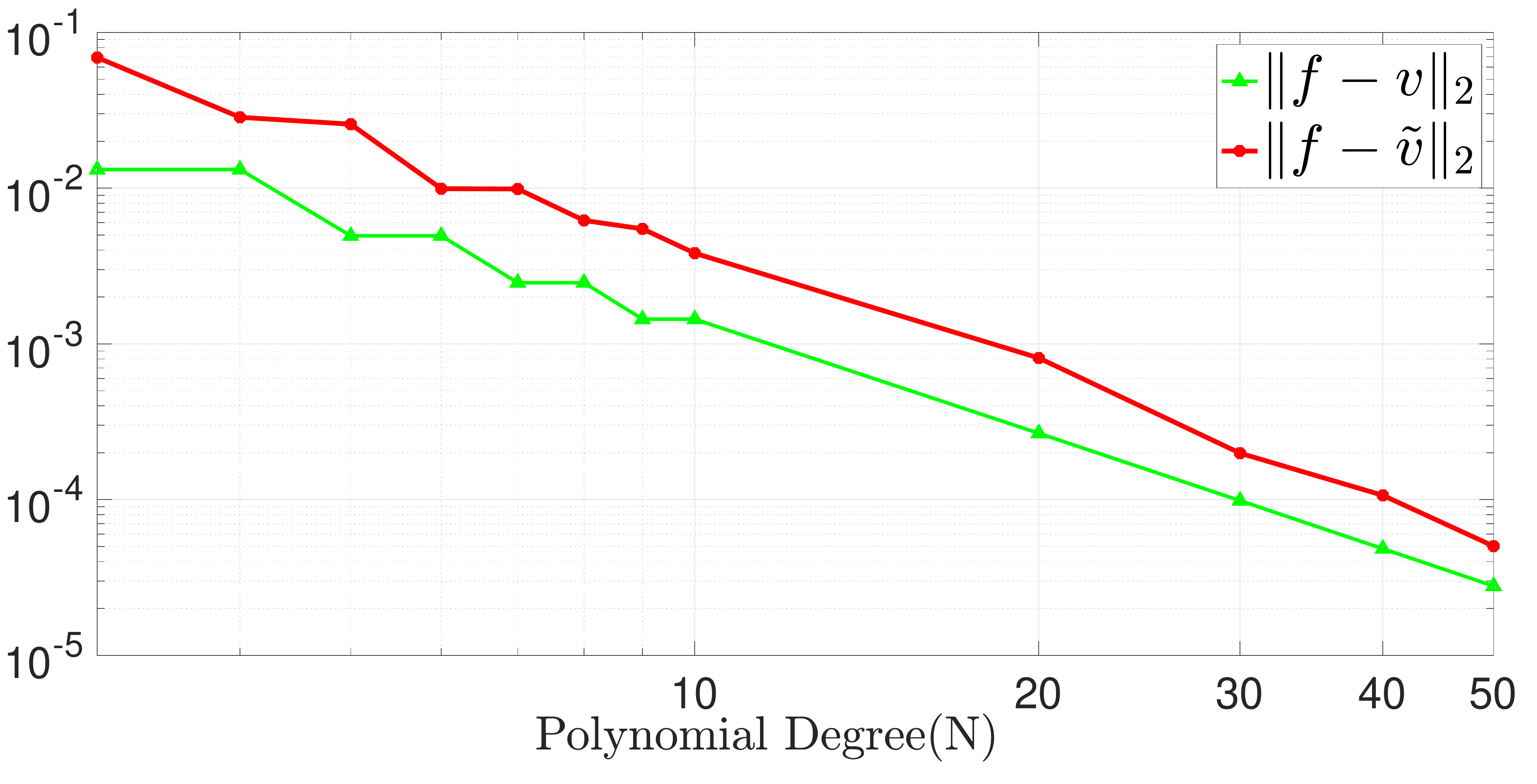}
\end{center}
\caption{$H = H^0$ convergence results for projection of test function $f = f_0$ (top row) and $f = f_2$ (bottom row). $V$ is a space of polynomials of degree $N$. Left: Constraint $E = F_0$. Center: Constraint $E = F_0 \cap G_0$. Right: Constraint $E = F_0 \cap G_0 \cap F_1$.}\label{fig:conv-plot-f}
\end{figure}

\subsection{More complicated constraints}
Finally, we show that our formalism allows for more complicated constraints than the ones we have previously shown. With $H = H^0$ and $V$ a space of degree-$(N-1)$ polynomials as before, we consider two new kinds of constraints:
\begin{itemize}
  \item $J_1 = \left\{ f \in V \; | f(x) \geq |x| \;\; \forall \; x \in [-1,1] \right\}$
  \item $J_2 = \left\{ f \in V \; | -\sign(x) f(x) \geq |x| \;\; \forall \; x \in [-1,1] \right\}$
\end{itemize}
Constraint set $J_1$ can be defined as the intersection of two conic constraints: for $x \in [-1,0]$, we enforce $f(x) \geq -x$. For $x \in [0,1]$ we enforce $f(x) \geq x$. Constraint set $J_2$ enforces $f(x) \geq -x$ for $x \in [-1,0]$ as before, but now enforces $f(x) \leq x$ for $x \in [0,1]$. Note that $J_2$ implicitly enforces $f(0) = 0$, but we do not explicitly require this in our algorithm. Since $x \in V$ when $N \geq 2$, we can handle these constraints with our setup.

We consider the test function $f(x) = |x|$; the optimization successfully terminates and results are shown in Figure \ref{fig:Resexpt2}.


%

\begin{figure}[H]
  \begin{center}
    \includegraphics[width=0.32\textwidth]{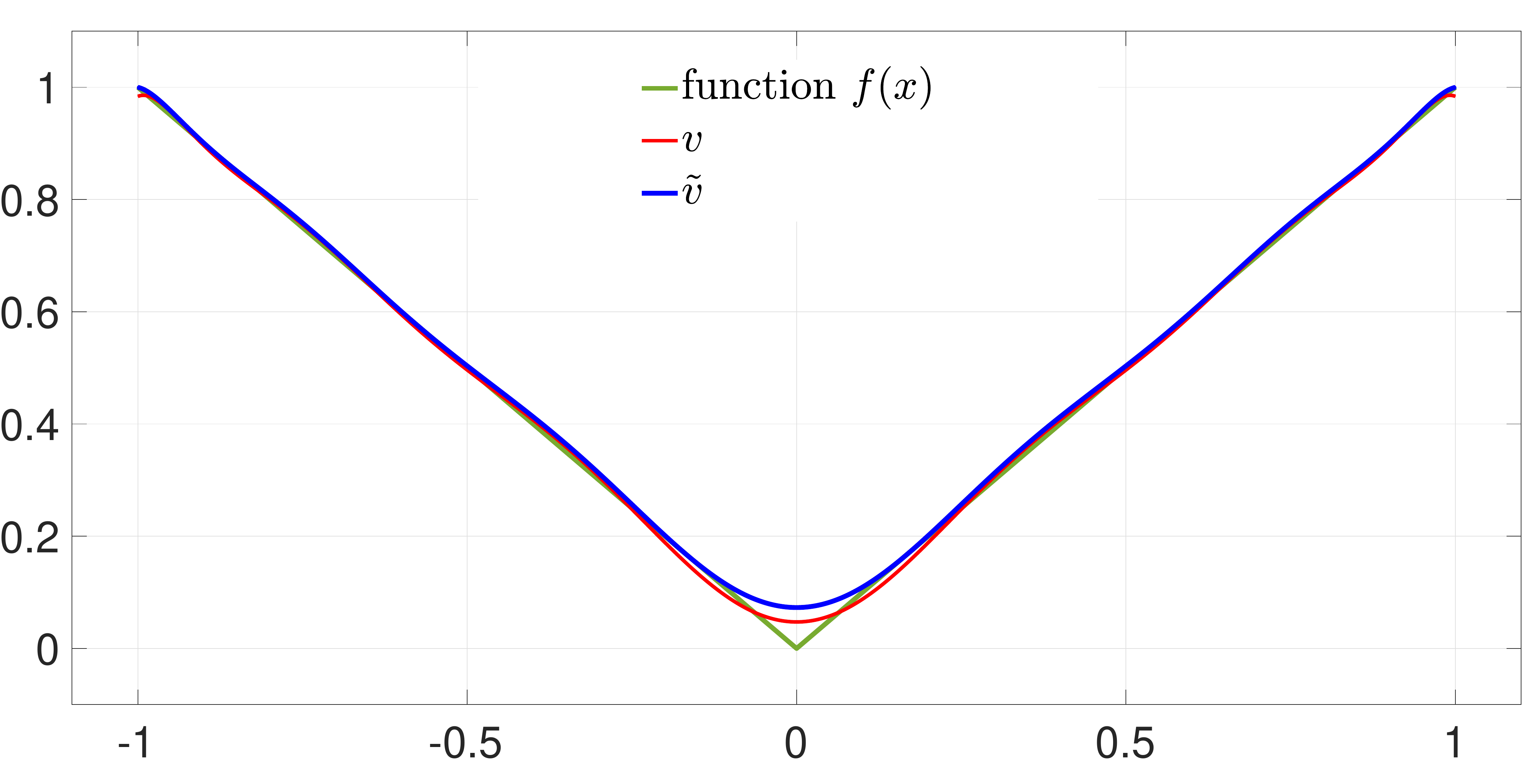}
    \includegraphics[width=0.32\textwidth]{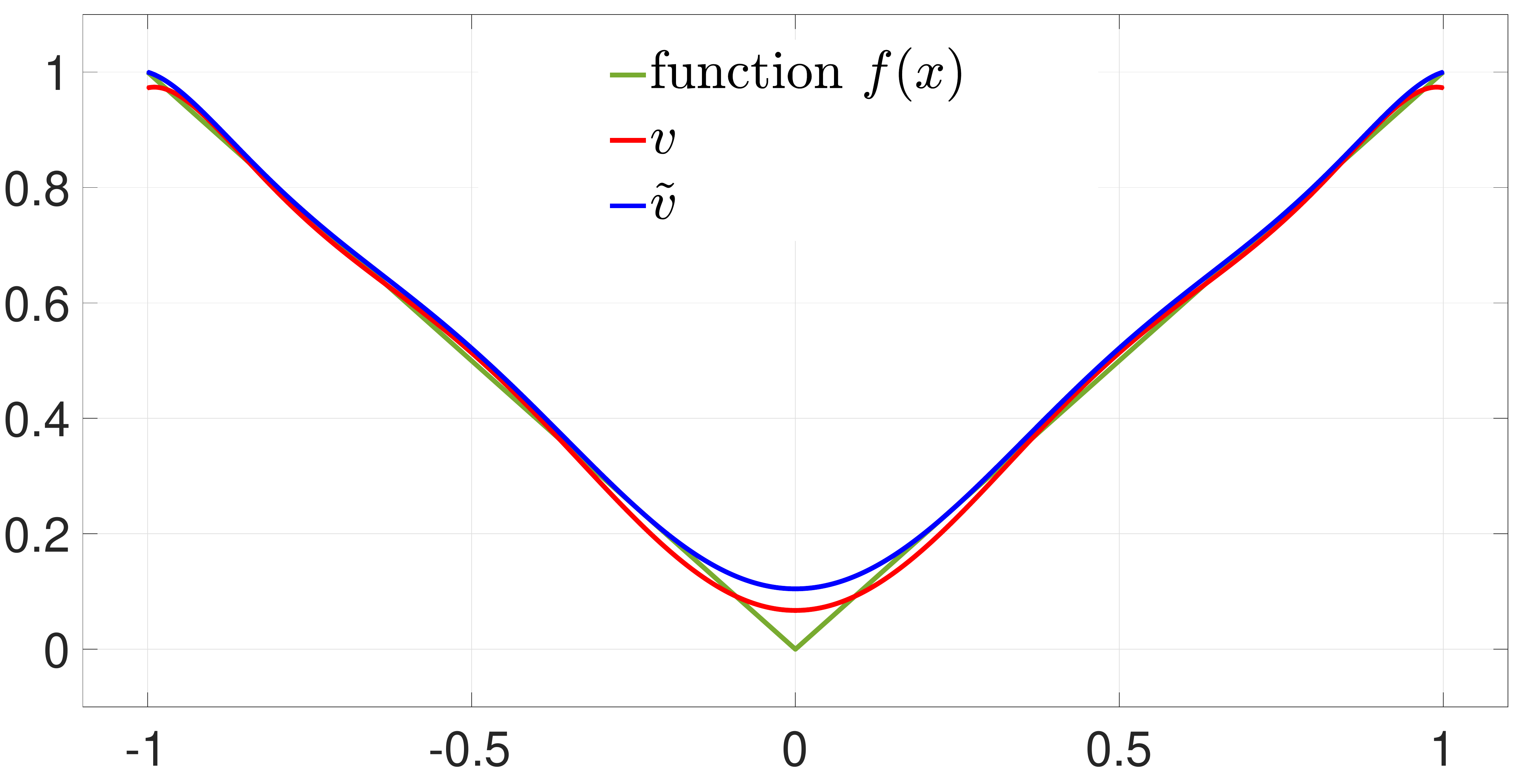}
    \includegraphics[width=0.32\textwidth]{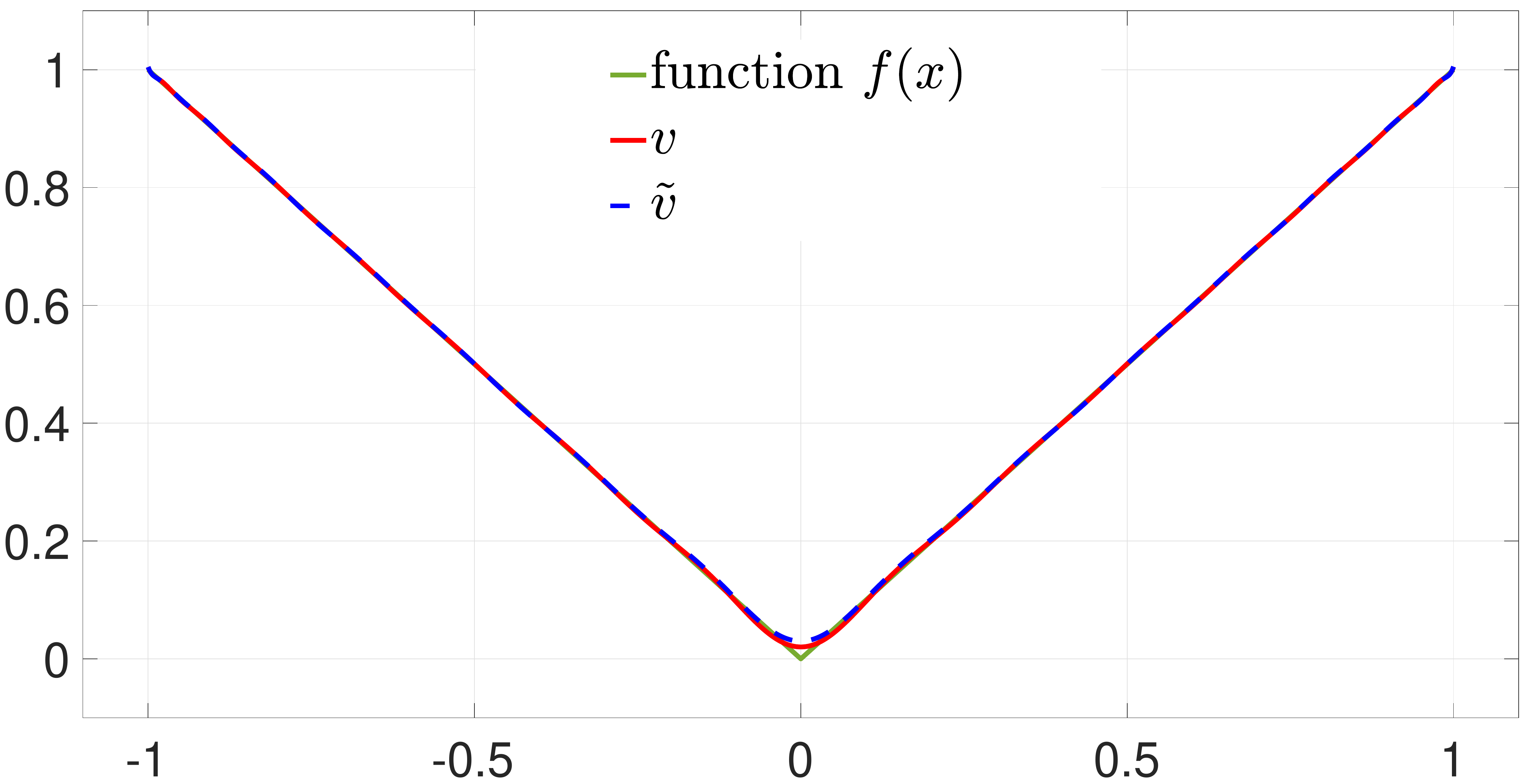} \\
    \includegraphics[width=0.32\textwidth]{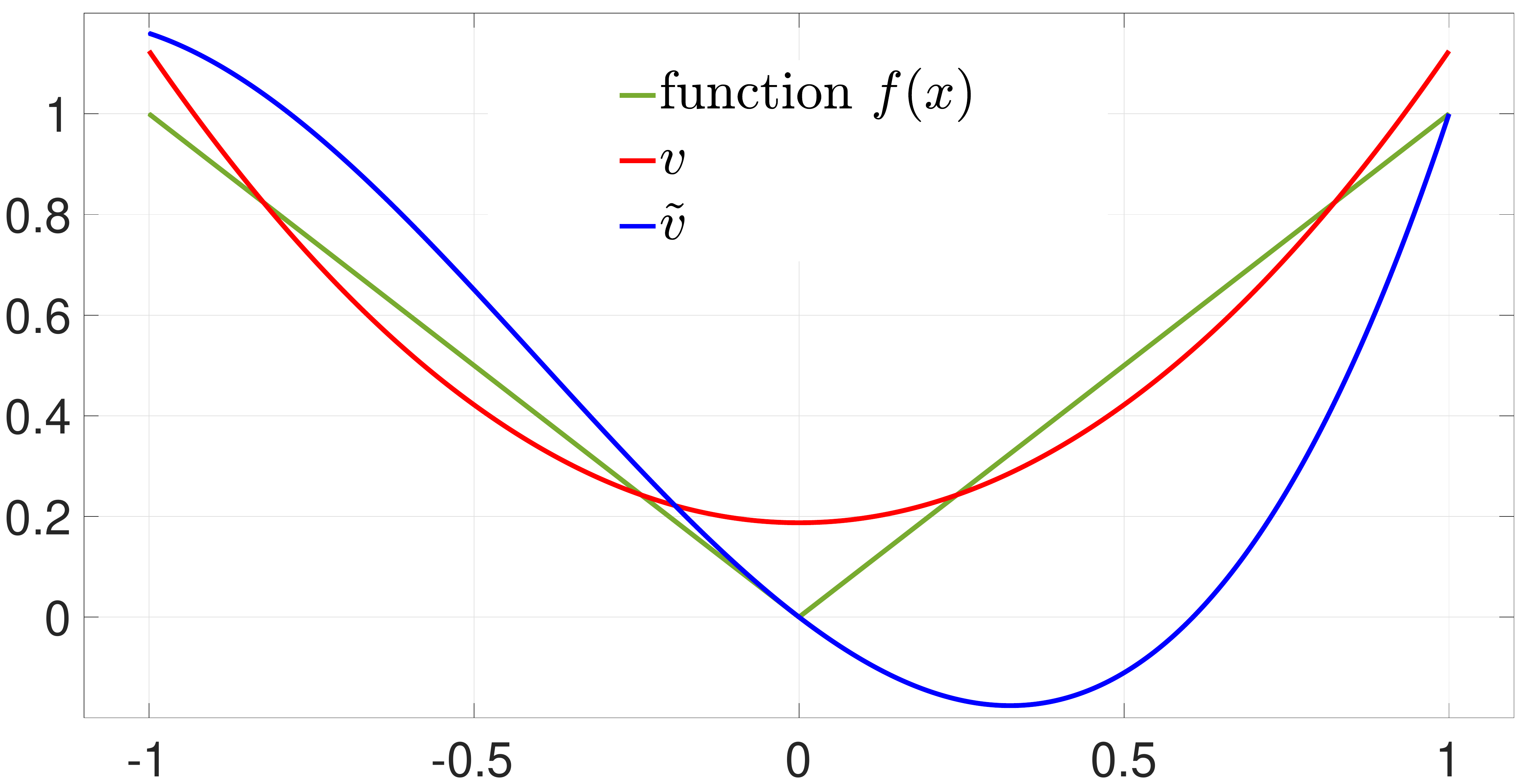}
    \includegraphics[width=0.32\textwidth]{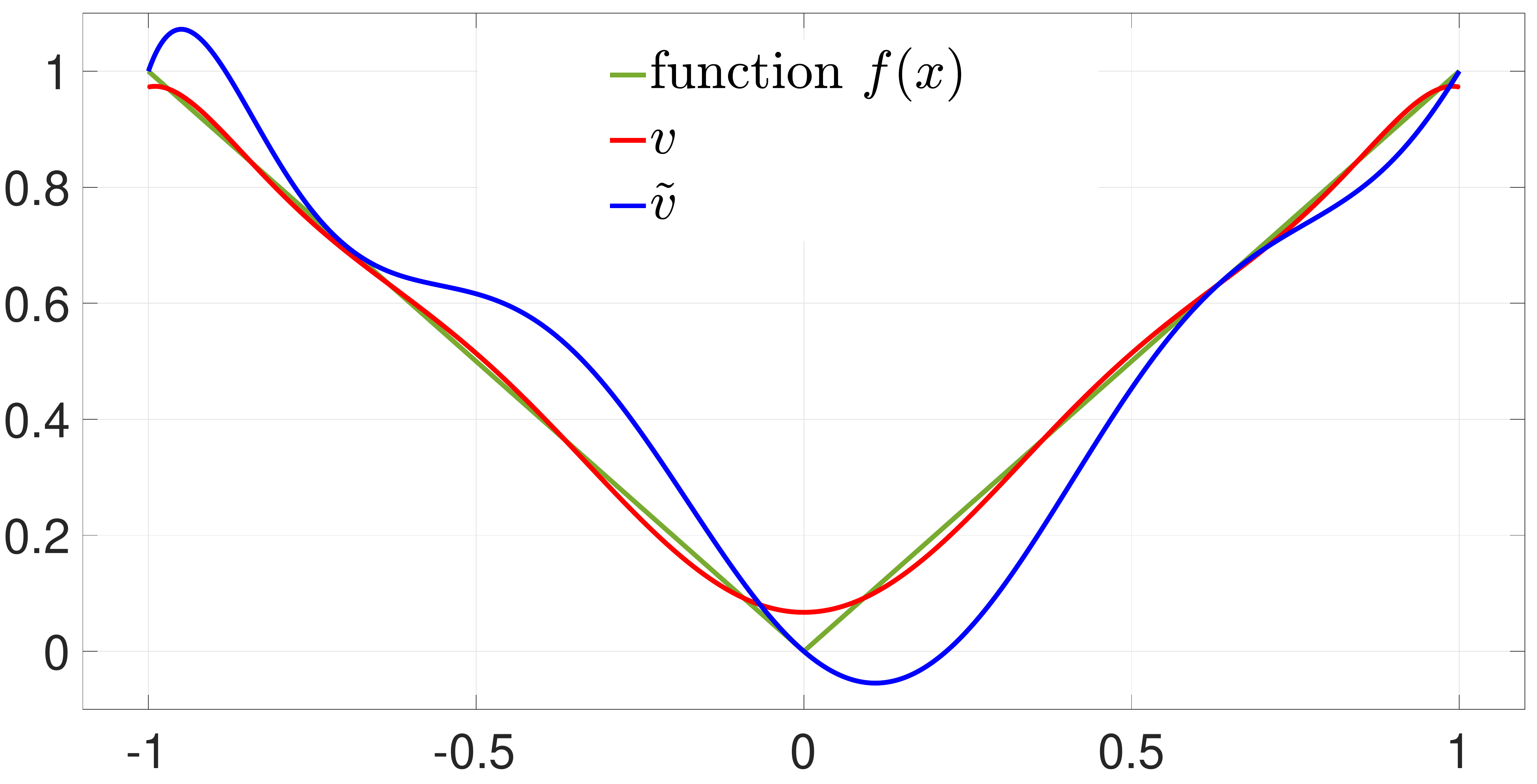} 
    \includegraphics[width=0.32\textwidth]{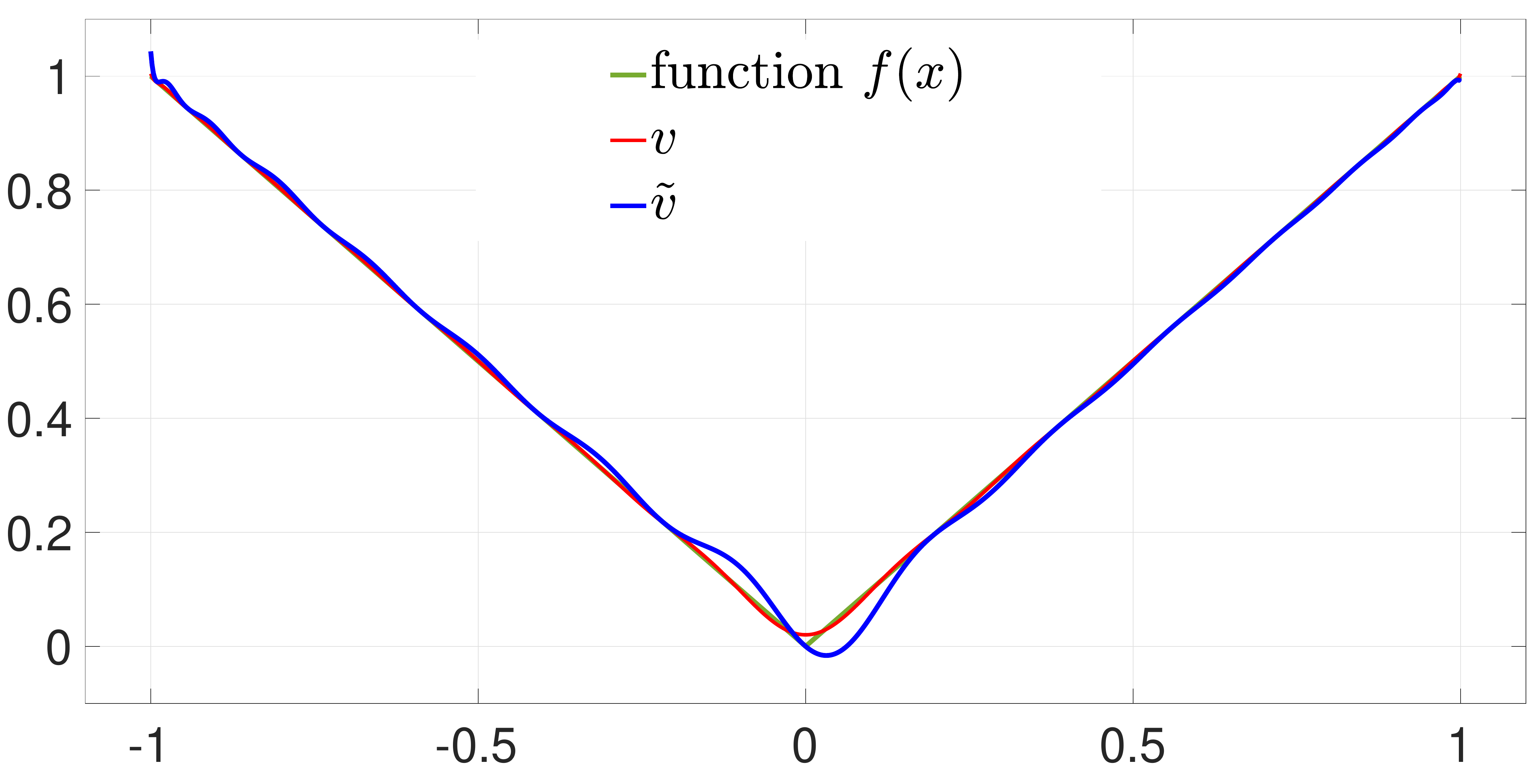} 
  \end{center}
  \caption{Algorithm results from unusual constraints for $f(x) = |x|$. Top: Constraint set $J_1$. Bottom: constraint set $J_2$. Left: $N = 4$. Center: $N=9$. Right: $N=31$.}\label{fig:Resexpt2}
\end{figure}

\section{Conclusions}
\label{sec:conclusions}

We have proposed a formalism for performing constrained function approximation. Restricting the class of possible constraints to those that are convex assures a unique solution to the constrained function approximation problem in Hilbert spaces. Typical constraints of interest such as positivity or monotonicity are specializations of our setup. We propose three iterative algorithms to compute solutions to the problem. Each algorithm requires minimization or level set detection on a weighted version of the current approximant, and thus can be expensive. In one dimension with polynomial approximation, our algorithms require only the ability to accurately compute roots of polynomials.  We have demonstrated the flexibility, feasibility, and utility of our constrained approximation setup with many examples, including empirical investigation of convergence rates.

\an{For higher dimensions, we require the ability to find the minimum of a non-polynomial multivariate function, and so our optimization problem becomes much more complex and expensive. Our difficulties in computing global minima correspond precisely to the known difficulty of globally solving the ``lower-level" problem in semi-infinite programming methods, and our algorithms do not provide novel or constructive approaches to addressing this more general challenge in SIP algorithms. Therefore, identifying approaches to make our algorithm usable for multivariate approximation problems is the subject of ongoing research.} 

\section*{Acknowledgments}
Vidhi Zala and Robert M. Kirby acknowledge support from the National Science Foundation under DMS-1521748 and the Army Research Office under ARO W911NF-15-1-0222 (Program Manager Dr. Mike Coyle). Akil Narayan was partially supported by NSF DMS-1848508. 

\appendix

\section{Algorithms for univariate polynomial subspaces}\label{app:poly-methods}

We present procedures for solving the greedy and averaging optimization procedures in sections \ref{ssec:alg-greedy} and \ref{ssec:alg-averaged} under the assumption that $V$ is a complete, univariate polynomial space. More formally, we make three specializing assumptions. 

The first assumption is that $H$ an $L^2$-type space. A typical setup in one dimension is that $\Omega$ is a interval in (and possibly equal to) $\R$, and a weighted $L^2$ space is defined by a probability density function $\rho$:
\begin{align*}
  \left\langle u, v \right\rangle_{L^2_\rho} \coloneqq \int_\Omega u(x) v(x) \rho(x) \dx{x}
\end{align*}
The second specializing assumption in this section is that $V$ is a complete polynomial space. For a finite $N \in \N$, the space $V$ contains polynomials up to degree $N-1$. Then, $\{v_j\}_{j=1}^N$ can be chosen as the first $N$ orthonormal polynomials under the weight $\rho$ on $\Omega$. It is classical knowledge that such a family of polynomials satisfies the three-term recurrence:
\begin{align*}
  x v_n(x) &= b_{n+1} v_{n+1}(x) + a_{n+1} v_n(x) + b_{n} v_{n-1}(x), & n &\geq 1,
\end{align*}
with the starting conditions $v_0 \equiv 1$ and $v_{-1} \equiv 0$, where $a_n = a_n(\rho)$ and $b_n = b_n(\rho)$ are the recurrence coefficients \cite{szego_orthogonal_1975}.

The third specializing assumption is that we are in the setup of Example \ref{ex:positivity} where the constraints enforce positivity $v(x) \geq 0$ for every $x \in \Omega$. We will see that this assumption can be relaxed substantially; indeed we make this assumption here to only clarify some computations.

An important technique that we will need to exploit for this special setup is the ability to compute roots of polynomials from their expansion coefficients, i.e., if $v \in V$ has expansion coefficients $\{\widehat{v}_j\}_{j=1}^N$, then the $N-1$ (complex-valued) roots of $v$ coincide with the spectrum of the $(N-1) \times (N-1)$ \textit{confederate} matrix $\bs{T} = \bs{T}(v)$:
\begin{align}\label{eq:T-def}
  \bs{T}(v) &= \bs{J} - \frac{b_{N-1}}{\widehat{v}_N} \bs{e}_{N-1} \bs{\widetilde{\widehat{v}}}^T, &   
  \bs{J} &= \left( \begin{array}{ccccc} a_1 & b_1 & & & \\
                                  b_1 & a_2 & b_2 & & \\
                                      & b_2 & a_3 & b_3 & \\
                                      &     & \ddots & \ddots & \\
                                      &     &     & b_{N-2} & a_{N-1}
      \end{array}\right)
\end{align}
where $\bs{e}_{N-1} \in \R^{N-1}$ is the cardinal unit vector in the $(N-1)$st direction and $\bs{\widetilde{\widehat{v}}}^T = (\widehat{v}_1, \ldots, \widehat{v}_{N-1})$. The matrix $\bs{J}$ is the Jacobi matrix and is independent of $v$. We use direct eigenvalue solvers to compute the spectrum of $\bs{T}(v) = v^{-1}(0)$. Note that there are backwards stable versions of the task of computing roots from the spectrum of related matrices \cite{nakatsukasa_stability_2016}. \an{An analogous approach that operates on expansion coefficients in a monomial basis uses the spectrum of the \emph{companion } matrix. Note that our strategy is rather rudimentary compared to more sophisticated methods for computing roots of polynomials \cite{boyd_computing_2003}, e.g., one can compute polynomial roots on subintervals and perform refinement. However, this consideration is not the main innovation of our algorithm, and so we use the procedure above mainly for simplicity. We do perform a numerical stability check where we switch between companion and confederate matrices depending on which has smaller condition number. In all the examples we attempted for this manuscript, this check was sufficient to robustly and accurately compute roots of polynomials.}

\subsection{Greedy projections}
With the setup of Example \ref{ex:positivity}, the problem \eqref{eq:global-minimization} requires us to compute 
\begin{align*}
  y^\ast = \argmin_{y \in \Omega} \mathrm{sdist}\left(\bs{\widehat{v}}, H_1(y)\right) \stackrel{\eqref{eq:sdist-1d}}{=} \argmin_{y \in \Omega} v(y) \lambda(y).
\end{align*}
To minimize the last expression, we can compute the critical points, which are the roots of the derivative. Using \eqref{eq:Riesz-example}, we have
\begin{align*}
  \ddx{y} [ v(y) \lambda(y) ] = \lambda^3(y) \left[ v'(y) \sum_{j=1}^N v_j^2(y) - v(y) \sum_{j=1}^N v_j(y) v_j'(y) \right].
\end{align*}
Note that $\lambda^3$ cannot vanish, so the critical points coincide with the roots of the bracketed expression above, which is a degree-$(3 N - 4)$ polynomial. Thus, 
\begin{align*}
  \frac{\ddx{y} [ v(y) \lambda(y) ]}{\lambda^3(y)} = \sum_{j=1}^{3 N - 3} \widehat{g}_j v_j(y) \eqqcolon g(y),
\end{align*}
for some coefficients $\widehat{g}_j$.
The computation $\left\{\widehat{v}_j \right\} \mapsto \left\{\widehat{g}_j \right\}$ can be accomplished using \textit{only} the recurrence coefficients in $\mathcal{O}(N^2)$ time without resorting to, e.g., quadrature. 

In summary, the global minimum in \eqref{eq:global-minimization} can be computed by first computing the $\widehat{g}_j$ expansion coefficients defined above, and then by computing the spectrum of the $(3N - 4) \times (3N-4)$ matrix $\bs{T}(g)$. To compute the global minimizer, we then need only evaluate the discrete minimum of $v(y) \lambda(y)$ over the eigenvalues located in $\Omega$.

\subsection{Averaged projections}
The main task for the averaged projections procedure is to compute the integral in \eqref{eq:c-update-averaged}. In our specialized setup, this task reduces to computing
\begin{align*}
  \frac{1}{|\omega_1^-|}\int_{\omega_1^-}  \bs{\widehat{\ell}}_1(y) v(y) \lambda(y) \dx{y},
\end{align*}
which is an $N$-component vector, where component $j$ of this vector has the entry
\begin{align}\label{eq:averaging-component}
  \frac{1}{|\omega_1^-|}\int_{\omega_1^-}  v_j(y) v(y) \lambda(y) \dx{y}.
\end{align}
The first step is to identify the set $\omega_1^-$ defined in \eqref{eq:averaging-set}, which in this special case is equivalent to
\begin{align*}
  \omega_1^- &= \left\{ y \in [-1,1] \; \big|\; v(y) < 0 \right\}.
\end{align*}
Therefore, this set can be identified by examining the roots of $v$, which are the eigenvalues of $\bs{T}(v)$. Thus, we partition $[-1,1]$ into subintervals on which $v$ is single-signed, after which determining the sign of $v$ on an interval can be accomplished by evaluating $v$ in this interval.

After $\omega_1^-$ is identified as a disjoint collection of subintervals of $[-1,1]$, we compute the components of the update \eqref{eq:averaging-component} by employing an $M$-point Gaussian quadrature rule; since the integrand $v_j v \lambda$ is a smooth function on $[-1,1]$, this can be completed efficiently. We employ $M = N+1$ quadrature points for this same computation.

\bibliographystyle{siamplain}
\bibliography{references}

\begin{thebibliography}{10}

\bibitem{anton2013elementary}
{\sc H.~Anton and C.~Rorres}, {\em Elementary Linear Algebra, Binder Ready
  Version: Applications Version}, John Wiley \& Sons, 2013.

\bibitem{bauschke_projection_1996}
{\sc H.~Bauschke and J.~Borwein}, {\em On {Projection} {Algorithms} for
  {Solving} {Convex} {Feasibility} {Problems}}, SIAM Review, 38 (1996),
  pp.~367--426, \url{https://doi.org/10.1137/S0036144593251710},
  \url{http://epubs.siam.org/doi/10.1137/S0036144593251710} (accessed
  2018-02-15).

\bibitem{beatson_restricted_1982}
{\sc R.~Beatson}, {\em Restricted {Range} {Approximation} by {Splines} and
  {Variational} {Inequalities}}, SIAM Journal on Numerical Analysis, 19 (1982),
  pp.~372--380, \url{https://doi.org/10.1137/0719023}.

\bibitem{beatson_degree_1978}
{\sc R.~K. Beatson}, {\em The degree of monotone approximation.}, Pacific
  Journal of Mathematics, 74 (1978), pp.~5--14,
  \url{https://projecteuclid.org/euclid.pjm/1102810431} (accessed 2018-10-30).

\bibitem{berzins_adaptive_2007}
{\sc M.~Berzins}, {\em Adaptive {Polynomial} {Interpolation} on {Evenly}
  {Spaced} {Meshes}}, SIAM Review, 49 (2007), pp.~604--627,
  \url{https://doi.org/10.1137/050625667}.

\bibitem{boyd_computing_2003}
{\sc J.~P. Boyd}, {\em Computing {Zeros} on a {Real} {Interval} through
  {Chebyshev} {Expansion} and {Polynomial} {Rootfinding}}, SIAM Journal on
  Numerical Analysis, 40 (2003), pp.~1666--1682.

\bibitem{boyd_convex_2004}
{\sc S.~Boyd and L.~Vandenberghe}, {\em Convex {Optimization}, {With}
  {Corrections} 2008}, Cambridge University Press, Cambridge, UK ; New York, 1
  edition~ed., Mar. 2004.

\bibitem{boyd_introduction_2018}
{\sc S.~Boyd and L.~Vandenberghe}, {\em Introduction to {Applied} {Linear}
  {Algebra}: {Vectors}, {Matrices}, and {Least} {Squares}}, Cambridge
  University Press, Cambridge, UK ; New York, NY, 1 edition~ed., Aug. 2018.

\bibitem{bregman_method_1965}
{\sc L.~Bregman}, {\em The method of successive projection for finding a common
  point of convex sets}, Soviet Math Dokl., 6 (1965), pp.~688--692.

\bibitem{campos-pinto_algorithms_2019}
{\sc M.~Campos-Pinto, F.~Charles, and B.~Després}, {\em Algorithms {For}
  {Positive} {Polynomial} {Approximation}}, SIAM Journal on Numerical Analysis,
  57 (2019), pp.~148--172, \url{https://doi.org/10.1137/17M1131891},
  \url{https://epubs.siam.org/doi/abs/10.1137/17M1131891} (accessed
  2019-11-06).

\bibitem{cheney_proximity_1959}
{\sc W.~Cheney and A.~A. Goldstein}, {\em Proximity {Maps} for {Convex}
  {Sets}}, Proceedings of the American Mathematical Society, 10 (1959),
  pp.~448--450, \url{https://doi.org/10.2307/2032864},
  \url{https://www.jstor.org/stable/2032864} (accessed 2019-11-04).

\bibitem{deutsch_rate_2006}
{\sc F.~Deutsch and H.~Hundal}, {\em The rate of convergence for the cyclic
  projections algorithm {I}: {Angles} between convex sets}, Journal of
  Approximation Theory, 142 (2006), pp.~36--55,
  \url{https://doi.org/10.1016/j.jat.2006.02.005}.

\bibitem{deutsch_best_2012}
{\sc F.~R. Deutsch}, {\em Best {Approximation} in {Inner} {Product} {Spaces}},
  Springer Science \& Business Media, Dec. 2012.

\bibitem{devore_degree_1974}
{\sc R.~A. DeVore}, {\em Degree of {Monotone} {Approximation}}, International
  {Series} of {Numerical} {Mathematics} / {Internationale} {Schriftenreihe} zur
  {Numerischen} {Mathematik} / {Série} {Internationale} {D}’{Analyse}
  {Numérique}, Birkhäuser Basel, Basel, 1974,
  \url{https://doi.org/10.1007/978-3-0348-5991-2_26},
  \url{https://doi.org/10.1007/978-3-0348-5991-2_26} (accessed 2019-10-17).

\bibitem{goberna_semi-infinite_2001}
{\em Semi-{Infinite} {Programming}: {Recent} {Advances}}, Nonconvex
  {Optimization} and {Its} {Applications}, Springer US, 2001,
  \url{https://doi.org/10.1007/978-1-4757-3403-4}.

\bibitem{gubin_method_1967}
{\sc L.~G. Gubin, B.~T. Polyak, and E.~V. Raik}, {\em The method of projections
  for finding the common point of convex sets}, USSR Computational Mathematics
  and Mathematical Physics, 7 (1967), pp.~1--24,
  \url{https://doi.org/10.1016/0041-5553(67)90113-9}.

\bibitem{hettich_semi-infinite_1993}
{\sc R.~Hettich and K.~O. Kortanek}, {\em Semi-{Infinite} {Programming}:
  {Theory}, {Methods}, and {Applications}}, SIAM Review, 35 (1993),
  pp.~380--429, \url{https://doi.org/10.1137/1035089},
  \url{http://epubs.siam.org/doi/abs/10.1137/1035089} (accessed 2020-07-01).
\newblock Publisher: Society for Industrial and Applied Mathematics.

\bibitem{lewis_local_2009}
{\sc A.~S. Lewis, D.~R. Luke, and J.~Malick}, {\em Local {Linear} {Convergence}
  for {Alternating} and {Averaged} {Nonconvex} {Projections}}, Foundations of
  Computational Mathematics, 9 (2009), pp.~485--513,
  \url{https://doi.org/10.1007/s10208-008-9036-y},
  \url{https://doi.org/10.1007/s10208-008-9036-y} (accessed 2019-09-28).

\bibitem{lewis_approximation_1973}
{\sc J.~Lewis}, {\em Approximation with {Convex} {Constraints}}, SIAM Review,
  15 (1973), pp.~193--217, \url{https://doi.org/10.1137/1015006},
  \url{https://epubs.siam.org/doi/abs/10.1137/1015006}.

\bibitem{nakatsukasa_stability_2016}
{\sc Y.~Nakatsukasa and V.~Noferini}, {\em On the stability of computing
  polynomial roots via confederate linearizations}, Mathematics of Computation,
  85 (2016), pp.~2391--2425, \url{https://doi.org/10.1090/mcom3049},
  \url{https://www.ams.org/home/page/} (accessed 2018-08-28).

\bibitem{nochetto_positivity_2002}
{\sc R.~Nochetto and L.~Wahlbin}, {\em Positivity preserving finite element
  approximation}, Mathematics of Computation, 71 (2002), pp.~1405--1419,
  \url{https://doi.org/10.1090/S0025-5718-01-01369-2},
  \url{https://www.ams.org/mcom/2002-71-240/S0025-5718-01-01369-2/} (accessed
  2019-11-06).

\bibitem{rice_approximation_1963}
{\sc J.~Rice}, {\em Approximation with {Convex} {Constraints}}, Journal of the
  Society for Industrial and Applied Mathematics, 11 (1963), pp.~15--32,
  \url{https://doi.org/10.1137/0111002},
  \url{http://epubs.siam.org/doi/abs/10.1137/0111002}.

\bibitem{stein_how_2012}
{\sc O.~Stein}, {\em How to solve a semi-infinite optimization problem},
  European Journal of Operational Research, 223 (2012), pp.~312--320,
  \url{https://doi.org/10.1016/j.ejor.2012.06.009}.

\bibitem{szego_orthogonal_1975}
{\sc G.~Szeg\"{o}"}, {\em Orthogonal {Polynomials}}, American Mathematical
  Soc., 4th~ed., 1975.

\bibitem{von_neumann_functional_1951}
{\sc J.~Von~Neumann}, {\em Functional {Operators} ({AM}-22), {Volume} 2}, 1951,
  \url{https://press.princeton.edu/titles/3136.html} (accessed 2018-11-13).

\bibitem{zhang_maximum-principle-satisfying_2011}
{\sc X.~Zhang and C.-W. Shu}, {\em Maximum-principle-satisfying and
  positivity-preserving high-order schemes for conservation laws: survey and
  new developments}, Proceedings of the Royal Society of London A:
  Mathematical, Physical and Engineering Sciences,  (2011), p.~rspa20110153,
  \url{https://doi.org/10.1098/rspa.2011.0153}.

\end{thebibliography}
\end{document}